\documentclass[11pt]{article}
\usepackage{fullpage}
\usepackage{preamble}
\usepackage{import}
\usepackage{amsmath}
\usepackage{bbm}
\usepackage[utf8]{inputenc}
\usepackage{mathtools, bm}
\usepackage{amssymb, bm}
\usepackage[
backend=biber,
style=alphabetic,
sorting=nyt,
maxnames=4,
maxalphanames=4
]{biblatex}
\title{Whittaker--Shintani functions for Fourier--Jacobi models on unitary groups}
\author{Paul Boisseau}
\date{\today}

\begin{document}
\maketitle

\begin{abstract}We state and prove a formula for the Whittaker--Shintani functions associated to Fourier--Jacobi models for $p$-adic unitary groups and general linear groups. These generalized spherical functions play a fundamental role in the proof of the Gan--Gross--Prasad conjecture for Fourier--Jacobi models. As an application we prove the unramified Ichino--Ikeda conjecture.
\end{abstract}

\setcounter{tocdepth}{1}
\tableofcontents

\section{Introduction}

\subsection{The Gan--Gross--Prasad conjecture for Fourier--Jacobi periods}
Let $V$ be a non-degenerate $n$-dimensional skew-Hermitian space relative to a quadratic extension of number fields $k'/k$, and let $W \subset V$ be a non-degenerate subspace of dimension $m$ such that the corank $n-m=:2r$ is even. Set $G:=\mathrm{U}(V)\times \mathrm{U}(W)$. Assume that the orthogonal complement $W^\perp$ is split, which means that it contains an isotropic subspace $X$ of dimension $r$. Gan Gross and Prasad have introduced in~\cite{GGP} a subgroup $H \leq G$ and an automorphic representation $\overline{\nu}_{\mu,\psi}=\otimes \overline{\nu}_{\mu,\psi,v}$ of $H$ built from the Weil representation of $\mathrm{U}(W)$. It depends on a character $\mu$ of the idele class group $(k')^\times \backslash \mathbb{A}_{k'}^\times$ lifting the quadratic character $\eta_{k'/k}$ attached to $k'/k$ by global class field theory, and on a non trivial character $\psi$ of $k \backslash \ad_k$. If $\sigma=\otimes \sigma_v$ is a cuspidal automorphic representation of $G(\ad_k)$, one can define a global Fourier--Jacobi period 
\begin{equation*}    \mathcal{P}_H(\varphi,\phi):=\int_{H(k)\backslash H(\ad_k)} \varphi(h) \theta(h,\phi)dh
\end{equation*}
for $\varphi \in \sigma$, $\phi \in \overline{\nu}_{\mu,\psi}$, where $\theta$ is a theta series. Assume that $\Hom_{H(k_v)}(\sigma_v \otimes \overline{\nu}_{\mu,\psi,v},\cc) \neq \{0\}$ for any place $v$ of $k$. The Gan--Gross--Prasad conjecture for Fourier--Jacobi models states that $\mathcal{P}_H$ is non-zero if and only if the central value of the complete Rankin--Selberg $L$-function $L(\frac{1}{2},\sigma \otimes \overline{\mu})$ is non-zero. Progress has been obtained by Xue in \cite{Xue1} and \cite{Xue} in the corank $0$ case under local assumptions. With Lu and Xue, we proved in \cite{BLX} the conjecture for any even corank and every cuspidal representation of $G$ whose base change is generic.

By uniqueness of local Fourier--Jacobi functionals, $\mathcal{P}_H$ decomposes as an Eulerian product of linear forms $\mathcal{L}_{H_v} \in \Hom_{H_v}(\sigma_v \otimes \overline{\nu}_{\mu,\psi,v},\cc)$. In \cite{BPC}, Beuzart-Plessis and Chaudouard used an unfolding argument to reduce the conjecture in the Bessel case (i.e. when $n-m$ is odd) to the corank $1$ situation. Moreover, if $\sigma$ is assumed to be everywhere tempered, integrating matrix coefficients yields a $H_v$-invariant sesquilinear form $\mathcal{P}_{H_v}$ on $\sigma_v \otimes \overline{\nu}_{\mu,\psi,v}$. In this setting, Xue formulated in \cite{Xue} a variation of the Ichino--Ikeda conjecture \cite{II} relating $\Val{\mathcal{P}_H}^2$ to $\prod_v \mathcal{P}_{H_v}$. For both problems, it is necessary to first check that the corresponding local statement holds in the unramified setting in order to regularize the Eulerian product. The goal of this article is to understand the unramified behaviour of the local Fourier--Jacobi functionals $\mathcal{L}_{H_v}$. This includes two cases: if the place $v$ is inert, then $\mathrm{U}(V)_v$ and $\mathrm{U}(W)_v$ are quasi-split unitary groups, if it is split they are isomorphic to $\GL_n(k_v)$ and $\GL_m(k_v)$. 

\subsection{Formulae for Whittaker--Shintani functions}

We now describe our main results. We work in the local unramified situation and we drop the place $v$ from the notation, so that $E/F$ is an unramified quadratic extension of $p$-adic fields, or $E=F\times F$. Denote by $q_F$ and $q_E$ the cardinals of the respective residual fields (if $E=F\times F$, take $q_E=q_F$). Assume that $2$ does not divide $q_F$, that $V$ and $W$ are split in the sense of \S\ref{subsubsec:V_W} and that $\mu$ and $\psi$ are unramified. We identify any algebraic group over $F$ with its group of $F$-points. Let $(T,B^+)$ be a Borel pair in $G$ such that $HB^+$ is open in $G$ and $T$ contains $N_G(H) \cap B^+$. Write $T=T_V \times T_W$. Let $K=K_V \times K_W$ be a hyperspecial open-compact subgroup of $G$ in good position relatively to $T$. Let $\chi$ and $\eta$ be unramified characters of $T_V$ and $T_W$. Assume that $\sigma$ is the unramified irreducible subquotient of the parabolic induction $I_{B^+}^{G} (\chi \boxtimes \eta)$. Let $\varphi^\circ \in \sigma^K$ be a non-zero $K$-spherical vector. The representation $\overline{\nu}_{\mu,\psi}$ of $H$ extends to a representation of a group $J$, named the Jacobi group (see \eqref{eq:Jacobi_defi} for the definition). There is a maximal compact subgroup $K_J$ in $J$, and $\overline{\nu}_{\mu,\psi}^{K_J}$ has dimension $1$. Let $\phi^\circ \in \overline{\nu}_{\mu,\psi}^{K_J}$ be non-zero. Consider

\begin{equation}
\label{eq:W_chi}
    \mathcal{W}_{\chi,\overline{\mu} \eta}(g):=\mathcal{L}_H(\sigma(g)\varphi^\circ \otimes \phi^\circ), \quad g \in \mathrm{U}(V).
\end{equation}
This $\mathcal{W}_{\chi,\overline{\mu}\eta}$ belongs to the space $\mathrm{WS}_{\chi,\overline{\mu} \eta}$ of \emph{Whittaker--Shintani functions} associated to $(\chi,\overline{\mu} \eta)$. It is roughly defined as follows. In \S\ref{subsec:WS_functions} we introduce a space $\mathrm{WS}$ of functions $\mathcal{W} \in C^\infty(\mathrm{U}(V))$ which satisfy some bi-invariance properties. It is a module on the Hecke algebra $\mathcal{H}(\mathrm{U}(V),K_V) \otimes \mathcal{H}(J,K_J)$ by left and right convolutions. $\mathrm{WS}_{\chi,\overline{\mu}\eta}$ is then the eigenspace of functions that transform by the Satake character of $\sigma\otimes \overline{\nu}_{\mu,\psi}$. Our first result is multiplicity one for such functions. It is proved in \S\ref{subsection:proof_mult1}.
\begin{theorem}
\label{thm:mult1intro}
    The space $\mathrm{WS}_{\chi,\overline{\mu} \eta}$ has dimension one.
\end{theorem}

The space $\mathrm{WS}$ admits another more natural description. Let $\nu_{\mu,\psi}$ be the complex conjugate of $\overline{\nu}_{\mu,\psi}$ and consider $(\Ind_H^G  (\nu_{\mu,\psi}) )^K$ the space of spherical vectors in the smooth induction. This is a $\mathcal{H}(G,K)$-module. There exists a surjective map of $\cc$-algebras $\mathcal{H}(\mathrm{U}(V),K_V) \otimes \mathcal{H}(J,K_J) \to \mathcal{H}(G,K)$ and we have an isomorphism of $\mathcal{H}(\mathrm{U}(V),K_V) \otimes \mathcal{H}(J,K_J)$-modules 
\begin{equation}
\label{eq:intro_WS_iso}
    \mathrm{WS} \simeq \left( \Ind_H^G  (\nu_{\mu,\psi}) \right)^K.
\end{equation}
Note that a modification is needed if $r=0$ (see \S\ref{subsec:r=0}). Let $\Lambda \subset T$ be the group of cocharacters of $T$, and let $\Lambda^-$ be the subset of negative cocharacters relatively to $B^+$. By duality, Theorem~\ref{thm:mult1intro} is a consequence of the following result which is of independent interest. It is a combination of Proposition~\ref{prop:c_compact_decompo} and Proposition~\ref{prop:rank1}.
\begin{prop}
\label{prop:module_intro}
    The spherical space in the compact induction $(\ind_H^G \nu_{\mu,\psi})^K$ decomposes as
      \begin{equation*}
       \left( \ind_H^G (\nu_{\mu,\psi}) \right)^K= \bigoplus_{\lambda \in \Lambda^-} \cc \; \Phi_{\lambda},
    \end{equation*}
     where $\Phi_{\lambda}$ is the unique vector in $(\ind_H^G (\nu_{\mu,\psi}))^K$ such that $\supp(\Phi_\lambda) \subset H \lambda K$ and $\Phi_\lambda(\lambda)=\phi^\circ$ (where $\Phi_\lambda$ is seen as a function on $G$ valued in $\nu_{\mu,\psi}$). Moreover, $(\ind_H^G (\nu_{\mu,\psi}))^K$ is a free $\mathcal{H}(G,K)$-module of rank $1$.
\end{prop}

The image in $(\Ind_H^G  (\nu_{\mu,\psi}))^K$ of the function $\mathcal{W}_{\chi,\overline{\mu}\eta}$ defined in \eqref{eq:W_chi} is easy to describe. By choosing an invariant inner product $(\cdot,\cdot)_{\nu}$, we have an identification $\overline{\nu}_{\mu,\psi}^\vee \simeq \nu_{\mu,\psi}$. By duality, the operator $\mathcal{L}_H$ corresponds to $\mathcal{L}_H^\vee \in \Hom_H(\sigma,\nu_{\mu,\psi})$. Consider the map
\begin{equation}
    \Phi_{\chi,\overline{\mu}\eta}(g)=\mathcal{L}_H^\vee(\sigma(g) \varphi^\circ) \in \nu_{\mu,\psi}, \quad g \in G.
\end{equation}
Then $\Phi_{\chi,\overline{\mu}\eta} \in (\Ind_H^G  (\nu_{\mu,\psi}))^K$ is the image of $\mathcal{W}_{\chi,\overline{\mu}\eta}$ by \eqref{eq:intro_WS_iso}.

Our second main result is a formula for $\Phi_{\chi,\overline{\mu}\eta}$. Set $n_-=\lfloor n/2 \rfloor$, $m_-=\lfloor m/2 \rfloor$ and $r'=m+r$. Let $\delta_{B^+}$ be the modular character of $B^+$. Denote by $\Sigma_{\mathrm{nd}}^+$ the set of non-divisible positive roots of $T$ in $B^+$. Then $T$ is identified in a natural way with the canonical maximal torus $T_G$ of $G$ which is $(\mathrm{Res}_{E/F} \LAG_{m})^{n_-+m_-}$ in the inert even case, $(\mathrm{Res}_{E/F} \LAG_{m})^{n_-+m_-} \times \mathrm{U}(1)^2$ in the inert odd case, and $\LAG_m^{n+m}$ in the split case. We identify the character $\chi \boxtimes \eta$ with a couple $(\chi,\eta) \in (\cc^{\times})^{n_-+m_-}$ or $(\cc^{\times})^{n+m}$. Let $W_G$ be the Weyl group of $G$ and let $w_0$ be its longest element. Define $\zeta_F(s)=(1-q_F^{-s})^{-1}$ and $\zeta_E(s)=(1-q_E^{-s})^{-1}$ the local zeta functions, and for $x,s \in \cc$ set $L_F(s,x)=(1-xq_F^{-s})^{-1}$ and $L_E(s,x)=(1-xq_E^{-s})^{-1}$. The RHS of \eqref{eq:formula_intro} is well-defined for $(\chi,\eta)$ in general position and extends to a regular function. The constants involved depend on whether $E/F$ is inert or split, and $n$ is even or odd. The theorem is proved in two steps in \S\ref{subsec:proofs} and \S\ref{section:proof_normalization}.

\begin{theorem}
\label{thm:W_formula}We have $(\Phi_{\chi,\overline{\mu}\eta}(1),\phi^\circ)_\nu=0$ if and only if $\Phi_{\chi,\overline{\mu}\eta}=0$. Moreover if $\Phi_{\chi,\overline{\mu}\eta}\neq0$ then for all $\lambda \in \Lambda^{-}$ we have
    \begin{equation}
    \label{eq:formula_intro}
        \frac{( \Phi_{\chi,\overline{\mu} \eta}(\lambda),\phi^\circ )_{\nu}}{(\Phi_{\chi,\overline{\mu}\eta}(1),\phi^\circ)_\nu} = \frac{\Delta_{\mathrm{U}(W)}}{\Delta_{T_W}} \sum_{w \in W_G} \mathbf{b}(w_V \chi, w_W \overline{\mu}\eta) \mathbf{d}(w(\chi \boxtimes \eta))\left( w (\chi \boxtimes \eta) \delta_{B^+}^{-\frac{1}{2}} \right)(\lambda),
    \end{equation}
    where $w=(w_V,w_W)$, 
    \begin{equation*}
        \Delta_{T_W}=\left\{
        \begin{array}{ll}
            \zeta_E(1)^{m_-} & \text{inert even case,} \\
            \zeta_E(1)^{m_-}L_F(1,\eta_{E/F}) & \text{inert odd case,} \\
            \zeta_F(1)^{m} & \text{split case,} \\
        \end{array}
        \right. \; \; \;   \Delta_{\mathrm{U}(W)}=\left\{
        \begin{array}{ll}
            \prod_{i=1}^m L_F(i,\eta_{E/F}^i) & \text{inert case,} \\
            \prod_{i=1}^m \zeta_F(i) & \text{split case,} 
        \end{array}
        \right.
    \end{equation*}
    \begin{equation*}
    \mathbf{d}(\chi \boxtimes \eta )=\prod_{\alpha \in \Sigma^+_{\mathrm{nd}}} \frac{1}{1-\langle \chi\boxtimes \eta,\alpha^\vee \rangle},
\end{equation*}
and $\mathbf{b}(\chi,\eta)$ is
\begin{equation}
\label{eq:b_defi}
    \left\{
    \begin{array}{ll}
        \displaystyle \prod_{\substack{1 \leq i \leq {n_-} \\ 1 \leq j \leq {m_-}}} L_E^{-1}(\frac{1}{2},\chi_i \eta_j) \prod_{i < r+j} L_E^{-1}(\frac{1}{2},\chi_i \eta_j^{-1})\prod_{i > r+j} 
        L_E^{-1}(\frac{1}{2},\chi_i^{-1} \eta_j) & \text{inert $n$ even,} \\
        \displaystyle \prod_{\substack{1 \leq i \leq {n_-} \\ 1 \leq j \leq {m_-}}} L_E^{-1}(\frac{1}{2},\chi_i \eta_j)L_E^{-1}(\frac{1}{2},-\chi_i) L_E^{-1}(\frac{1}{2},\eta_j )\prod_{i < r+j} L_E^{-1}(\frac{1}{2},\chi_i \eta_j^{-1})\prod_{i > r+j}L_E^{-1}(\frac{1}{2},\chi_i^{-1} \eta_j) & \text{inert $n$ odd,} \\
        \displaystyle \prod_{i+j<r'+1} L_F^{-1}(\frac{1}{2},\chi_i \eta_j) \prod_{i+j > r'+1}  L_F^{-1}(\frac{1}{2},\chi_i^{-1} \eta_j^{-1}) & \text{split.} \\
    \end{array}
    \right.
\end{equation}
\end{theorem}

Although Theorem~\ref{thm:W_formula} is the appropriate result for applications, the factors $\mathbf{b}(\chi,\eta)$ lack a satisfying interpretation. In the split case, we propose a more robust version. Let $\widehat{G}$ be the Langlands dual group of $G$ with Borel pair $(\widehat{T},\widehat{B}^+)$. We have the $L$-groups ${}^L G$ and ${}^L T$. Let $S \in {}^L T$ be a representative of the Satake parameter of $\sigma$. We define in \S\ref{eq:Rankin--Selberg_rep} a representation $\mathcal{R}_{\overline{\mu}} : {}^L G \to \GL(\cc^n \otimes \cc^m \oplus \cc^n \otimes \cc^m)$. It satisfies
\begin{equation}
\label{eq:RL}
    L(s,\sigma \otimes \overline{\mu})=\det(1-q_F^{-s} \mathcal{R}_{\overline{\mu}}(S)).
\end{equation}
The representation $\mathcal{R}_{\overline{\mu}}$ is symplectic. In \S\ref{subsubsec:symplectic} we choose a Lagrangian subspace $\mathcal{Y}$ which is stable by the restriction of $\mathcal{R}_{\overline{\mu}}$ to ${}^L T$. Let $\mathcal{Y}_{\overline{\mu}}$ be this representation of ${}^L T$. Denote by $D_{\widehat{G}/\widehat{B}^+}(S)$ the determinant of $1-\mathrm{Ad}(S)$ on $\mathrm{Lie}(\widehat{G}) / \mathrm{Lie}(\widehat{B}^+)$. We introduce in \S\ref{subsubsec:spheri_vectors} an element $\phi^\times \in \overline{\nu}_{\mu,\psi}$. It is not spherical but is fixed by $I_W \cap w_0 I_W w_0^{-1}$, where $I_W$ is a Iwahori subgroup of $\GL_m$. It depends on the choice of a Schrodinger model of $\overline{\nu}_{\mu,\psi}$ which is equivalent to the choice of a Lagrangian subspace $Y$ of the symplectic space $\mathrm{Res}_{E/F} W$. The Lagrangian subspaces $\mathcal{Y}$ and $Y$ are chosen to be compatible. The following proposition is proved in \S\ref{subsubsec:end_proof_formula_2}.

\begin{prop}
\label{prop:W_formula2}
    Assume that we are in the split case and that $\Phi_{\chi,\overline{\mu} \eta} \neq 0$. For $\lambda \in \Lambda^{-}$ we have 
     \begin{equation*}
        \frac{( \Phi_{\chi,\overline{\mu} \eta}(\lambda),\phi^\times )_{\nu}}{(\Phi_{\chi,\overline{\mu}\eta}(1),\phi^\circ)_\nu} = \Delta_{\mathrm{U}(W)} \sum_{w \in W_G} \frac{\det(1-q_F^{-\frac{1}{2}} \mathcal{Y}_{\overline{\mu}}
        (wS))}{D_{\widehat{G}/\widehat{B}^+}(wS)} 
         \left(w_0 w (\chi \boxtimes \eta) \delta_{B^+}^{-\frac{1}{2}}\right) (\lambda).
    \end{equation*}
\end{prop}
In the inert case, there is no counterpart to $\phi^\times$ as any $I_W \cap w_0 I_W w_0^{-1}$-fixed vector is $K_W$-fixed (Lemma~\ref{lem:spheri_vectors}), and moreover there is no Lagrangian subspace of $\mathcal{R}_{\overline{\mu}}$ stable by ${}^L T$. However, $\mathcal{R}_{\overline{\mu}}$ makes sense and \eqref{eq:RL} is satisfied. The factor $\mathbf{b}(\chi,\eta)$ of Theorem~\ref{thm:W_formula} can then be interpreted as $\det(1-q_F^{-\frac{1}{2}} \left(\mathcal{R}_{\overline{\mu}}\right)_{|\mathcal{V}_-}(S))$, where $\mathcal{V}_-$ is an isotropic subspace of $\mathcal{R}_{\overline{\mu}}$ which is stable by ${}^L T$ but not Lagrangian.

Our proof of Theorem~\ref{thm:W_formula} and Proposition~\ref{prop:W_formula2} follows the strategy initiated in \cite{Cas} and \cite{CS} for Whittaker functionals, which has been used to establish formulae for Whittaker--Shintani functions for $\GL_n$ in \cite{KMS2}, orthogonal groups in \cite{KMS}, Bessel models for unitary groups in \cite{Khoury} and $\GL_n$ in \cite{Zhang}, and for symplectic groups in \cite{Shen}. This last case is the closest in spirit to ours. In our setting, the main difficulties come from dealing with the representation $\overline{\nu}_{\mu,\psi}$: this is done in \S\ref{subsec:Jacobi_defi} by interpreting $\sigma \otimes \overline{\nu}_{\mu,\psi}$ as a principal series on the bigger Jacobi group $\mathrm{U}(V) \times J$. This does not fit into the usual theory of induction on reductive groups and complicates the computations. 

\subsection{Ichino--Ikeda conjecture}
We give an application of Theorem~\ref{thm:W_formula} to the Ichino--Ikeda conjecture for Fourier--Jacobi periods formulated in \cite{Xue}. Assume that $\sigma$ is tempered. We equip it with an invariant inner product $(\cdot,\cdot)$. We take a model over $\oo_F$ of $H$ (see \S \ref{subsubsec:model}) and give $H$ the left-invariant measure $dh$ normalized so that $H(\oo_F)$ has volume $1$. Define the local Ichino--Ikeda period for $\varphi \in \sigma$, $\phi \in \overline{\nu}_{\mu,\psi}$
\begin{equation}
\label{eq:P_defi}
    \mathcal{P}_{H}(\varphi,\phi):=\int_{H} (\sigma(h) \varphi,\varphi) (\overline{\nu}_{\mu,\psi}(h) \phi,\phi)_{\nu} dh. 
\end{equation}
This integral is only convergent if $r=0$ and needs to be interpreted as the unique continuous extension of a linear form on $C_c^\infty(G)$ to the space of tempered functions $C^w(G)$ otherwise (see \S\ref{subsec:FJ_p} for a precise statement which is not specific to the unramified setting and will be given for any local field of characteristic zero). The following theorem is the unramified Ichino--Ikeda conjecture. It had been proved in \cite{Harris} for Bessel models on unitary groups. For Fourier--Jacobi models it was only known in corank zero by \cite{Xue}. It is proved in \S\ref{subsec:proof_II}.

\begin{theorem}
\label{thm:II}
    Let $\varphi^\circ \in \sigma$ be a non-zero spherical vector. Then
    \begin{equation*}
        \frac{\mathcal{P}_{H}(\varphi^\circ,\phi^\circ)}{(\varphi^\circ,\varphi^\circ)(\phi^\circ,\phi^\circ)_{\nu}}=\Delta_{\mathrm{U}(V)} \frac{L(\frac{1}{2},\sigma \otimes \overline{\mu})}{L(1, \sigma, \mathrm{Ad})}.
    \end{equation*}
\end{theorem}

\subsection{Organization of the paper}

The paper is organized as follows. In \S\ref{section:WS} we give a precise definition of Whittaker--Shintani functions for Fourier-Jacobi models using the Jacobi group $J$. In \S\ref{sec:mult_1}, we relate these functions to the Heisenberg-Weil representation $\overline{\nu}_{\mu,\psi}$ and prove the multiplicity one Theorem \ref{thm:mult1intro}. In the process, we show Proposition~\ref{prop:module_intro}. We then proceed in \S\ref{section:pairing} to produce an element $\mathcal{W}_{\chi,\eta}^I \in \mathrm{WS}_{\chi,\eta}$ defined by an integral expression for $(\chi,\eta)$ in some open subset, which we then extend to $(\chi,\eta)$ in general position by meromorphic continuation. In \S\ref{section:Formula}, we compute an explicit formula for $\mathcal{W}_{\chi,\eta}^I$: this is done by calculating the $\gamma$ factors caused by the uniqueness of Fourier--Jacobi models. To prove Theorem~\ref{thm:W_formula}, it then remains to compute the normalization value $\mathcal{W}_{\chi,\eta}^I(1)$. In \S\ref{section:unfolding} we temporarily leave the unramified setting to discuss convergence and state an unfolding identity for tempered Ichino--Ikeda periods. In \S\ref{section:L}, we reformulate our results in terms of Satake parameters and representations of the Langlands dual group ${}^L G$. We then use the Weyl character formula to determine $\mathcal{W}_{\chi,\eta}^I(1)$ (thus completing the proof of Theorem~\ref{thm:W_formula}), and finally prove Theorem \ref{thm:II} by reduction to the known corank zero result using \S\ref{section:unfolding}.

\subsubsection{Acknowledgements} The author thanks Rapha\"el Beuzart-Plessis for helpful discussions and comments. He is also grateful to Hang Xue for explanations on his results. He thanks the anonymous referee for useful suggestions to improve the exposition.

This work was partly funded by the European Union ERC Consolidator Grant, RELANTRA, project number 101044930. Views and opinions expressed are however those of the author only and do not necessarily reflect those of the European Union or the European Research Council. Neither the European Union nor the granting authority can be held responsible for them.

\section{Whittaker--Shintani functions}

\label{section:WS}

\subsection{General notations}
\label{subsec:notations}

\subsubsection{}
Unless specified otherwise, $F$ will be a non-Archimedean local field of characteristic zero, with ring of integers $\oo_F$ and maximal ideal $\pedro_F$. We denote by $q_F$ the cardinality of the residual field $\oo_F / \pedro_F$ and assume that it is odd. We fix an uniformizer $\varpi$ and normalize the valuation $v_F$ and the absolute value $\Val{\cdot}_F$ on $F$ such that $v_F(\varpi)=1$ and $\Val{x}_F=q_F^{-v_F(x)}$ for $x \in F$. 

Let $E$ be a quadratic unramified extension of $F$ (referred to as the \emph{inert} case), or $E:=F \times F$ (the \emph{split} case). In the inert case, write $\oo_E$ and $q_E$ as before, and normalize $v_E$ and $\Val{\cdot}_E$ such that $v_E(\varpi)=1$ and $\Val{x}_E=q_E^{-v_E(x)}$. Let $c \in \gal(E/F)$ be the non-trivial element, and set $E^-:=\{ x \in E \; | \; c(x)=-x\}$. Fix $\tau \in \oo_E^\times \cap E^-$ (which exists as $q_F$ is odd). In the split case, write $\oo_E:=\oo_F \times \oo_F$. We have $\gal(E/F)=\{1,c\}$ where $c$ switches the coordinates of elements in $F \times F$. Let $\Tr_{E/F}$ and $N_{E/F}$ be the trace and norm maps of $E/F$ which make sense in both settings. Let $\Val{\cdot}$ be $\Val{\cdot}_E$ in the inert case and $\Val{\cdot}_F$ in the split case, $v$ be $v_E$ in the inert case and $v_F$ in the split case, $q$ be $q_E$ in the inert case and $q_F$ in the split case.

\subsubsection{}
\label{subsubsec:V_W}
Let $(V,\langle.,.\rangle_V)$ be an $n$-dimensional non-degenerate skew-Hermitian space over $E/F$ and $W \subset V$ be a non-degenerate subspace of dimension $m$ such that $n=m+2r$ with $r \geq 0$. In the inert case, set $n_{-}:=\left\lfloor \frac{n}{2} \right\rfloor$ and $m_{-}:=\left\lfloor \frac{m}{2} \right\rfloor$. If $n$ is even (referred to as the \emph{even} case), define $n_+=n_-$ and $m_+=m_-$, while if $n$ is odd (the \emph{odd} case), set $n_+=n_-+1$ and  $m_+=m_-+1$. In the split case, set $n_-=n_+=n$ and $m_-=m_+=m$. Set $r'=m+r$. 

In the inert case, we assume that $V=W\oplus W^\perp$ is split. This means that there exists an $E$-basis $\{ v_i, v_i^*  \; | \; 1 \leq i \leq r \}$ of $W^\perp$ and an $E$-basis $\{w_i,w_i^* \; | \; 1 \leq i \leq m_+ \}$ of $W$ (with the convention $w_{m_+}=w_{m_+}^*$ in the odd case) such that for all $1 \leq i, j \leq r$, $1 \leq i' \leq m_+$ and $1 \leq j' \leq m_-$ we have 
\begin{equation}
\label{eq:split_basis}
\left\{ \begin{array}{l}
     \langle v_i,v_j \rangle_V=\langle v_i^*, v_j^* \rangle_V=\langle w_{i'},w_{j'} \rangle_V=\langle w_{i'}^*,w_{j'}^* \rangle_V=0,   \\
     \langle v_i,v_j^* \rangle_V= \delta_{i,j}, \\
     \langle w_{i'},w_{j'}^* \rangle_V= \delta_{i',j'}, \\
     \langle w_{m_+},w_{m_+} \rangle_V= \tau \text{ in the odd case.} \\
\end{array} \right.
\end{equation}
We obtain an $E$-basis $\mathfrak{B}=(v_1, \hdots, v_r, w_1, \hdots ,w_{m_+}, w_{m_-}^*, \hdots, w_1^*, v_r^*, \hdots v_1^*)$ of $V$.

In the split case, the above construction can be made explicit and it will be convenient to choose a particular $\mathfrak{B}$. We have $V=F^n \times F^n$ and we can and will assume that $\langle.,.\rangle_V$ is given by $\langle (x,y), (x',y') \rangle_V=({}^t x y', -{}^tyx')$, where we identify $F^n$ with column vectors. Let $\{\mathbf{e}_i\}$ and $\{\mathbf{e}_i^*\}$ be the canonical basis of $F^n \times \{0\}$ and $\{0\} \times F^n$ respectively. Define for $1 \leq i \leq n_+$ the elements $v_i=\mathbf{e}_i+\mathbf{e}_{n-i+1}^*$ and $v_i^*=-\mathbf{e}_{n-i+1}+\mathbf{e}_{i}^*$. The basis $\mathfrak{B}=(v_1, \hdots, v_{n_+},v_{n_-}^*, \hdots, v_1^*)$ is an $E$-basis of $V$ satisfying \eqref{eq:split_basis} with $\tau=(1,-1)$ (up to taking $w_i=v_{r+i}$ and $w_i^*=v_{r+i}^*$). Set $W=\mathrm{span}_E(v_i,v_i^* \; | \; r+1 \leq i \leq n_+)$. For $1 \leq j \leq m$, set $w_j=\mathbf{e}_{j+r}$ and $w_j^*=\mathbf{e}_{j+r}^* \in W$. Note that $(w_1, \hdots, w_m, w_{1}^*, \hdots ,w_m^*)$ is a $F$-basis of $W$.

In both cases, set $X=\mathrm{span}_E(v_i \; | \; 1 \leq i \leq r)$ and $X^*=\mathrm{span}_E(v_i^* \; | \; 1 \leq i \leq r)$, so that $V=X \oplus W \oplus X^*$.

\subsubsection{}
Denote by $\mathrm{U}(V)$ and $\mathrm{U}(W)$ the subgroups of $\GL(V)$ and $\GL(W)$ consisting of the $E$-linear unitary transformations of $V$ and $W$ respectively. Note that $n_-$ (resp. $m_-$) is always the split $F$-rank of $\mathrm{U}(V)$ (resp. $\mathrm{U}(W)$). We will identify $\mathrm{U}(W)$ as the subgroup of $\mathrm{U}(V)$ of $g \in \mathrm{U}(V)$ which act by the identity on $W^\perp$. The basis $\mathfrak{B}$ chosen in \S\ref{subsubsec:V_W} yields an isomorphism $\GL(V) \cong \GL_n(E)$ and two embeddings $\mathrm{U}(W), \mathrm{U}(V) \hookrightarrow \GL_n(E)$. Denote by $(T_V,B_V)$ and $(T_W,B_W)$ the inverse image in $\mathrm{U}(V)$ and $\mathrm{U}(W)$ of the torus of diagonal matrices and the Borel subgroup of upper triangular matrices by this morphism. In the split case, we will furthermore identify $\mathrm{U}(V)$ with $\GL_n(F)$ and $\mathrm{U}(W)$ with $\GL_m(F)$ by projecting on the first coordinate. Note that through this isomorphism $B_V$ and $B_W$ are identified with the Borel subgroups of upper triangular matrices in $\GL_n(F)$ and $\GL_m(F)$ respectively.

Denote by $N_V$ and $N_W$ the unipotent radicals of $B_V$ and $B_W$ respectively. The opposite Borel subgroups will be denoted by $B_V^-=T_VN_V^-$ and $B_W^-=T_WN_W^-$. Set $T=T_V \times T_W$, $B=B_V \times B_W$ and $B^+=B_V^- \times B_W$. 

Any $t_V \in T_V$ is of the form $t_V(\mathbf{t})$ where
\begin{equation}
\label{eq:t_V}
    t_V(\mathbf{t})=\left\{
     \begin{array}{ll}
          \mathrm{diag}(t_1,\hdots,t_{n_-},c(t_{n_-})^{-1},\hdots,c(t_{1})^{-1}), \quad \mathbf{t} \in (E^\times)^{n_-}, &  \text{ inert even case,}\\
          \mathrm{diag}(t_1,\hdots,t_{n_-},t_{n_+},c(t_{n_-})^{-1},\hdots,c(t_{1})^{-1}), \quad \mathbf{t} \in (E^\times)^{n_-} \times \mathrm{U}(1), &  \text{ inert odd case,}\\
           \mathrm{diag}(t_1,\hdots,t_{n}), \quad \mathbf{t} \in (F^\times)^n, &  \text{ split case.}\\
     \end{array}
    \right.
\end{equation}
If $t \in T_V$ we will write $t_i$ for the coordinates of $\mathbf{t}$ in \eqref{eq:t_V}. Set $\Val{t}=\prod_{i=1}^{n_-} \Val{t_i}$. We let $\Lambda_V$ be the group of cocharacters of $T_V$, and $\Lambda_V^+$ be its cone of positive cocharacters with respect to $B_V$, i.e.
\begin{equation*}
    \Lambda_V^+:=
     \left\{ 
    \begin{array}{ll}
       \{(a_1,\hdots,a_{n_-}) \in \zz^{n_-} \; | \; a_1 \geq \hdots \geq a_{n_-} \geq 0 \},  & \text{ inert case;}  \\
        \{(a_1,\hdots,a_{n}) \in \zz^n \; | \; a_1 \geq \hdots \geq a_{n}\}, &  \text{ split case}.
    \end{array}
    \right.
\end{equation*}
If $\lambda_V \in \Lambda_V^+$, we will simply write $\lambda_V$ for $\lambda_V(\varpi)$. 

The subgroup of $\mathrm{U}(V)$ of automorphisms preserving $X$ and $X^*$ and acting by identity on $W$ is isomorphic to $\GL_r(E)$ and is denoted by $G_r$. Set $T_r=G_r \cap T_V$ and $B_r=G_r \cap B_V$. Let $\Lambda_r$ be its group of cocharacters and $\Lambda_r^+$ be its the subcone of positive cocharacters with respect to $B_r$.

With the same definitions, we also have $T_W$, $t_W(\mathbf{t})$, $\Lambda_W$, $\Lambda_W^{+}$, $\Lambda_W^{-}$ and $\Val{\cdot}$. Let $\Lambda=\Lambda_V \times \Lambda_W$ be the group of cocharacters of $T$, and let $\Lambda^-=\Lambda_V^+ \times \Lambda_W^-$ be the cone of negative cocharacters relatively to $B^+$.

\subsubsection{}
\label{subsubsec:root_subgroups}
Let $A_V$ (resp. $A_W$) be the split component of $T_V$ (resp. $T_W$). Let $\langle.,.\rangle$ be the canonical pairing between the groups of characters $X^*(A_V)$ and cocharacters $X_*(A_V)$ of $A_V$. Denote by $\Sigma_V$ the set of roots of $A_V$ in $\mathrm{U}(V)$, $\Sigma_{V,\mathrm{nd}}$ the subset of non-divisible roots, $\Sigma_{V,\mathrm{nd}}^+$ the set of non-divisible positive roots with respect to $B_V$. Let $\Delta_V$ be the simple roots in $\Sigma_{V,\mathrm{nd}}^+$. If $\alpha \in \Sigma_V$, write $\alpha^\vee$ for the associated coroot. Then we write $\alpha >0$ if for all $\beta \in \Delta_{V}$ we have $\langle \alpha, \beta^\vee \rangle >0$. Let $W_V$ be the Weyl group of $\mathrm{U}(V)$. It has an action on $\Sigma_V$ which we will write $w. \alpha$. The longest element in $W_V$ will be denoted by $w_0$

For each $1 \leq i \leq n$, let $e_i$ be the character of $A_V$ such that $e_i(t)=t_i$ with $t \in A_V$. If $\alpha \in \Delta_V$, we will write $N_{\alpha}$ for the corresponding root subgroup. Denote by $P_\alpha$ the parabolic subgroup corresponding to $\alpha$, with Levi component $G_\alpha$ whose derived group is $D(G_\alpha)$. Denote by $w_\alpha$ the simple reflection corresponding to $\alpha$. The $N_\alpha$'s are isomorphic to product of copies of $F$ or $E$ through maps $n_\alpha$. More precisely, denote by $\{E_{i,j}\}$ the canonical basis of $\mathrm{M}_n(E)$ in the inert case, and $\mathrm{M}_n(F)$ in the split case, where $\mathrm{M}_n(E)$ (resp. $\mathrm{M}_n(F)$) is the vector space of $n \times n $-matrices over $E$ (resp. over $F$). Then we can write each $\alpha \in \Delta_V$ as $\alpha=e_i-e_{i+1}$ and we have the description 
\begin{center}
\begin{tabular}{ |c|c|c|c| } 
 \hline
 & Case  & $n_\alpha$ & $D(G_\alpha)$ \\
 \hline \hline
 $1 \leq i < n_- $ & inert & $n_\alpha(x)=I_n+x E_{i,i+1}-c(x)E_{n-i,n-i+1}$, 
 $x \in E$ & $\SL_2(E)$  \\ 
 \hline
  $1 \leq i < n $  & split & $n_\alpha(x)=I_n+x E_{i,i+1}$, 
 $x \in F$ & $\SL_2(F)$ \\ 
 \hline
  $i=n_-$ &  inert even & $n_\alpha(x)=I_n+x E_{n_-,n_-+1}$, 
 $x \in F$ & $\SL_2(F)$ \\ 
  \hline
  $i=n_-$ & inert odd
  &  \begin{tabular}{@{}c@{}}$n_\alpha(x_1,x_2)=I_n-\tau c(x_1) E_{n_-,n_+} + x_1 E_{n_+,n_++1} +x_2 E_{n_-,n_++1}$   \\ $x_1, x_2 \in E, \; \Tr_{E/F}(\frac{x_2}{\tau})=-N_{E/F}(x_1)$\end{tabular} & $\mathrm{SU}_3(F)$ \\ 
 \hline
\end{tabular}
\end{center}

Similar notations hold for $\mathrm{U}(W)$. To avoid any confusion we will write $n'_{\beta}$ if $\beta \in \Sigma_W$. 

\subsubsection{}
\label{subsubsec:Jacobi}
Consider the $F$-vector spaces $\mathbb{V}:=\R_{E/F} V$ and $\mathbb{W}:=\R_{E/F} W$ equipped with the symplectic form $\langle.,.\rangle_{\mathbb{V}}:=\Tr_{E/F} \circ \langle.,.\rangle_V$ and let $\mathbb{H}(\mathbb{W})=\mathbb{W} \times F$ be its Heisenberg group, with group law $(w_1,z_1).(w_2,z_2)=(w_1+w_2,z_1+z_2+\frac{1}{2}\langle w_1,w_2 \rangle_{\mathbb{V}})$. If $r \geq 1$, we have an injective morphism $h : \mathbb{H}(\mathbb{W}) \to N_V$ characterized by 
\begin{equation}
\label{eq:h_defi}
    \left\{
    \begin{array}{llll}
        h(w,z)(v_i)&=&v_i & 1 \leq i \leq r, \\
        h(w,z)(w')&=& w'-\langle w',w \rangle_V v_r & w' \in W, \\
        h(w,z)(v_r^*)&=&(-\frac{1}{2}\langle w,w \rangle_V +z)v_r + w + v_r^*, & \\
        h(w,z)(v_i^*)&=&v_i^* & 1 \leq i \leq r-1.
    \end{array}
    \right.
\end{equation}
If $g_1, g_2 \in \mathrm{U}(V)$, we will write ${}^{g_1}g_2:=g_1 g_2 g_1^{-1}$. The subgroup $\mathbb{H}(\mathbb{W})$ of $\mathrm{U}(V)$ is stable by conjugation by $\mathrm{U}(W)$ and satisfies
\begin{equation}
\label{eq:Jacobi_relation}
     {}^{g_W}h(w,z)=h(g_W w,z), \; g_W \in \mathrm{U}(W),\; w \in W, \; z \in F.
\end{equation}
We define
\begin{equation}
\label{eq:Jacobi_defi}
    J:=\mathrm{U}(W) \ltimes \mathbb{H}(\mathbb{W}),
\end{equation}
the \emph{Jacobi} group of $W$. If $r \geq 1$ it is a subgroup of $\mathrm{U}(V)$ by \eqref{eq:h_defi}. Its center $h(0,F)$ will be denoted by $Z$.

Let $U$ be the unipotent radical of the parabolic subgroup of $\mathrm{U}(V)$ stabilizing the flag $E v_1 \subset E v_1 \oplus E v_2 \subset \hdots \subset E v_1 \oplus \hdots E v_{r-1}$. Define
\begin{equation}
\label{eq:G_defi}
   G=\mathrm{U}(V)\times \mathrm{U}(W), \quad H:=J \ltimes U \text{ if } r \geq 1, \quad H=\mathrm{U}(W) \text{ if } r=0.
\end{equation}
Then $H$ is a subgroup of $G$ by the natural inclusion in the first component and the projection on $\mathrm{U}(W)$ in the second (this is the diagonal embedding if $r=0$). Set $W_G:=W_V \times W_W$.

\subsubsection{}
\label{subsubsec:L}
For $1 \leq k \leq r$, set $G_k:=\GL_k(E)$ seen as the subgroup of $\mathrm{U}(V)$ stabilizing $\mathrm{Span}_E(v_i \;| \; 1 \leq i \leq k)$ and $\mathrm{Span}_E(v_i^* \;|\; 1 \leq i \leq k)$ and acting trivially on their orthogonal complement. Set $B_k:=G_k \cap B_V=T_k N_k$. Let $P(X)$ be the parabolic subgroup of $\mathrm{U}(V)$ stabilizing $X$. Its Levi subgroup stabilizing $X^*$ is isomorphic to $G_r \times \mathrm{U}(W)$. Let $N(X)$ be its unipotent radical. Set $L:=G \times G_r$. It is a Levi subgroup of the parabolic $\mathrm{U}(V) \times P(X)$ of $\mathrm{U}(V) \times \mathrm{U}(V)$. If $r \geq 1$ define $H^L=H \times N_r \subset L$.

Let $P$ be the parabolic subgroup of $\mathrm{U}(V)$ stabilizing the full flag $0 \subset E v_1 \subset \hdots \subset X$ of $X$, and $N$ be its unipotent radical. Then we have the alternative description $ H=\mathrm{U}(W) \ltimes N$. Note that $HB^+$ is open in $G$.

\subsubsection{}
\label{subsubsec:Y_defi}
In the inert case, if $R$ is a subgroup of $E$ set $W(R):=\mathrm{span}_R(w_i, w_i^* \; | \; 1 \leq i \leq m_+)$, $Y_-(R):=\mathrm{span}_R(w_i \; | \; 1 \leq i \leq m_-)$, $Y_-^*(R):=\mathrm{span}_R(w_i^* \; | \; 1 \leq i \leq m_-)$ and $Y_{m_+}=\{0\}$ in the even case, $Y_{m_+}(R)=Y^*_{m_+}(R)=R w_{m_+}$ in the odd case. Let $Y_+:=Y_- \oplus Y_{m_+}$ and $Y_+^*:=Y_{m_+}^* \oplus Y_-^*$. Set $ Y:=Y_-(E) \oplus Y_{m_+}(\oo_E)$.

In the split case, we use the same notations with $Y_-(R)=\mathrm{span}_R(w_i \; | \; 1 \leq i \leq m)$, $Y_-^*(R)=\mathrm{span}_R(w_i^* \; | \; 1 \leq i \leq m)$ and $Y_{m_+}=\{0\}$. Set $Y:=Y_-(E)$. All these groups embed into $J$ by $h$ and we identify them with their image. Consider 
\begin{equation}
    N_J:=N_W \ltimes YZ, \text{ and } B_J:=T_W \ltimes N_J.
\end{equation}

We can identify $Y_-(E)$ and $Y_-^*(E)$ with $E^{m_-}$ (inert case) or $F^m$ (split case) with our choice of basis. For $y \in Y_-(E)$ or $y^* \in Y_-^*(E)$, we will use $(y_i)$ and $(y_i^*)$ to write the coordinates and write $t_W(y)$ or $t_W(y^*)$ as in \eqref{eq:t_V} if they are all non-zero. Let $1_{Y_+}^* \in Y_+^*(E)$ whose coordinates are all $1'$s. Define $\mu_{+}^*:=h(1_{Y_+^*},0)$. 

The subgroup of $\mathrm{U}(W)$ stabilizing $Y_-$ and $Y_-^*$ and acting by identity on $Y_{m_+}$ can be identified with $\GL_{m_-}(E)$ in the inert case and $\GL_m(F)$ in the split case by restricting to $Y_-$. Then $g \in \GL_{m_-}(E)$ (resp. $\GL_m(F)$) acts by $g$ on $Y_-(E)$ and $(g^*)^{-1}$ on $Y_-^*(E)$, where $g^*={}^tc(g)$ (resp. $g^*={}^tg$). Note that this property holds in the split case because of our particular choice of $F$-basis $\{w_j, w_j^*\}$ of $W$ (see \S \ref{subsubsec:V_W}). Combined with \eqref{eq:Jacobi_relation}, we get for $w=y+y_{m_+}+y^* \in Y_-(E) \oplus Y_{m_+}(E) \oplus Y_-^*(E)=W$ and $g \in \GL_{m_-}(E)$ (resp. $\GL_{m}(F)$ in the split case) 
\begin{equation}
\label{eq:conj_identity}
    {}^gh(w,0)=h(gy + y_{m_+} + (g^*)^{-1} y^*,0).
\end{equation}

\subsubsection{}
\label{subsubsec:model}
Let $K_V \leq \mathrm{U}(V)$ be the stabilizer of the lattice $\oplus_{v \in \mathfrak{B}} \oo_E v$ and set $K_W=K_V \cap \mathrm{U}(W)$. These are hyperspecial maximal open compact subgroups of $\mathrm{U}(V)$ and $\mathrm{U}(W)$ respectively. Recall that $\mathrm{U}(V) \subset \GL_n(E)$. For any group $\LAG$ that we have considered so far, consider for $i \geq 0$ the group $\LAG^i = \LAG \cap K_V \cap (I_n + M_{n \times n}(\varpi^i \oo_E))$.

Set $K_J:=K_W \ltimes W^0 Z^0$, an open-compact subgroup of $J$. Set $I_V=N_V^{-,1} T_V^0 N_V^0$ and $I_W=N_W^{-,1} T_W^0 N_W^0$ which are Iwahori subgroups of $\mathrm{U}(V)$ and $\mathrm{U}(W)$. Define $I_J:=I_W \ltimes W^0 Z^0$. Set $K=K_V \times K_W \subset G$. 

\subsubsection{}
\label{subsubsec:measure} 
We equip all our groups $\LAG$ with the left-invariant Haar-measure giving $\LAG^0$ volume $1$. In particular, $db_J:=db_W dy dz$ on $B_J=B_W \ltimes YZ$. Recall that $\Val{t_W}=\prod_{j=1}^{m_-} \Val{t_{W,j}}$. The modular character is
\begin{equation}
\label{eq:mod_char}
    \delta_{B_J}(t_W n_J)=\delta_{B_W}(t_W) \Val{t_W} \; t_W \in T_W, \; n_J \in N_J.
\end{equation}

\subsubsection{}
\label{subsubsec:Gross_defi}
Denote by $\zeta_F$ and $\zeta_E$ the zeta functions associated to $F$ and $E$, i.e. $\zeta_F(s)=(1-q_F^{-s})^{-1}$, $\zeta_E(s)=(1-q_E^{-s})^{-1}$ in the inert case and $\zeta_E(s)=(1-q_F^{-s})^{-2}$ in the split case. For $x \in \cc$, define $ L_F(s,x)=(1-xq_F^{-s})^{-1}$ and $L_E(s,x)=(1-xq_E^{-s})^{-1}$. Set $\Delta_{\mathrm{U}(V)}=L_{\mathrm{U}(V)}(0)$, $\Delta_{G_k}=L_{G_k}(0)$ and $\Delta_{T_W}=L_{T_W}(0)$, where $L_{\mathbb{G}}$ is the Artin--Tate $L$-function associated to $\mathbb{G}$ defined in~\cite{Gro97}. In our case,
\begin{equation}
\label{eq:Delta_defi}
    \Delta_{\mathrm{U}(V)}=\prod_{i=1}^n L_F(i,\eta^i_{E/F}), \; \Delta_{G_k}=\prod_{i=1}^k \zeta_E(i), \; \Delta_{T_W}=\left\{
        \begin{array}{ll}
            \zeta_F(1)^{m_-} L(1,\eta_{E/F})^{m_+} & \text{inert case,} \\
            
            \zeta_F(1)^{m} & \text{split case,} \\
        \end{array} \right.
\end{equation}
where $\eta_{E/F}$ is the quadratic character associated to $E/F$. We will also use $\Delta_{T_W}'$ which is $\Delta_{T_W}$ in the inert even case and the split case, and $\Delta_{T_W}'=\zeta_F(1)^{m_-} L(1,\eta_{E/F})^{m_-}$ in the inert odd case.

\subsection{Representations}
\label{subsec:Jacobi_defi}

Let $G_2$ be a locally compact totally disconnected group, let $G_1 \leq G_2$ be a closed subgroup. If $\pi$ is a smooth representation of $G_1$, let $\Ind_{G_1}^{G_2} (\pi)$ be the smooth induction, and $\ind_{G_1}^{G_2} (\pi)$ be the compact smooth induction (see \cite[Section~III.2.2]{Ren}). We reserve the notation $I$ for parabolic induction.

We fix once and for all $\psi$ an unramified character of $F$ of conductor $\oo_F$ which will appear in the definition of Whittaker--Shintani functions. It is a consequence of Theorem~\ref{thm:W_formula} that the latter are independent of $\psi$, so that we will often drop it from the notations.

Let $\chi$ and $\eta$ be unramified characters of $T_V$ and $T_W$ extended to $B_V$ and $B_W$, that we identify with elements of $(\cc^\times)^{n_-}$ and $(\cc^\times)^{m_-}$ via the rules $\chi(t_V(\mathbf{t}))=\prod_{i=1}^{n_-} \chi_i^{v(t_i)}$ and $\eta(t_W(\mathbf{t}))=\prod_{j=1}^{m_-} \eta_j^{v(t_j)}$. We also extend $\Val{\cdot}$ to $B_V$ and $B_W$. Note that the map $ (\eta \overline{\psi})(b_W h(y,z)):=\eta(b_W) \overline{\psi}(z)$, where $b_W \in B_W$,  $y \in Y$ and $z \in F$, defines a character of $B_J$. We introduce the parabolic induced representations
\begin{equation*}
    I_{B_V}^{\mathrm{U}(V)}(\chi)=\Ind_{B_V}^{\mathrm{U}(V)}(\chi \delta_{B_V}^{\frac{1}{2}}), \quad I_{B_W}^{\mathrm{U}(W)}(\eta)=\Ind_{B_W}^{\mathrm{U}(W)}(\eta \delta_{B_W}^{\frac{1}{2}}), \quad I_{B_J}^{J}(\eta \overline{\psi})=\Ind_{B_J}^{J}(\eta \overline{\psi} \delta_{B_J}^{\frac{1}{2}}).
\end{equation*}

\begin{lem}
\label{lem:spheri}
    We have $I_{B_J}^{J}(\eta \overline{\psi})=\ind_{B_J}^{J}(\eta \overline{\psi}\delta_{B_J}^{\frac{1}{2}})$. Moreover,  $(I_{B_J}^{J}(\eta\overline{\psi}))^{K_J}$ is of dimension $1$.
\end{lem}

\begin{proof}
    Let $f \in I_{B_J}^{J}(\eta\overline{\psi})$. By the Iwasawa decomposition, we have $J=B_J Y_+^*(E)K_J$. Let $y^* \in Y_+^*(E)$. By smoothness, there exists an open neighborhood $\mathcal{U}$ of $0$ in $Y$ independent of $y^*$ such that for all $k \in K_J$ and $y \in \mathcal{U}$ we have $f(h(y^*,0)h(y,0)k)=f(h(y^*,0)k)$. But $f(h(y^*,0)h(y,0)k)=\overline{\psi}(\langle y^*,y \rangle_{\mathbb{W}})f(h(y^*,0)k)$. If $f(h(y^*,0)k) \neq 0$, this implies that $y^*$ lives in a compact of $Y_{+}^*(E)$. If $f$ is $K_J$-invariant, we take $\mathcal{U}=Y_+(\oo_E)$ so that $f(h(y^*,0)) \neq 0$ implies $y^* \in Y_+^*(\oo_E)$ as $\psi$ is unramified.
\end{proof}

For $\LAG \in \{\mathrm{U}(V), \mathrm{U}(W),J\}$, let $C_c^\infty(\LAG)$ be the space of smooth compactly supported functions on $\LAG$. We have surjective intertwining maps $F_\chi : C_c^\infty(\mathrm{U}(V)) \to I_{B_V}^{\mathrm{U}(V)}(\chi)$, $F_\eta : C_c^\infty(\mathrm{U}(W)) \to I_{B_W}^{\mathrm{U}(W)}(\eta)$ and $F_{\eta}^J : C_c^\infty(J) \to I_{B_J}^{J}(\eta\overline{\psi})$, the first one being for example defined by

\begin{equation*}
    F_\chi(f_V)(g_V)= \int_{B_V} (\chi^{-1} \delta_{B_V}^{-\frac{1}{2}})(b_V)f_V(b_V g_V) db_V, \quad f_V \in C_c^\infty(\mathrm{U}(V)), \; g_V \in \mathrm{U}(V).
\end{equation*}
We introduce the spherical vectors
\begin{equation}
\label{eq:spheri_vector}
    \Phi_\chi^\circ := F_\chi(1_{K_V}), \; \Phi_\eta^\circ := F_\eta(1_{K_W}) \text { and } \Phi_{\eta }^{\circ,J} := F_{\eta }^J(1_{K_J}),
\end{equation}
and for $w_V \in W_V$ and $w_W \in W_W$ the Iwahori-fixed vectors
\begin{equation}
\label{eq:Iwahori_vector}
    \Psi_{w_V,\chi} := F_\chi(1_{I_V w_V I_V}), \; \Psi_{w_W,\eta} := F_\eta(1_{I_W w_W I_W}) \text { and } \Psi_{w_W,\eta}^J := F_{\eta }^J(1_{I_J w_W I_J}).
\end{equation}

\subsection{Whittaker--Shintani functions}
\label{subsec:WS_functions}
Let $\mathcal{H}(\mathrm{U}(V),K_V)$, $\mathcal{H}(\mathrm{U}(W),K_W)$ and $\mathcal{H}(J,K_J)$ be the spherical Hecke algebras of $\mathrm{U}(V)$, $\mathrm{U}(W)$ and $J$. By Lemma~\ref{lem:spheri}, they act by characters $\xi_\chi$, $\xi_\eta$ and $\xi_{\eta\overline{\psi}}^J$ on $(I_{B_V}^{\mathrm{U}(V)}(\chi))^{K_V}$, $(I_{B_W}^{\mathrm{U}(W)}(\eta))^{K_W}$ and $(I_{B_J}^{J}(\eta \overline{\psi}))^{K_J}$ respectively. For any function $f$ on $\mathrm{U}(V)$, $\mathrm{U}(W)$ or $J$, we have the regular actions $L(g)f=f(g^{-1}.)$ and $R(g)f=f(.g)$. Set $\psi_E:=\psi \circ \frac{1}{2}\Tr_{E/F}$ and consider the generic character of $U$ defined by $\psi_{U}(u):=\psi_E \left( \sum_{i=1}^{r-1} \langle u(v_{i+1}),v_{i}^* \rangle_{V} \right)$ for $u \in U$.

We define the space $\mathrm{WS}$ of Whittaker--Shintani functions to be the $\cc$-vector space of functions $\mathcal{W} \in C^\infty(\mathrm{U}(V))$ if $r \geq 1$, or $\mathcal{W} \in C^\infty(J)$ if $r=0$, such that for $z \in Z, \; u \in U, \; k_J \in K_J, \; k_V \in K_V$ and $g\in \mathrm{U}(V)$ if $r\geq 1$, $g\in J$ if $r=0$, we have
\begin{equation*}
    \mathcal{W}(zuk_J g k_V)=\psi(z) \psi_U(u) \mathcal{W}(g).
\end{equation*}
It has a structure of $ \mathcal{H}(\mathrm{U}(V),K_V) \otimes \mathcal{H}(J,K_J)$-module by the $R \otimes L$ action. We denote by $\mathrm{WS}_{\chi,\eta}$ the $\xi_{\chi} \otimes \xi_{\eta \overline{\psi}}^J$-isotypic component of $\mathrm{WS}$, that is the space of $\mathcal{W}_{\chi,\eta} \in \mathrm{WS}$ such that
\begin{equation*}
    R(f_V)L(f_J)\mathcal{W}_{\chi,\eta}=\xi_{\chi}(f_V) \xi_{\eta \overline{\psi}}^J(f_J) \mathcal{W}_{\chi,\eta}, \quad f_V \in \mathcal{H}(\mathrm{U}(V),K_V), \quad f_J \in \mathcal{H}(J,K_J).
\end{equation*}

\section{Multiplicity one for Whittaker--Shintani functions}

\label{sec:mult_1}

In this section we show that for every unramified characters $\chi$ and $\eta$ of $T_V$ and $T_W$ respectively, the space $\mathrm{WS}_{\chi,\eta}$ has dimension at most one. In \S\ref{section:pairing} and \S\ref{section:Formula} we will build a non zero element $W_{\chi,\eta}^\circ$, so that it has dimension exactly one. This proves the multiplicity one result Theorem~\ref{thm:mult1intro}.

\subsection{Heisenberg--Weil representation of the Jacobi group}
\label{subsec:H-W}

By the $p$-adic Stone--von Neumann Theorem~\cite[Th\'eor\`eme~I.1]{MVW} there exists a unique smooth irreducible representation $\rho_{\psi}$  of $\mathbb{H}(\mathbb{W})$ with central character $\psi$. Consider a $S^1$ metaplectic cover $\mathrm{Mp}(\mathbb{W})$ of the symplectic group $\mathrm{Sp}(\mathbb{W})$, which is unique up to isomorphism. There is a natural map $\mathrm{U}(W) \to \mathrm{Sp}(\mathbb{W})$. The data of $\psi$ determines a smooth representation $\omega_{\psi}$ of $\mathrm{Mp}(\mathbb{W})$ (\cite[Section~2.II]{MVW}), realised on $\rho_{\psi}$. Let $\mu$ be a character of $E^\times$ lifting $\eta_{E/F}$. Then the data of $(\mu,\psi)$ determines a splitting of $\mathrm{Mp}(\mathbb{W})$ over $\mathrm{U}(W)$ which yields the Weil representation $\omega_{\mu,\psi}$ of $\mathrm{U}(W)$. We will write $\overline{\omega}_{\mu,\psi}=\omega_{\overline{\mu},\overline{\psi}}$ and $\overline{\rho}_\psi=\rho_{\overline{\psi}}$. By definition, for any $g \in \mathrm{U}(W)$ the automorphism $\omega_{\mu,\psi}(g)$ is an intertwining operator between $\rho_{\psi}$ and $\rho_{\psi}^g$ and moreover the action of $J$ by conjugation on $\psi_U$ is trivial, so that we can define the representation $\nu_{\mu,\psi}$ of $J \ltimes U$ by the rule 
\begin{equation*}
    \nu_{\mu,\psi}(g_Whu)=\psi_U(u) \omega_{\mu,\psi}(g_W) \rho_{\psi}(h), \; u \in U, \; g_W \in \mathrm{U}(W), \; h \in \mathbb{H}(\mathbb{W}).
\end{equation*}
Let $\overline{\nu}_{\mu,\psi}$ be its complex conjugation which is defined by the same formula with $\overline{\psi}_U$, $\overline{\omega}_{\mu,\psi}$ and $\overline{\rho}_{\psi}$.

\subsubsection{Mixed models}
\label{subsubsec:mixed_Schr_latt}

In the split case or the inert even case, we have a polarization $\mathbb{W}=Y_-(E) \oplus Y_-^*(E)$, so that $\nu_{\mu,\psi}$ is realized on $C_c^{\infty}(Y_-^*(E))=C_c^{\infty}(Y_+^*(E))$ by the Schrodinger model (\cite[Section~II.6]{MVW}). Otherwise, set 
\begin{equation}
\label{eq:C_c_psi}
    C_c^\infty(Y_{m_+}(E),\psi):=\{ f \in C_c^\infty(Y_{m_+}(E)) \; | \; f(y+y')=\psi(\frac{1}{2} \langle y,y' \rangle_{\mathbb{V}}) f(y), \; y \in Y_{m_+}(E), \; y' \in Y_{m_+}(\oo_E)\}.
\end{equation}
We realize $\nu_{\mu,\psi}$ on $C_c^{\infty}(Y_-^*(E)) \otimes C_c^\infty(Y_{m_+}(E),\psi) \subset C_c^\infty(Y_+^*(E))$ by mixing a Schrodinger and a lattice model (\cite[Sections~II.7 and II.8]{MVW}). Then $\nu_{\mu,\psi}$ is unitary for the inner product $(\phi,\phi')_{\nu}= \int_{Y_+^*(E)} \phi(y^*) \overline{\phi'(y^*)} dy^*$. This identifies $\nu_{\mu,\psi}^\vee$ with $\overline{\nu}_{\mu,\psi}$. There is a unique $\phi^\circ \in \nu_{\mu,\psi}^{K_J}$ with $( \phi^\circ,\phi^\circ)_\nu=1$, namely $\phi^\circ:=1_{Y_+^*(\oo_E)}$. 

Let $P(Y)$ be the parabolic subgroup of $\mathrm{U}(W)$ stabilizing the flag $\{0\} \subset Y_- \subset Y_+ \subset W$. Let $M(Y)$ be its Levi subgroup stabilizing $Y_{m_+}$ and $Y_-^*$. Then $M(Y)$ is isomorphic to $\GL(Y_-)$ in the inert even case and the split case, and $\GL(Y_-) \times \mathrm{U}(1)$ in the inert odd case, where $\GL(Y_-)=\GL_{m_-}(E)$ in the inert case and $\GL(Y_-)=\GL_m(F)$ in the split case. In the mixed model, the representation $\nu_{\mu,\psi}$ admits an explicit description as a $P(Y) \ltimes \mathbb{H}(\mathbb{W})$ module (\cite[Section~7.4]{GI}). It implies in particular that
\begin{equation}
\label{eq:nu_induite_schr}
    (\nu_{\mu,\psi})_{| B_W \ltimes \mathbb{H}(\mathbb{W})} \simeq \Ind_{B_J}^{B_W \ltimes \mathbb{H}(\mathbb{W})} (\Val{\cdot}^{\frac{1}{2}} \mu \psi) \simeq \ind_{B_J}^{B_W \ltimes \mathbb{H}(\mathbb{W})} (\Val{\cdot}^{\frac{1}{2}}  \mu \psi),
\end{equation}
where we write $\mu$ for $\mu \circ \det$. For $\phi \in \nu_{\mu,\psi}$, $a \in \GL(Y_-)$, $y \in Y_-(E)$, $y^* \in Y_-^*(E)$ and $y^*_0 \in Y_+^*(E)$ we have
\begin{align}
    \nu_{\mu,\psi}(a)\phi(y^*_0)&=\mu(\det a) \Val{\det a}^{\frac{1}{2}} \phi(a^* y^*_0), \label{eq:a_action} \\
    \nu_{\mu,\psi}(h(y+y^*,0))\phi(y_0^*)&=\psi \left(\langle y_0^*,y \rangle_{\mathbb{W}}+\frac{1}{2}\langle y^*,y \rangle_{\mathbb{W}} \right)\phi(y^*+y^*_0) \label{eq:rho_action}.
\end{align}
and in the inert odd case, for $g \in \mathrm{U}(1)$, $y^* \in Y_-^*(E)$ and $y_+^* \in E w_{m_+}$
\begin{equation}
\label{eq:mixed_model}
    \nu_{\mu,\psi}(g)\phi(y^*+y_+^*)=\phi(y^*+g^{-1} y_+^*).
\end{equation}

\subsubsection{Lattice models}
\label{subsubsec:lattice}
There exists another model of $\nu_{\mu,\psi}$ called the "lattice" model (see~\cite[Section~1.2.3]{GKT}). Set $\mathbb{H}(\mathbb{W}^0)=\mathbb{H}(W^0,F)$. Then $\nu_{\mu,\psi}$ can be realized on the induction $\Ind_{\mathbb{H}(\mathbb{W}^0)}^{\mathbb{H}(\mathbb{W})}( \psi)$. In this model, $\mathbb{H}(\mathbb{W})$ acts by right translations and the subgroup $K_W$ acts by left translation : for $\phi \in \ind_{\mathbb{H}(\mathbb{W}^0)}^{\mathbb{H}(\mathbb{W})} (\psi)$, $k_W \in K_W$, $w \in W$ and $z \in Z$ we have
\begin{equation}
    \label{eq:nu_lattice}
    \nu_{\mu,\psi}(k_W) \phi(h(w,z))=\phi(h(k_W^{-1}w,z)),
\end{equation}
which by \eqref{eq:Jacobi_relation} amounts to 
\begin{equation}
\label{eq:nu_induite_lattice}
    (\nu_{\mu,\psi})_{|K_W \ltimes H(W)} \simeq \Ind_{K_W \ltimes \mathbb{H}(\mathbb{W}^0)}^{K_W \ltimes \mathbb{H}(\mathbb{W})} (\psi) \simeq \ind_{K_W \ltimes \mathbb{H}(\mathbb{W}^0)}^{K_W \ltimes \mathbb{H}(\mathbb{W})} (\psi). 
\end{equation}
The space $(\nu_{\mu,\psi})^{K_J}$ has dimension one and is spanned by $\phi^\circ_{\mathrm{latt}} : h(w,z) \mapsto \psi(z) 1_{W(\oo_E)}(w)$. 

There is an unitary isomorphism $F_{W^0,Y}$ between the two models described in \cite[Proposition~1.23]{GKT}. It sends $\phi^\circ_{\mathrm{latt}}$ to $\phi^\circ=1_{Y_+^*(\oo_E)}$ and for $\phi \in \ind_{\mathbb{H}(\mathbb{W}^0)}^{\mathbb{H}(\mathbb{W})} (\psi) $ and $y^* \in Y_+^*(E)$ we have
\begin{equation}
\label{eq:intertwining_nu}
    F_{W^0,Y}(\phi)(y_+^*)=\int_{Y_-(E)} \phi(h(y,0)h(y_+^*,0)) dy.
\end{equation}

\subsubsection{Spherical vectors}
\label{subsubsec:spheri_vectors}

Set $\phi^1=1_{Y_-^*(\varpi \oo_E) \otimes Y_{m_+}(\oo_E)}$ and $\phi^\times=\Delta_{T_W}'1_{Y_-^*(\oo_E^\times) \otimes Y_{m_+}(\oo_E)}$ in the mixed model of $\nu_{\mu,\psi}$. By \eqref{eq:a_action} and \eqref{eq:rho_action} we have
\begin{equation}
\label{eq:phi_times}
    \phi^\times=q^{m_-}\nu_{\mu,\psi}(1_{T_W^0}) \nu_{\mu,\psi}(\mu_+^*) \phi^1,
\end{equation}
where we recall that $\vol(T_W^0)=1$. Note that in the split case $\Delta_{T_W}=\Delta_{T_W}'$ and $m=m_-$.

\begin{lem}
\label{lem:spheri_vectors}
    In the inert case, $\nu_{\mu,\psi}^{I_W \cap {}^{w_0} I_W}=\cc \phi^\circ$. In the split case, $\phi^\times \in \nu_{\mu,\psi}^{I_W \cap {}^{w_0} I_W}$. 
\end{lem}

\begin{proof}
    We first treat the inert case with $\mathrm{U}(W)=\mathrm{U}(1)$. Let $\phi \in \nu_{\mu,\psi}^{\mathrm{U}(1)}$ which we identify with a subspace of $C_c^\infty(E,\psi)$ by \eqref{eq:C_c_psi}. Let $y=u \varpi^{d} \in E$ with $u \in \oo_E^\times$ and $d<0$. Let $s_2 \in \oo_F^\times$. By Hensel's Lemma and because $q_F$ is odd, there exists $s_1 \in \oo_F$ such that $s:=\frac{ \varpi^{d}}{(\varpi^{d}+s_1+\tau s_2)}  \in \mathrm{U}(1)$. By \eqref{eq:nu_lattice} we have
    \begin{equation*}
        \phi(y)=\nu_{\mu,\psi}(s)\phi(y)=\psi(N_{E/F}(u)\tau^2 s_2 \varpi^{d}) \phi(y).
    \end{equation*}
    As $s_2$ is arbitrary this implies that $\phi(y)=0$ and therefore that $\phi \in \cc \; \phi^\circ$.

    To conclude the proof in the inert case, by \eqref{eq:mixed_model} it remains to show that $\nu_{\mu,\psi}^{I_W \cap {}^{w_0} I_W}=\cc \phi^\circ$ in the even case. In \cite[Section~7.1]{GI}, a subgroup $\mathrm{Herm}(Y_-^*,Y_-)\subset P(Y)$ is defined which has an integral structure over $\oo_F$. We see using \cite[Section~7.4]{GI} that any $\phi \in \nu_{\mu,\psi}^{{}^{w_0} I_W}$ is in fact fixed by $\mathrm{Herm}(Y_-^*,Y_-)(\oo_F)$ and thus supported in $Y_-^*(\oo_E)$. By \cite[Section~7.4]{GI}, $w_0$ acts by the Fourier transform $\hat{\cdot}$ on $\nu_{\mu,\psi}$. Thus, $\nu_{\mu,\psi}^{I_W \cap {}^{w_0} I_W}$ is contained in the space of functions $\phi$ such that $\phi$ and $\hat{\phi}$ are supported in $Y_-^*(\oo_E)$. They are multiple of $\phi^\circ.$

    In the split case, note that $\phi^1=\mu(\varpi) q_F^{m/2} \nu_{\mu,\psi}(t_W(\varpi,\hdots,\varpi)^{-1}) \phi^\circ$, so that $\phi^1 \in \nu_{\mu,\psi}^{K_W}$. Moreover, for any $g \in I_W \cap {}^{w_0} I_W$ we have ${}^g \mu_+^* \in T_W^0 \mu_+^* T_W^0$. The result follows from \eqref{eq:phi_times}.
\end{proof}

\subsection{Spherical vectors of \texorpdfstring{$\ind_H^G (\nu_{\mu,\psi})$}{the induction}}
\label{subsec:spheri_induction}

In this subsection, we always assume that $r \geq 1$. We show that the space of $K$-spherical vectors of the compact induction $\ind_H^G (\nu_{\mu,\psi})$ is a free $\mathcal{H}(G,K)$-module of rank $1$ (Proposition~\ref{prop:module_intro}). We explain in \S\ref{subsec:r=0} how to adapt this result to the $r=0$ case.

\subsubsection{}
\label{subsubsec:homo}
Let $X$ be the homogeneous space $X=H \backslash G$. Set $\Lambda_X=\Lambda_r + \Lambda_W^+$ and $\Lambda_X^-=\Lambda_r^+ + \Lambda_W^+$, which we identify with subsets of $\Lambda_V \subset G$. By the Iwasawa decomposition on $\mathrm{U}(V)$ and the Cartan decomposition on $\mathrm{U}(W)$, we see that $\Lambda_X$ is a complete set of representatives of $H \backslash G / K$. Recall that $\Lambda=\Lambda_V \times \Lambda_W$ and  $\Lambda^-=\Lambda_V^+ \times \Lambda_W^-$. There is a projection $\lambda=(\lambda_V,\lambda_W) \in \Lambda^- \mapsto \lambda_X:=\lambda_V-\lambda_W \in \Lambda_X^-$. Note that $\lambda_X=\lambda_X'$ is equivalent to $H \lambda K=H \lambda' K$. For $\lambda_V, \lambda_V' \in \Lambda_V$, we write $\lambda_V \leq_V \lambda_V'$ if $\lambda_V'-\lambda_V = \sum_{\alpha \in \Delta_V} c_\alpha \alpha^\vee$ for some $c_\alpha \in \zz_{\geq 0}$. For $\lambda_X, \lambda_X' \in \Lambda_X^-$, we write $\lambda_X \leq_X \lambda_X'$ if $\lambda_X \leq_V \lambda_X'$. Note that here there is a slight inconsistency of notation as any $\lambda_X \in \Lambda_X^-$ satisfies $\lambda_X \geq_X 0$. The following lemma holds for any quasi-split reductive group.

\begin{lem}
\label{lem:sat}
    Let $\lambda_V, \lambda_V' \in \Lambda_V^+$. Then $K_V\lambda_V K_V \lambda_V^{-1} \cap K_V (\lambda_V')^{-1} K_V \lambda_V' \subset K_V$.
\end{lem}

\begin{proof}
    Denote by $*$ the convolution on $\mathcal{H}(\mathrm{U}(V),K_V)$. By~\cite{Gross}, $1_{K_V \lambda_V' K_V} * 1_{K_V \lambda_V K_V}(\lambda_V' \lambda_V)=1$ and thus $\vol(K_V \lambda_V' K_V \cap \lambda_V' \lambda_V K_V \lambda_V^{-1} K_V)=1$. As $\lambda_V' K_V \subset K_V \lambda_V' K_V \cap \lambda_V' \lambda_V K_V \lambda_V^{-1} K_V$, this implies that
    \begin{equation*}
        \lambda_V' K_V = K_V \lambda_V' K_V \cap \lambda_V' \lambda_V K_V \lambda_V^{-1} K_V.
    \end{equation*}
    This directly implies the desired inclusion.
\end{proof}

We now describe the combinatorics between the $K \times K$ and $H \times K$ orbits in $G$. We write $H=\mathrm{U}(W) \ltimes N$, where $N$ is the unipotent radical of $P$ the parabolic subgroup of $\mathrm{U}(V)$ stabilizing the flag $0 \subset E v_1 \subset \hdots \subset X$.

\begin{lem}
\label{lem:orbits_in_G}
    Let $\lambda, \lambda' \in \Lambda^-$. If $ K\lambda K \cap H\lambda'K \neq \emptyset$, then $\lambda'_X \leq_X \lambda_X$. Moreover, $H \cap K \lambda K \lambda^{-1} \subset K$.
\end{lem}

\begin{proof} By hypothesis, $K_W \lambda_W K_W (\lambda_W')^{-1} \cap K_V \lambda_V K_V (\lambda_V')^{-1} N \neq \emptyset$, so there exists $n \in N$ such that
\begin{equation*}
    {}^{(\lambda_W')^{-1}}n (\lambda_W')^{-1} \lambda_V'  \in K_W \lambda_W^{-1} K_V \lambda_V K_V.
\end{equation*}
Choose $\lambda_r \in \Lambda_r^+$ with $\lambda_r -\lambda_W\in \Lambda_V^+$ and ${}^{\lambda_r (\lambda_W')^{-1}}n \in K_V$. As $\lambda_r$ commutes with $\mathrm{U}(W)$, we get $\lambda_r (\lambda_W')^{-1} \lambda_V'  \in K_V \lambda_r \lambda_W^{-1} K_V \lambda_V K_V$. By~\cite[Equation~(2.9)]{Gross}, we have
\begin{equation*}
    K_V \lambda_r \lambda_W^{-1} K_V \lambda_V K_V \subset \bigsqcup_{\substack{\tilde{\lambda}_V \in \Lambda_V^+ \\ \tilde{\lambda}_V \leq_V \lambda_r-\lambda_W+\lambda_V}} K_V \tilde{\lambda}_V K_V,
\end{equation*}
so that $\lambda_r-\lambda_W'+\lambda_V' \leq_V \lambda_r-\lambda_W+\lambda_V$. This is exactly $\lambda'_X \leq_X \lambda_X$.

Assume now that $g_W \in \mathrm{U}(W)$ and $n \in N$ are such that $(g_W,g_W n) \lambda \in K \lambda K$. Choose $\lambda_r$ as before. Then
\begin{equation*}
    g_W n \in K_V \lambda_V K_V \lambda_V^{-1} \cap K_W \lambda_W K_W \lambda_W^{-1} n \subset K_V \lambda_V K_V \lambda_V^{-1} \cap K_V \lambda_r^{-1} \lambda_W K_V \lambda_r \lambda_W^{-1}.
\end{equation*}
Lemma~\ref{lem:sat} yields $g_W n \in K_V$, which concludes the proof.
\end{proof}

\subsubsection{Proof of Proposition~\ref{prop:module_intro}}

We end the proof of Proposition~\ref{prop:module_intro}, restated in Proposition~\ref{prop:rank1}.

\begin{lem}
\label{lem:decompo}
    We have a decomposition as a direct sum
    \begin{equation}
    \label{eq:B^0_W_invariant}
        \nu_{\mu,\psi}^{B_W^0} = \bigoplus_{\substack{\lambda_W \in \Lambda_W^+}} \cc \; \nu_{\mu,\psi}(\lambda_W^{-1}) \phi^\circ.
    \end{equation}
   Moreover, let $\tilde{\lambda}_X \in \Lambda_X^-$. Then 
    \begin{equation}
          \label{eq:decomp_space}
        \nu_{\mu,\psi}^{ {}^{\tilde{\lambda}_X} K_W \cap K_W \ltimes {}^{\tilde{\lambda}_X}\mathbb{H}(\mathbb{W})^0} = \bigoplus_{\substack{\lambda \in \Lambda^{-} \\ \lambda_X=\tilde{\lambda}_X}} \cc \; \nu_{\mu,\psi}(\lambda_W^{-1}) \phi^\circ.
    \end{equation}
\end{lem}

\begin{proof}
    We identify $\nu_{\mu,\psi}$ with its mixed model which is included in the space $C_c^\infty(Y_+^*(E))$. By Lemma~\ref{lem:spheri_vectors}, every $\phi \in \nu_{\mu,\psi}^{B_W^0}$ is of the form $\phi_- \otimes 1_{Y_{m_+}^*(\oo_E)}$ for $\phi_- \in C_c^\infty(Y_-^*(E))$. By \eqref{eq:a_action}, for any $\lambda_W \in \Lambda_W$ we know that $\nu_{\mu,\psi}(\lambda_W^{-1}) \phi^\circ$ is, up to a scalar, the function $1_{\lambda_W Y_+^*(\oo_E)}$. By doing induction on any finite subset of $\Lambda_W$ equipped with the lexicographic order, we see that the family $(\nu_{\mu,\psi}(\lambda_W^{-1}) \phi^\circ)_{\lambda_W \in \Lambda_W}$ is a basis of $\nu_{\mu,\psi}^{T_W^0}$ so that the RHS of \eqref{eq:B^0_W_invariant} is indeed a direct sum. 
    
    We prove \eqref{eq:B^0_W_invariant}. Let $\phi \in \nu_{\mu,\psi}^{B_W^0}$ be non-zero. By the previous discussion, there exist a finite set $\Lambda_\phi \subset \Lambda_W$ and constants $c_{\lambda_W} \in \cc^\times$ for $\lambda_W \in \Lambda_\phi$ such that $\phi=\sum_{\lambda_W \in \Lambda_\phi} c_{\lambda_W} \nu_{\mu,\psi}(\lambda_W^{-1}) \phi^\circ$. Let $\beta \in \Delta_W$, and set $l:= \min\{ \langle \lambda_W,\beta \rangle \; | \; \lambda_W \in \Lambda_\phi \}$. We have to show that $l \geq 0$. Assume by contradiction that $l<0$. Write
    \begin{equation*}
        \phi_l:= \sum_{\substack{\lambda \in \Lambda_\phi \\ \langle \lambda_W, \beta \rangle =l}} c_{\lambda_W} \nu_{\mu,\psi}(\lambda_W^{-1}) \phi^\circ, \;\;\;  \phi_{>l}:= \sum_{\substack{\lambda \in \Lambda_\phi \\ \langle \lambda_W, \beta \rangle >l}} c_{\lambda_W} \nu_{\mu,\psi}(\lambda_W^{-1}) \phi^\circ.
    \end{equation*}
    
    Recall that $G_\beta$ is the Levi subgroup of the parabolic subgroup $P_\beta$ of $\mathrm{U}(W)$ corresponding to $\beta$. Set $K_\beta=K_W \cap G_\beta \subset H$. For $\lambda_W \in \Lambda_W$, the subgroup $\lambda_W^{-1} K_\beta \lambda_W$ only depends on $\langle \lambda_W, \beta \rangle$. If $\langle \lambda_W, \beta \rangle=k$, write $K_\beta^k=\lambda_W^{-1} K_\beta \lambda_W$. Set $\Lambda_{W,\beta}=\Lambda_W \cap G_\beta$. The Cartan decomposition gives $G_\beta=K_\beta^l \Lambda_{W,\beta} K_\beta^l$. Note that $\phi_l$ is stabilized by $K_\beta^l$. Moreover, as the family $(\nu_{\mu,\psi}(\lambda_W^{-1}) \phi^\circ)_{\lambda_W \in \Lambda_W}$ is free the stabilizer of $\phi_l$ in $\Lambda_{W,\beta}$ is trivial. Therefore, the stabilizer of $\phi_l$ in $G_\beta$ is exactly $K_\beta^l$.

    For any $k$, set $B_\beta^k=K_\beta^k \cap B_W$ so that $B_\beta^{k+1} \subsetneqq B_\beta^{k}$. For any $\lambda_W \in \Lambda_W$ with $\langle \lambda_W,\beta \rangle >l$, $\nu_{\mu,\psi}(\lambda_W^{-1})(\phi^\circ)$ is stable by $B_\beta^{l+1}$. It follows that $\phi_{>l}$ is fixed by $B_\beta^{l+1}$, but by the previous discussion $\phi_{l}$ is not as its stabilizer in $B_W \cap G_\beta$ is $B_\beta^l$. This is a contradiction as $B_\beta^{l+1} \subset B_W^0$ because $l<0$. This shows \eqref{eq:B^0_W_invariant}.

    We now show \eqref{eq:decomp_space}. The inclusion $\supset$ is automatic so that we prove the reverse. Write $\tilde{\lambda}_X=\tilde{\lambda}_r + \tilde{\lambda}_W$ with $\tilde{\lambda}_r \in \Lambda_r^+$ and $\tilde{\lambda}_W \in \Lambda_W^+$. As ${}^{\tilde{\lambda}_X} B_W^0 \subset {}^{\tilde{\lambda}_X} K_W \cap K_W$ and ${}^{w_0} B^0_W \subset {}^{\tilde{\lambda}_X} K_W \cap K_W$ (where $w_0$ is the longest element in $W_W$), we see by \eqref{eq:B^0_W_invariant} that
    \begin{equation}
        \label{eq:decomp_partial}
        \nu_{\mu,\psi}^{{}^{\tilde{\lambda}_X} K_W \cap K_W \ltimes {}^{\tilde{\lambda}_X}\mathbb{H}(\mathbb{W})^0} \subset \bigoplus_{\substack{\lambda_W \in \Lambda_W^- \\ \tilde{\lambda}_W+\lambda_W \in \Lambda_W^+}} \cc \nu_{\mu,\psi}(\lambda_W^{-1}) \phi^\circ.
    \end{equation}
    To conclude, we show that any $\lambda_W$ appearing in the RHS of \eqref{eq:decomp_partial} satisfies $\tilde{\lambda}_X+\lambda_W \in \Lambda_V^+$, i.e. $\langle \tilde{\lambda}_X +\lambda_W,\alpha \rangle \geq 0$ for $\alpha=e_r-e_{r+1} \in \Delta_V$ (or $\alpha=e_{r'}-e_{r'+1} \in \Delta_V$ in the split case). This is the same argument as in the proof of \eqref{eq:B^0_W_invariant}, noting that by \eqref{eq:rho_action} the stabilizer of $\phi^\circ$ in $N_V \cap G_\alpha \subset H$ is $N_V^0 \cap G_\alpha$.
\end{proof}

\begin{prop}
\label{prop:c_compact_decompo}
    We have a decomposition
    \begin{equation*}
       \left( \ind_H^G (\nu_{\mu,\psi}) \right)^K= \bigoplus_{\lambda \in \Lambda^-} \cc \; \Phi_{\lambda},
    \end{equation*}
    where $\Phi_{\lambda}$ is the unique vector in $(\ind_H^G (\nu_{\mu,\psi}))^K$ such that $\supp(\Phi_\lambda) \subset H \lambda K$ and $\Phi_\lambda(\lambda)=\phi^\circ$.
\end{prop}

\begin{proof}
    As $\Lambda_X$ is a complete system of representatives of $H \backslash G /K$, by evaluating we have
    \begin{equation*}
        \left( \ind_H^G  (\nu_{\mu,\psi} )\right)^K=\bigoplus_{\tilde{\lambda}_X \in \Lambda_X} \nu_{\mu,\psi}^{{}^{\tilde{\lambda}_X} K_W \cap K_W \ltimes {}^{\tilde{\lambda}_X}(\mathbb{H}(\mathbb{W})^0U^0)}.
    \end{equation*}
    As $\psi$ is unramified, $\nu_{\mu,\psi}^{{}^{\tilde{\lambda}_X} U^0}$ is zero unless ${}^{\tilde{\lambda}_X} U^0 \subset U^0$ which implies ${\tilde{\lambda}_X} \in \Lambda_X^-$. We conclude from Lemma~\ref{lem:decompo} by letting $\Phi_\lambda$ be the unique vector in $(\ind_H^G (\nu_{\mu,\psi}))^K$ such that $\supp(\Phi_\lambda) \subset H \lambda K$ and  $\Phi_\lambda(\lambda_X)=\nu_{\mu,\psi}(\lambda_W^{-1}) \phi^\circ$.
\end{proof}

\begin{prop}
\label{prop:rank1}
    The $\mathcal{H}(G,K)$-module $(\ind_H^G (\nu_{\mu,\psi}))^K$ is free of rank $1$ generated by $\Phi_0$.
\end{prop}
\begin{proof}
    Let $\lambda \in \Lambda^-$. By Lemma~\ref{lem:orbits_in_G} and Proposition~\ref{prop:c_compact_decompo}, for every $\lambda' \in \Lambda^-$ there exists $c(\lambda,\lambda') \in \cc$ with
\begin{equation*}
    R(1_{K \lambda^{-1} K}) \Phi_0 = \sum_{\substack{\lambda' \in \Lambda^- \\ \lambda'_X \leq_X \lambda_X}} c(\lambda,\lambda') \Phi_{\lambda'},
\end{equation*}
and the $c(\lambda,\lambda')$ are almost all zero. By the second part of Lemma~\ref{lem:orbits_in_G}, we have $R(1_{K \lambda^{-1} K}) \Phi_0(\lambda)=\phi^\circ$, so that by Proposition~\ref{prop:c_compact_decompo} we see that
\begin{equation}
\label{eq:R_translation}
    R(1_{K \lambda^{-1} K}) \Phi_0 = \Phi_\lambda+\sum_{\substack{\lambda' \in \Lambda^- \\ \lambda'_X <_X \lambda_X}} c(\lambda,\lambda') \Phi_{\lambda'}.
\end{equation}
As $r\geq 1$, the map $\Lambda^- \to \Lambda_X^-$ has finite fibers, and for every $\lambda_X \in \Lambda_X^-$ the set of $\lambda_X' \in \Lambda_X^-$ such that $\lambda_X' \leq_X \lambda_X$ is finite. Moreover, by the Cartan decomposition, the family $(1_{K \lambda^{-1} K})_{\lambda \in \Lambda^-}$ is a basis of $\mathcal{H}(G,K)$. We conclude by induction on~\eqref{eq:R_translation} and by Proposition~\ref{prop:c_compact_decompo} that the $\mathcal{H}(G,K)$-module $ (\ind_H^G (\nu_{\mu,\psi}))^K$ is generated by the torsion-free element $\Phi_0$, so that it is free of rank $1$. 
\end{proof}

\subsection{Proof of Theorem~\ref{thm:mult1intro}}
\label{subsection:proof_mult1}

We prove Theorem~\ref{thm:mult1intro} under the assumption that $r \geq 1$ (see \S\ref{subsec:r=0} for the $r=0$ case). 

\begin{lem}
\label{lem:induced_reps}
    Let $\eta$ be an unramified character of $T_W$. There is an isomorphism of $H$-representations (which also holds for $r=0$)
    \begin{equation}
        \label{eq:S_eta2}
        S_{\overline{\mu} \eta} : \left(I_{B_W}^{\mathrm{U}(W)} (\eta )\right) \otimes \overline{\nu}_{\mu,\psi} \to \overline{\psi}_U \boxtimes \left(I_{B_J}^{J} (\overline{\mu} \eta \overline{\psi}) \right),
    \end{equation}
    which satisfies $S_{\overline{\mu}\eta}(\Phi^{\circ}_{\eta } \otimes \phi^\circ)=\Phi^{\circ,J}_{\overline{\mu}\eta }$ and $S_{\overline{\mu}\eta}(\Psi_{w,\eta}\otimes \phi^\circ)=\Psi^J_{w, \overline{\mu}\eta}$.
\end{lem}

\begin{proof}
    The existence of $S_{\overline{\mu} \eta}$ follows from \eqref{eq:mod_char}, \eqref{eq:nu_induite_schr} and the isomorphisms of $J$-representations 
    \begin{equation*}
        \left(I_{B_W}^{\mathrm{U}(W)} (\eta) \right) \otimes \overline{\nu}_{\mu,\psi} \simeq \Ind_{B_W \ltimes \mathbb{H}(\mathbb{W}) }^{J} \left(\delta_{B_J}^{\frac{1}{2}}\eta \otimes I_{B_J}^{B_W \ltimes \mathbb{H}(\mathbb{W})} \overline{\mu} \overline{\psi} \right)  \simeq I_{B_J}^J (\overline{\mu} \eta \overline{\psi}).
    \end{equation*}
    Explicitly, $S_{\overline{\mu} \eta}(f \otimes \phi)(h g_W)=(\overline{\nu}_{\mu,\psi}(h g_W)\phi)(0) f(g_W)$ for $h \in \mathbb{H}(\mathbb{W})$ and $g_W \in \mathrm{U}(W)$. This concludes.
\end{proof}

\begin{prop}
\label{prop:isos}
    There exists a surjective morphism of $\cc$-algebras
    \begin{equation}
    \label{eq:power_psi}
        f_J \in \mathcal{H}(J,K_J) \mapsto f_J^{\psi} \in \mathcal{H}(\mathrm{U}(W),K_W),
    \end{equation}
    such that for all $f_J \in \mathcal{H}(J,K_J)$ we have $\xi_{\overline{\mu} \eta \overline{\psi}}^J(f_J)=\xi_{\eta }(f_J^\psi)$. Moreover, there exists an isomorphism of $\mathcal{H}(\mathrm{U}(V),K_V) \otimes \mathcal{H}(J,K_J)$-modules
    \begin{equation}
    \label{eq:WS_iso}
        \mathcal{W} \in \mathrm{WS} \mapsto \Phi_\mathcal{W} \in \left( \Ind_H^G (\nu_{\mu,\psi} )\right)^K.
    \end{equation}
\end{prop}

\begin{rem}
\label{rem:iso_explicit}
    Let $\mathcal{L}_H \in \Hom_H(I_{B_V}^{\mathrm{U}(V)}( \chi) \otimes I_{B_W}^{\mathrm{U}(W)} (\eta) \otimes \overline{\nu}_{\mu,\psi},\cc)$, and let $\mathcal{L}_H^\vee \in \Hom_H(I_{B_V}^{\mathrm{U}(V)} (\chi )\otimes I_{B_W}^{\mathrm{U}(W)} (\eta), \nu_{\mu,\psi})$ be the associated operator by duality. If $\mathcal{W}_{\chi,\overline{\mu}\eta}$ is of the form $\mathcal{W}_{\chi,\overline{\mu}\eta}(g)=\mathcal{L}_H(R(g) \Phi_\chi^\circ \otimes \Phi_\eta^\circ \otimes \phi^\circ)$ for $\Phi_\chi^\circ$ and $\Phi_\eta^\circ$ spherical, then we have
    \begin{equation}
    \label{eq:Phi_equation}
        \Phi_{\mathcal{W}_{\chi,\overline{\mu}\eta}}(g_V,g_W)=\mathcal{L}_H^\vee(R(g_V) \Phi_\chi^\circ \otimes R(g_W)\Phi_\eta^\circ), \quad g_V \in \mathrm{U}(V), \quad g_W \in \mathrm{U}(W). 
    \end{equation}
\end{rem}

\begin{proof}
    We first describe the map $\mathcal{W} \mapsto \Phi_\mathcal{W}$ and show that it is an isomorphism of $\cc$-vector spaces. Let $\mathcal{W} \in \mathrm{WS}$. For every $g_V \in \mathrm{U}(V)$, consider the function $\phi_{\mathcal{W},g_V} : h \in \mathbb{H}(\mathbb{W}) \mapsto \mathcal{W}(hg_V)$. Then $\phi_{\mathcal{W},g_V} \in \ind_{\mathbb{H}(\mathbb{W}^0)}^{\mathbb{H}(\mathbb{W})} (\psi)$, which we identify with the lattice model of $\nu_{\mu,\psi}$ by \S\ref{subsubsec:lattice}. We now define
    \begin{equation*}
        \Phi_\mathcal{W} : (g_V,g_W) \in G \mapsto \nu_{\mu,\psi}(g_W) \phi_{\mathcal{W},g_W^{-1} g_V} \in \nu_{\mu,\psi}.
    \end{equation*}
    This yields the desired morphism $\mathrm{WS} \to (\Ind_H^G (\nu_{\mu,\psi}))^K$. For $\Phi \in ( \Ind_H^G (\nu_{\mu,\psi}))^K$, define
    \begin{equation}
    \label{eq:iso_explicit}
        \mathcal{W}_\Phi : g_V \mapsto \Phi(g_V)(0)=( \Phi(g_V),\phi^\circ_{\mathrm{latt}} )_{\mathrm{latt}} \in \cc,
    \end{equation}
    where we identify $\Phi(g_V)$ with an element of $\Ind_{\mathbb{H}(\mathbb{W}^0)}^{\mathbb{H}(\mathbb{W})}( \psi)$. Then $\mathcal{W}_{\Phi} \in \mathrm{WS}$. Using \eqref{eq:nu_lattice}, it is easy to check that $\mathcal{W}_\Phi \in \mathrm{WS}$ and that $\mathcal{W} \mapsto \Phi_\mathcal{W}$ and $\Phi \mapsto \mathcal{W}_\Phi$ are inverse of each other.

    We now describe the morphism~\eqref{eq:power_psi}. Let $C_c^\infty((K_J Z, \psi) \backslash J)$ be the compact induction $\ind_{K_J Z}^J (\psi)$. Define $\mathcal{H}(J,K_J Z,\psi)$ to be the Hecke algebra of functions $J$ which transform by $\psi$ under left and right translations by $K_J Z$, and which have compact support modulo the action of $K_J Z$. Note that $K_J Z$ is an open subgroup of the locally compact totally disconnected group $J$ which is compact modulo its center. We therefore have an isomorphism of $\cc$-algebras
    \begin{equation}
    \label{eq:end1}
        \mathrm{End}_J\left(C_c^\infty((K_J Z, \psi) \backslash J) \right) \simeq \mathcal{H}(J,K_J Z, \psi),
    \end{equation}
    where $\mathcal{H}(J,K_J Z, \psi)$ acts on $C_c^\infty((K_J Z, \psi) \backslash J)$ by left-translations. Using \eqref{eq:nu_induite_lattice} we see that 
    \begin{equation}
    \label{eq:iso_induite}
        \ind_{K_J Z}^J (\psi) \simeq \ind_{K_W \ltimes \mathbb{H}(\mathbb{W})}^J \left( \ind_{K_W \ltimes \mathbb{H}(\mathbb{W}^0)}^{K_W \ltimes \mathbb{H}(\mathbb{W})} (\psi) \right)  \simeq \ind_{K_W \ltimes \mathbb{H}(\mathbb{W})}^J(1) \otimes \nu_{\mu,\psi} \simeq C_c^\infty(K_W \backslash \mathrm{U}(W)) \otimes \nu_{\mu,\psi},
    \end{equation}
    where $J$ acts on $\mathrm{U}(W)$ by the composition of $R$ with the projection $J \to \mathrm{U}(W)$. As $\nu_{\mu,\psi}$ is irreducible (it is already as a $\mathbb{H}(\mathbb{W})$-representation), Schur's lemma implies that we have an isomorphism of $\cc$-algebras
    \begin{equation}
    \label{eq:end2}
        \mathrm{End}_J\left(C_c^\infty((K_J Z, \psi) \backslash J) \right) \simeq \mathrm{End}_{\mathrm{U}(W)}(C_c^\infty(K_W \backslash \mathrm{U}(W)))  \simeq \mathcal{H}(\mathrm{U}(W),K_W).
    \end{equation}
    Composing \eqref{eq:end2} with \eqref{eq:end1} and the surjective morphism $\mathcal{H}(K,K_J) \to \mathcal{H}(K,K_J Z, \psi)$ given by integrating against $\overline{\psi}$ on $Z$, we obtain the desired map \eqref{eq:power_psi}. 

    Let us show that $\mathcal{W} \mapsto \Phi_{\mathcal{W}}$ is an morphism of $\mathcal{H}(\mathrm{U}(V),K_V) \otimes \mathcal{H}(J,K_J)$-modules. For any $f_V \in \mathcal{H}(\mathrm{U}(V),K_V)$, we have $R(f_V)\Phi_\mathcal{W}=\Phi_{R(f_V)\mathcal{W}}$. Let $\mathcal{W} \in \mathrm{WS}$ and $f_J \in \mathcal{H}(J,K_J)$. It is enough to show that for any $g_W \in \mathrm{U}(W)$ we have $ R(f_J^\psi)\Phi_\mathcal{W}(1,g_W)=\Phi_{L(f_J)\mathcal{W}}(1,g_W)$. Note that $j \mapsto \mathcal{W}(j) \in \ind_{K_JZ}^J (\psi)$. We can assume that this element is sent to $f_W \otimes \phi \in C_c(K_W \backslash \mathrm{U}(W)) \otimes \nu_{\mu,\psi}$ by \eqref{eq:iso_induite}. By definition, we have $L(f_J)\mathcal{W}=L(f_J^\psi) f_W \otimes \phi, \text{   and   }  \Phi_\mathcal{W}(1,g_W)=f_W(g_W^{-1}) \phi$. This concludes that $R(f_J^\psi)\Phi_\mathcal{W}=\Phi_{L(f_J)\mathcal{W}}$ and therefore that \eqref{eq:WS_iso} is indeed an isomorphism of modules.
    
    Finally, let us prove that $\xi_{\overline{\mu} \eta \overline{\psi}}^J(f_J)=\xi_{\eta }(f_J^\psi)$ for every $f_J \in \mathcal{H}(J,K_J)$. We have
    \begin{equation}
    \label{eq:J_psi_defi}
        \int_Z f_J(z g_W h) \overline{\psi}(z)dz= f_J^\psi(g_W) \nu_{\mu,\psi}(g_W h) \phi^\circ_{\mathrm{latt}}(0), \quad g_W \in \mathrm{U}(W), \; h \in \mathbb{H}(\mathbb{W}).
    \end{equation}
    By Lemma~\ref{lem:induced_reps}, the isomorphism of $J$-representations $\left(I_{B_W}^{\mathrm{U}(W)} (\eta )\right) \otimes \overline{\nu}_{\mu,\psi} \to I_{B_J}^{J} (\overline{\mu} \eta \overline{\psi})$ sends $\Phi_{\eta}^\circ \otimes \phi^\circ$ to $\Phi_{\overline{\mu} \eta}^{\circ,J}$. Let $f_J \in \mathcal{H}(J,K_J)$. It follows from \eqref{eq:rho_action}, \eqref{eq:intertwining_nu} and \eqref{eq:J_psi_defi} by direct computations that
    \begin{equation*}
        \xi_{\overline{\mu} \eta \overline{\psi}}^J(f_J)=\int_{\mathrm{U}(W)} f_J^{\psi}(g_W) \Phi_{\eta}^\circ(g_W)  ( \overline{\nu}_{\mu,\psi}( g_W ) \phi^\circ, \overline{\nu}_{\mu,\psi}( g_W ) \phi^\circ )_{\overline{\nu}} dg_W.
    \end{equation*}
    As $\overline{\nu}_{\mu,\psi}$ is unitary, we see that this last expression reduces to $\xi_{\eta }(f_J^\psi)$, which concludes the proof. 
\end{proof}

We end the proof of Theorem~\ref{thm:mult1intro}. Let $\chi$ and $\eta$ be unramified characters of $T_V$ and $T_W$ respectively. The space $\mathrm{WS}_{\chi,\overline{\mu}\eta}$ is the $\xi_{\chi} \otimes \xi_{\overline{\mu}\eta \overline{\psi}}^J$-eigenspace in $\mathrm{WS}$ (see \S\ref{subsec:WS_functions}). By Proposition~\ref{prop:isos}, it is enough to prove than for any character $\xi$ of $\mathcal{H}(G,K)$ the $\xi$-eigenspace $(\Ind_H^G (\nu_{\mu,\psi}))^K_{\xi}$ has dimension at most one. As $\nu_{\mu,\psi}$ is unitary and admissible, its contragredient $\widetilde{\nu_{\mu,\psi}}$ is $\overline{\nu}_{\mu,\psi}$. By \cite[Th\'eor\`eme~III.2.7]{Ren} we have the isomorphism of $\mathcal{H}(G)$-modules $\Ind_H^G (\nu_{\mu,\psi}) \simeq (\ind_H^G  (\overline{\nu}_{\mu,\psi}) )\;\widetilde{}$. Let $L \in \left((\ind_H^G  (\overline{\nu}_{\mu,\psi}) )\;\widetilde{}\;\right)^K_\xi$. Then $L$ is determined by its values on $(\ind_H^G (\overline{\nu}_{\mu,\psi}) )^K$, and therefore by $L(\Phi_0)$ by Proposition~\ref{prop:rank1}. This shows that $(\Ind_H^G (\nu_{\mu,\psi}))^K_{\xi}$ has dimension at most one. In \S\ref{section:pairing} and \S\ref{section:Formula} we build a non-zero element in this space. This proves Theorem~\ref{thm:mult1intro}.

\subsection{The \texorpdfstring{$r=0$}{r=0} case}
\label{subsec:r=0}
We now assume that $r=0$ and explain how to adapt the proof of Theorem~\ref{thm:mult1intro}. Set $\Lambda_X^-=\Lambda_V^+$ embedded in the first factor of $G$. Define $\Delta_J=\Delta_V \cup \{-e_1\}$ in the inert case, and $\Delta_J=\Delta_V \cup \{-e_1,e_n\}$ in the split case. Set $\Lambda_J^- = \{ \lambda_V \in \Lambda_V \; | \; \forall \alpha \in \Delta_J, \; \langle \lambda_V,\alpha \rangle \geq 0 \}$. If $\lambda_V \in \Lambda_X^-$, the proof of Lemma~\ref{lem:decompo} shows that
\begin{equation*}
      \nu_{\mu,\psi}^{ {}^{\lambda_V} K_V \cap K_V \ltimes {}^{\lambda_V}\mathbb{H}(\mathbb{W})^0} = \bigoplus_{\substack{\lambda \in \Lambda_V^{-} \\ \lambda_V+\lambda \in \Lambda_J^-}} \cc \; \nu_{\mu,\psi}(\lambda^{-1}) \phi^\circ.
\end{equation*}
But $\Lambda_J^-=\{0\}$, so that $\nu_{\mu,\psi}^{ {}^{\lambda_V} K_V \cap K_V \ltimes {}^{\lambda_V}\mathbb{H}(\mathbb{W})^0}=\cc \; \overline{\nu}_{\mu,\psi}(\lambda_V) \phi^\circ$.

As all the other results of \S\ref{subsec:spheri_induction} remain valid, we prove as in Proposition~\ref{prop:rank1} that the $\mathcal{H}(\mathrm{U}(V),K_V) \otimes \mathcal{H}(J,K_J)$-module $(\ind_J^{\mathrm{U}(V) \times J} (\overline{\nu}_{\mu,\psi}) )^{K_V \times K_J}$ is free of rank $1$. An easy adaptation of Proposition~\ref{prop:isos} shows that we have an isomorphism $\mathrm{WS} \simeq (\Ind_J^{\mathrm{U}(V) \times J} (\nu_{\mu,\psi}))^{K_V \times K_J}$ with the same properties as in the $r \geq 1$ case. Theorem~\ref{thm:mult1intro} now follows from argument of \S\ref{subsection:proof_mult1}. Note that \eqref{eq:Phi_equation} still holds in this case.

\section{Integral expression and analytic continuation}

\label{section:pairing}

In this section we produce a non trivial element $\mathcal{W}^I_{\chi,\eta}$ in $\mathrm{WS}_{\chi,\eta}$, first for $(\chi,\eta)$ in a non-empty open subset of $(\cc^\times)^{n_-}\times(\cc^\times)^{m_-}$, and then for $(\chi,\eta)$ in general position by analytic continuation.

\subsection{Integral pairing}

\subsubsection{}
\label{subsec:integral_rep}

For $A$ and $B$ subsets of $\{1, \hdots ,n\}$ and $\{1, \hdots, m\}$ respectively with $|A|=|B|$ and $g \in \mathrm{U}(V)$, we define $\Delta_{A,B}(g)$ to be the determinant of the $A \times B$ minor of $g$. We will also use this notation if $g \in J$ by identifying it with a $(m+2)\times(m+2)$ matrix.

If $r \geq 1$, define for $1 \leq k \leq n_+$ the set $I_k=\{n+1-k, \hdots, n\}$, and for $1 \leq l \leq m_+$ the sets $J_l=\{1, \hdots, l\}$ and $J'_l=\{1, \hdots r-1, r+1, \hdots, r+l \}$. If $r=0$, we use rather $I_k=\{n+2-k, \hdots ,n+1\}$, $I'_l=\{1,n+3-l,\hdots,n+1\}$ and $J_l=\{2, \hdots ,1+l\}$. Set
\begin{equation}
\label{eq:alpha_beta_defi}
    \alpha_k(g)=\Delta_{I_k,J_k}(g) \text{ and } \beta_l(g)=\Delta_{I_{r+l-1},J'_l}(g) \text{ if } r\geq 1\; \left( \text{resp. } \beta_l(g)=\Delta_{I'_l,J_l}(g) \text{ if } r= 0 \right).
\end{equation}
We define $\tilde{\beta}_{m_+}$ to be $\beta_{m_+}$ in the inert odd case, and $0$ otherwise. It is easily checked that these applications enjoy the following properties for all $g$. Recall that $\mu_+^*=h(1_{Y_+^*},0)$ (see \S\ref{subsubsec:Y_defi}).

\begin{itemize}
    \item For $n_1, n_2 \in N_V$, $\alpha_k(n_1 g n_2)=\alpha_k(g)$ ($n_2 \in N_J$ if $r=0$),
    \item $\alpha_k(t_V(\mathbf{t})gt_V(\mathbf{s}))=\left\{ \begin{array}{ll}
       \prod_{i=1}^k {}^c t_i^{-1} s_i \cdot \alpha_k(g)  & \text{inert case}, \\
       \prod_{i=1}^k t_{n-i+1} s_i  \cdot \alpha_k(g)  & \text{split case}.
    \end{array} \right.$
    \item If $w \in W_V$, then $\alpha_k(w) \neq 0$ for all $k$ if and only if $w=w_0$.
    \item For any $n_V \in N_V$, $\beta_l(n_Vg)=\beta_l(g)$.
    \item For any $n_J \in N_J$, $u \in U$ and $l\leq m_-$, $\beta_l(gn_Ju)=\beta_l(g)$, and in the odd inert case, for $y \in Y_-(E)$, $z \in F$, $n_W \in N_W$, $\tilde{\beta}_{m_+}(gh(y,z)n_Wu)=\tilde{\beta}_{m_+}(g)$.
    \item $\beta_l(t_V(\mathbf{t})gt_W(\mathbf{s}))=\left\{ \begin{array}{ll}
       \prod_{i=1}^r {}^c t_i^{-1} \cdot \prod_{i=1}^{l-1} {}^c t_{r+i}^{-1} \cdot \prod_{i=1}^l s_i \cdot \beta_l(g)  & \text{inert case}, \\
       \prod_{i=1}^r t_{n-i+1} \cdot \prod_{i=1}^{l-1} t_{n-r-i+1} \cdot \prod_{i=1}^l s_i \cdot \beta_l(g)  & \text{split case}.
    \end{array} \right.$
    \item For $y^* \in Y^*_-(E)$ and $l \leq m_-$, $\beta_l(w_0 h(y^*,0))=\pm y^*_l$, and in the inert odd case, for $y_{m_+} \in E$ we have $\tilde{\beta}_{m_+}(w_0 h(y^*+y_{m_+} w_{m_+},0))=\pm y_{m_+}$.
\end{itemize}

\begin{prop}
\label{prop:doucle_class_description}
    We have the equality
    \begin{equation*}
        B_V w_0 \mu_+^* B_J U = \left\{ g \; \middle| \; \alpha_k(g) \neq 0, \; \beta_l(g) \neq 0, \; \Val{\frac{\tilde{\beta}_{m_+}(g)\prod_{i=1}^{m_-} \alpha_{r+i-1}(g)}{\alpha_{n_-}(g)\prod_{i=1}^{m_-}\beta_i(g) }} \leq 1, \begin{array}{l}
             1 \leq k \leq n_+, \\
             1 \leq l \leq m_-
        \end{array}\right\},
    \end{equation*}
    where $g \in \mathrm{U}(V)$ if $r \geq 1$ and $g \in J$ if $r=0$.
\end{prop}

\begin{proof}
    Assume $r \geq 1$ and that we are in the inert odd case. For $\mathbf{a}=(a_i) \in \{0,1\}^{m_-}$, write $y^*_{\mathbf{a}}=\sum a_i w_i^*$. Then the Bruhat decomposition on $\mathrm{U}(V)$ yields
    \begin{equation*}
         \mathrm{U}(V)=\bigcup_{w \in W_V} \bigcup_{\mathbf{a} \in \{0,1\}^{m_-}}\bigcup_{y_{m_+} \in E / \mathrm{U}(1)}  B_V w  h(y^*_\mathbf{a}+y_{m_+} w_{m_+},0)B_J U,
    \end{equation*}
   where $\mathrm{U}(1)$ acts on $E$ by multiplication. The result now follows from the above properties of the $\alpha_k$ and $\beta_l$. The proof is the same for $r=0$ and for the other cases.
\end{proof}

\subsubsection{}

Let $\chi$ and $\eta$ be unramified characters of $T_V$ and $T_W$. For $g \in \mathrm{U}(V)$ (or $g \in J$ for $r=0$) set
\begin{equation}
\label{eq:Y_defi}
    Y_{\chi,\eta}(g):=\chi^{-1} (\delta_{B_V}^{\frac{1}{2}})(b_V) (\eta \overline{\psi} \delta_{B_J}^{-\frac{1}{2}})(b_J) \overline{\psi}_\mathrm{U}(u) \quad \text{if } g=b_V w_0 \mu_+^* b_J u \in B_V w_0 \mu_+^* B_J U,
\end{equation}
and $Y_{\chi,\eta}(g)=0$ if $g \notin B_V w_0 \mu_+^* B_J U$. From section~\ref{subsec:integral_rep} we see that for $g=t_V(\mathbf{t})n_V w_0 \mu_+^* t_W(\mathbf{s}) n_J u$ we have, writing $\alpha_i$ (resp. $\beta_j$) for $\alpha_i(g)$ (resp. $\beta_j(g)$)
\begin{equation*}
    \Val{Y_{\chi,\eta}(g)}=\prod_{i=1}^{r-1} \Val{\chi_i \chi_{i+1}^{-1} \Val{.}_E^{-1}}(\alpha_i)  
       \prod_{i=r}^{n_--1} \Val{\chi_i \eta_{i-r+1}^{-1} \Val{.}_E^{-\frac{1}{2}}}(\alpha_i)  
       \prod_{j=1}^{m_-} \Val{\eta_j \chi_{j+r}^{-1} \Val{.}_E^{-\frac{1}{2}}}(\beta_j) \times \left\{ \begin{array}{l}
          \Val{\chi_{n_-}\Val{.}_E^{-\frac{1}{2}}}(\alpha_{n_-}),  \\
          \Val{\chi_{n_-}\Val{.}_E^{-1}}(\alpha_{n_-}),
       \end{array}
       \right.
\end{equation*}
in the inert case, depending on whether $n$ is even or odd, and
\begin{align*}
    \Val{Y_{\chi,\eta}(g)}=  &\Val{\chi_{1}^{-1} \Val{.}_F^{\frac{n-1}{2}}}(\det(g))
          \prod_{i=1}^{r} \Val{\chi_i \chi_{i+1}^{-1} \Val{.}_F^{-1}}(\alpha_{n-i})  
          \prod_{i=r+1}^{r'} \Val{ \chi_{i+1}^{-1} \eta_{r'-i+1}^{-1} \Val{.}_F^{-\frac{1}{2}}}(\alpha_{n-i}) \\       
          &\times \prod_{i=r'+1}^{n-1} \Val{\chi_i \chi_{i+1}^{-1} \Val{.}_F^{-1}}(\alpha_{n-i}) 
          \prod_{j=1}^{m} \Val{ \eta_{j} \chi_{r'-j+1}\Val{.}_F^{-\frac{1}{2}}}(\beta_{j}).
\end{align*}
in the split case. We now define $\mathcal{U}$ to be the non empty open subset of $(\chi,\eta)$ satisfying
\begin{equation*}
    \left\{
\begin{array}{ll}
    \Val{\chi_i \chi_{i+1}^{-1}} < q_E^{-1} & 1 \leq i < r,  \\
    \Val{\chi_i \eta_{i-r+1}^{-1}} < q_E^{-\frac{1}{2}} & r \leq i < n_-, \\
    \Val{\chi_{i+1}^{-1} \eta_{i-r+1}} < q_E^{-\frac{1}{2}} & r \leq i < n_-, \\
    \Val{\chi_{n_-}^{-1}} < q_E^{-\frac{1}{2}} & \text{even case}, \\
    \Val{\chi_{n_-}^{-1}} < q_E^{-1} & \text{odd case}, \\
\end{array}
    \right.
    \text{and }
      \left\{
\begin{array}{ll}
    \Val{\chi_i \chi_{i+1}^{-1}} < q_F^{-1} & 1 \leq i \leq r,  \\
    \Val{\chi_{i+1}^{-1} \eta_{r'-i+1}^{-1}} < q_F^{-\frac{1}{2}} & r < i \leq r', \\
    \Val{\chi_i \chi_{i+1}^{-1}} < q_F^{-1} & r'< i < n,  \\
     \Val{\eta_{j} \chi_{r'-j+1}} < q_F^{-\frac{1}{2}} & 1 \leq j \leq m,
\end{array}
    \right.
\end{equation*}
in the inert and split case respectively. It follows from Proposition~\ref{prop:doucle_class_description} that for any $(\chi,\eta) \in \mathcal{U}$ the function $Y_{\chi,\eta}$ is continuous on $\mathrm{U}(V)$.

\subsubsection{}
For $(\chi,\eta) \in \mathcal{U}$, $f_V \in C_c^\infty(\mathrm{U}(V)), \; f_J \in C_c^\infty(J)$, consider the absolutely convergent integral
\begin{equation}
\label{eq:generic_pairing}
    \mathcal{L}_{\chi,\eta}(F_{\chi}(f_V) \otimes F_{\eta }^J(f_J))=\int_{\mathrm{U}(V)} \int_J f_V(g_V) f_J(g_J) Y_{\chi,\eta}(g_V g_J^{-1}) dg_V dg_J.
\end{equation}
Then $\mathcal{L}_{\chi,\eta}$ defines a non zero element in $\Hom_H(I_{B_V}^{\mathrm{U}(V)}( \chi) \otimes \overline{\psi}_U. I_{B_J}^{J} (\eta \overline{\psi}),\cc)$, and it follows that
\begin{equation}
    \mathcal{W}^I_{\chi,\eta}(g):=
    \left\{
    \begin{array}{lll}
    \mathcal{L}_{\chi,\eta}(R(g).\Phi_\chi^\circ \otimes \Phi_{\eta}^\circ),  &  g \in \mathrm{U}(V), & \text{if } r\geq 1; \\
     \mathcal{L}_{\chi,\eta}(\Phi_\chi^\circ \otimes  R(g^{-1}).\Phi_{\eta}^\circ),    &  g \in J, & \text{if } r=0.
    \end{array}
    \right.
\end{equation}
is a Whittaker--Shintani function. We will explicitly compute $\mathcal{W}_{\chi,\eta}^I(1)$ in Proposition \ref{prop:normalization}.

\subsection{Analytic section}
\label{subsec:analytic}
By Lemma~\ref{lem:induced_reps} we have
\begin{equation}
\label{eq:pairing_space}
    \Hom_H(I_{B_V}^{\mathrm{U}(V)} (\chi) \otimes I_{B_W}^{\mathrm{U}(W)} (\mu \eta) \otimes \overline{\nu}_{\mu,\psi},\cc)=\Hom_H(I_{B_V}^{\mathrm{U}(V)} (\chi )\otimes \overline{\psi}_U. I_{B_J}^J(\eta \overline{\psi}), \cc).
\end{equation}
For $\chi$ and $\eta$ in general position, $I_{B_V}^{\mathrm{U}(V)} (\chi) \otimes I_{B_W}^{\mathrm{U}(W)} (\mu \eta)$ is irreducible. The multiplicity one result~\cite[Corollary~16.3]{GGP} says
\begin{equation}
\label{eq:mult_one}
    \dim \Hom_H(I_{B_V}^{\mathrm{U}(V)}( \chi )\otimes I_{B_W}^{\mathrm{U}(W)} (\mu \eta) \otimes \overline{\nu}_{\mu,\psi},\cc) \leq 1.
\end{equation}
We can now apply Bernstein's Theorem~\cite{Ban} (see also~\cite[Lemma~5.1]{Shen}).
\begin{prop}
\label{prop:analytic_pairing}
    There exists a dense subset $\mathcal{V}$ of $(\cc^\times)^{n_-}\times (\cc^\times)^{m_-}$ such that for every $f_V$ and $f_J$ the map $(\chi,\eta) \mapsto \mathcal{L}_{\chi,\eta}(F_{\chi}(f_V) \otimes F_{\eta}^J(f_J))$ defined on $\mathcal{U}$ by \eqref{eq:generic_pairing} extends to a rational function on $\mathcal{V}$.  
\end{prop}
In particular, we obtain a non zero $\mathcal{W}^I_{\chi,\eta} \in \mathrm{WS}_{\chi,\eta}$ for $(\chi,\eta)$ in general position. Let $\Phi_{\chi,\eta}^I \in (\Ind_H^G(\nu_{\mu,\psi}))^K$ (($\Ind_J^{\mathrm{U}(V) \times J} (\nu_{\mu,\psi}))^{K_V \times K_J}$ if $r=0$) be the associated function by Proposition~\ref{prop:isos}.

\section{Formula for the normalized Whittaker--Shintani function}

\label{section:Formula}

We prove the formulae for Whittaker--Shintani functions of Theorem~\ref{thm:W_formula} (in \S\ref{subsubsec:proof_prop}) and Proposition~\ref{prop:W_formula2} (in \S\ref{subsubsec:end_proof_W}), although additional results from \S\ref{section:L} will be needed.

\subsection{Iwahori-fixed vectors}
\subsubsection{}
\label{subsubsection:intertwining}
Let $\chi$ be an unramified character of $T_V$ in general position. For $\alpha \in \Sigma_{V,\mathrm{nd}}^+$, let $c_{\alpha}(\chi)$ be the constant defined in~\cite[Section~3]{Cas}, i.e. 
\begin{equation*}
    c_{e_a \pm e_b}(\chi)=
    \left\{ 
    \begin{array}{ll}
       \frac{1-q_E^{-1}\chi_a\chi_b^{\pm 1}}{1-\chi_a\chi_b^{\pm 1}} & \text{inert case,} \\
       \frac{1-q_F^{-1}\chi_a\chi_b^{-1}}{1-\chi_a\chi_b^{-1}}  & \text{split case,}
    \end{array}
    \right. \text{ and } \left\{ 
    \begin{array}{ll}
        c_{2e_a}(\chi)=\frac{1-q_F^{-1} \chi_a}{1-\chi_a}  & \text{ inert even case,} \\
        c_{e_a}(\chi)=\frac{(1-q_E^{-1} \chi_a)(1+q_F^{-1} \chi_a)}{1-\chi_a^2}  & \text{ inert odd case,}
    \end{array}
    \right.
\end{equation*}
where we only consider $e_a-e_b$ in the split case. For $w \in W_V$, set
\begin{equation}
\label{eq:c_w_defi}
    c_w^V(\chi)=\prod_{\substack{\alpha \in \Sigma_{V,\mathrm{nd}}^+ \\ w \alpha<0}} c_\alpha(\chi).
\end{equation}
 There exists an intertwining operator $T_{w,\chi} : I_{B_V}^{\mathrm{U}(V)} (\chi )\to I_{B_V}^{\mathrm{U}(V)} (w \chi)$ introduced in~\cite[Section~3]{Cas}. Define the normalization $\overline{T}_{w,\chi}:=c_w^V(\chi)^{-1} T_{w,\chi}$. Recall that we have introduced a spherical vector $\Phi^\circ_\chi$ and Iwahori-fixed vectors $\Psi_{w,\chi}$ in \eqref{eq:spheri_vector} and \eqref{eq:Iwahori_vector} respectively. By~\cite[Theorem~3.1]{Cas} it satisfies the relations $ \overline{T}_{w_1, w_2 \chi} \circ \overline{T}_{w_2, \chi}=\overline{T}_{w_1 w_2,\chi}$ and $\overline{T}_{w,\chi}(\Phi_\chi^\circ)=\Phi_{w\chi}^\circ$.

Using the Bruhat decomposition and evaluating $\Phi_\chi^\circ$ on $W_V$, we have
\begin{equation*}
    \Phi_{\chi}^\circ= \frac{1}{\vol(I_V)} \sum_{w \in W_V} c_{w_0}^V(w\chi) \overline{T}_{w^{-1},w \chi} \Psi_{1,w \chi},
\end{equation*}
and it follows from~\cite[Proposition~1.10]{KMS} that for $\lambda_V \in \Lambda_V^-$
\begin{equation}
\label{eq:I_VaI_V}
    R(I_V\lambda_V I_V) \Phi_{\chi}^\circ=\frac{\vol(I_V\lambda_V I_V)}{\vol(I_V)} \sum_{w \in W_V} c_{w_0}^V(w \chi).(w \chi)\delta_{B_V}^{-\frac{1}{2}}(\lambda_V) \overline{T}_{w^{-1},w \chi} \Psi_{1,w \chi}.
\end{equation}

\subsubsection{}
Let $\eta$ be an unramified character of $T_W$ in general position and let $w \in W_W$. We adapt the preceding discussion to $I_{B_J}^{J} (\eta \overline{\psi})$. Denote by $T_{w,\mu \eta} : I_{B_W}^{\mathrm{U}(W)} (\mu \eta) \to I_{B_W}^{\mathrm{U}(W)} (w (\mu \eta))$ the map built in \S\ref{subsubsection:intertwining}. Then 
\begin{equation*}
    T_{w,\eta}^J:= S_{w\eta} \circ \left(T_{w,\mu \eta} \otimes \mathrm{id} \right) \circ S_{\eta}^{-1},
\end{equation*}
provides an intertwining operator $I_{B_J}^{J} (\eta \overline{\psi}) \to I_{B_J}^J (w \eta \overline{\psi})$ (where $S_{\eta}$ is defined in \eqref{eq:S_eta2}). Define $c_w^W(\eta)$ as in \eqref{eq:c_w_defi}, and consider the normalized version $\overline{T}_{w,\eta}^J=c_w^W(\mu \eta)^{-1} T_{w,\eta }^J$ and recall that $\phi^\circ:=1_{Y_+^*(\oo_E)} \in \overline{\nu}_{\mu,\psi}$. By Lemma~\ref{lem:induced_reps} we have $S_{\eta}(\Phi^{\circ}_{\mu\eta } \otimes \phi^\circ)=\Phi^{\circ,J}_{\eta }$ and $S_{\eta}(\Psi_{w,\mu\eta}\otimes \phi^\circ)=\Psi^J_{w, \eta }$, where $\Phi^{\circ}_{\mu\eta}$ and  $\Psi_{w,\mu\eta}$ are the spherical and Iwahori fixed vectors in $I_{B_W}^{\mathrm{U}(W)} (\mu \eta)$ introduced in \eqref{eq:spheri_vector} and \eqref{eq:Iwahori_vector} respectively. In particular, $\overline{T}_{w,\eta }^J(\Phi^{\circ,J}_{\eta })=\Phi^{\circ,J}_{w .\eta }$. By the same proof as Lemma~\ref{lem:spheri}, we now see that any $f \in (I_{B_J}^{J} (\eta \overline{\psi}))^{I_J}$ is supported on $B_J W_W I_J$ so that we also obtain 
\begin{equation}
\label{eq:spheri_decompo_J}
    \Phi_{\eta }^{\circ,J}= \frac{1}{\vol(I_J)} \sum_{w \in W_W} c_{w_0}^W(w \mu \eta ) \overline{T}_{w^{-1},w \eta }^J \Psi^J_{1,w \eta},
\end{equation}
and therefore for $\lambda_W \in \Lambda_W^-$
\begin{equation}
\label{eq:I_WaI_W}
    R(I_W\lambda_W I_W) \Phi_{\eta }^{\circ,J}=\frac{\vol(I_W\lambda_W I_W)}{\vol(I_J)} \sum_{w \in W_W} c_{w_0}^W(w \mu \eta).(w \mu \eta )\delta_{B_W}^{-\frac{1}{2}}(\lambda_W) \overline{T}^J_{w^{-1},w \eta } \Psi_{1,w \eta }^J.
\end{equation}

\subsection{\texorpdfstring{$\gamma$}{Gamma}-factors}
\label{subsec:gamma_factor}

For $(\chi, \eta) \in (\cc^\times)^{n_-} \times (\cc^\times)^{m_-}$ in general position, set $\Gamma_1^V(\chi)=\frac{c_{w_0}^V(\chi)}{\mathbf{d}_V(\chi)}$, $\Gamma_1^W(\eta)=\frac{c_{w_0}^W(\eta)}{\mathbf{d}_W(\eta)}$ and let $\Gamma_2(\chi,\eta)$ be
\begin{equation}
\label{eq:gamma_defi_2}
    \left\{
\begin{array}{ll}
    \displaystyle  \prod_{j=1}^{m_-} \left( 
    \left(\prod_{i=1}^{n_-} L_E(\frac{1}{2},\chi_i \eta_j)L_E(\frac{1}{2},\chi_i^{-1} \eta_j)
    \right)
    \left(\prod_{i=1}^{r+j-1} \frac{L_E(\frac{1}{2},\chi_i \eta_j^{-1})}{L_E(\frac{1}{2},\chi_i^{-1} \eta_j)} \right)
    \right),
     \\
    \displaystyle  \prod_{j=1}^{m_-} \left( 
    \left(\prod_{i=1}^{n_-} L_E(\frac{1}{2},\chi_i \eta_j)L_E(\frac{1}{2},\chi_i^{-1} \eta_j)
    \right)
    \left(\prod_{i=1}^{r+j-1} \frac{L_E(\frac{1}{2},\chi_i \eta_j^{-1})}{L_E(\frac{1}{2},\chi_i^{-1} \eta_j)} \right)
    \right) \prod_{i=1}^{n_-} L_F(1,-\chi_i)\prod_{i=1}^{m_-} L_F(1,\eta_i),
    \\
    \displaystyle  \prod_{i+j \leq r'+1} L_F(\frac{1}{2},\chi_i \eta_j)
    \prod_{i+j > r'+1} L_F(\frac{1}{2},\chi_i^{-1} \eta_j^{-1}),
\end{array}
    \right.
\end{equation}
in the inert even case, inert odd case and split case respectively. Set $\Gamma(\chi,\eta):=\Gamma_1^V(\chi)\Gamma_1^W(\mu \eta)\Gamma_2(\chi,\eta)$. For $w=(w_V,w_W) \in W_G$, set $\overline{T}_{w,\chi, \eta }=\overline{T}_{w_V,\chi} \otimes \overline{T}^J_{w_W,\eta}$.

\begin{theorem}
\label{thm:functional_equation}
    For $\chi$ and $\eta$ in general position we have for every $w \in W_G$ the functional equation
    \begin{equation*}
        \frac{\mathcal{L}_{\chi,\eta}}{\Gamma(\chi,\eta)}=\frac{\mathcal{L}_{w_V \chi,w_W \eta} \circ \overline{T}_{w,\chi, \eta}}{\Gamma(w_V \chi,w_W \eta)}.
    \end{equation*}
\end{theorem}

By multiplicity one \eqref{eq:mult_one}, for $\chi$ and $\eta$ in general position there exists $\gamma(\chi,\eta,w_V,w_W) \in \cc$ such that 
\begin{equation*} 
\mathcal{L}_{w_V \chi,w_W \eta} \circ \overline{T}_{w,\chi, \eta}=\gamma(\chi,\eta,w_V,w_W) \mathcal{L}_{\chi,\eta}.
\end{equation*}
It is enough to prove that $\gamma(\chi,\eta,w_V,w_W)=\frac{\Gamma(w_V \chi,w_W \eta)}{\Gamma(\chi,\eta)}$ for $(w_V,w_W)=(w_\alpha,1)$ or $(1,w_\beta)$ with $\alpha \in \Delta_V$ and $\beta \in \Delta_W$. By Proposition~\ref{prop:analytic_pairing}, $\gamma$ is a rational function so we assume that $(\chi,\eta) \in \mathcal{U}$ is in general position and use the integral expression \eqref{eq:generic_pairing}. Theorem~\ref{thm:functional_equation} will follow from \eqref{eq:gamma_defi_2}, Propositions~\ref{prop:gamma_alpha_value} and~\ref{prop:gamma_beta_value}, and elementary computations. 

\subsection{Computation of \texorpdfstring{$\gamma(\chi,\eta,w_\alpha,1)$}{ the gamma factor}}

\subsubsection{}
\label{subsubsec:gamma_alpha}
Let $\alpha=e_i-e_{i+1} \in \Delta_V$. Define $I_J^1:=I_W \ltimes Y_+^0 Y_-^{*,1} Z^0$ and $\Psi^{1,J}_{\eta}=F_{\eta }^J(1_{I_J^1})$. By~\cite[Theorem~3.4]{Cas} we have
\begin{equation*}
    \overline{T}_{w_\alpha,\chi}(\Psi_{1,\chi} + \Psi_{w_\alpha,\chi})=\Psi_{1,w_\alpha \chi} + \Psi_{w_\alpha,w_\alpha \chi},
\end{equation*}
so that
\begin{equation}
\label{eq:gamma_value}
    \gamma(\chi,\eta,w_\alpha,1)= \frac{\mathcal{L}_{w_\alpha \chi,\eta}(R(\mu_+^* w_0) (\Psi_{1,w_\alpha \chi} + \Psi_{w_\alpha,w_\alpha \chi}) \otimes \Psi^{1,J}_{\eta})}{\mathcal{L}_{\chi,\eta}(R(\mu_+^* w_0)( \Psi_{1,\chi} + \Psi_{w_\alpha,\chi}) \otimes \Psi^{1,J}_{\eta})},
\end{equation}
granted the denominator is non-zero, where we recall that $\mu_+^*=h(1_{Y_+^*},0)$. 
\begin{rem}
\label{rem:r=0_case}
    In \eqref{eq:gamma_value} we implicitly assume that $r \geq 1$. Otherwise, we compute instead
    \begin{equation*}
    \gamma(\chi,\eta,w_\alpha,1)= \frac{\mathcal{L}_{w_\alpha \chi,\eta}(R(w_0) (\Psi_{1,w_\alpha \chi} + \Psi_{w_\alpha,w_\alpha \chi})\otimes R((\mu_+^*)^{-1})\Psi^{1,J}_{\eta})}{\mathcal{L}_{\chi,\eta}(R( w_0)( \Psi_{1,\chi} + \Psi_{w_\alpha,\chi}) \otimes R((\mu_+^*)^{-1})\Psi^{1,J}_{\eta})}.
\end{equation*}
As the calculations are identical, we ignore this issue to simplify the proof. We will work under the assumption $r\geq1$ throughout this section, and leave the easy modifications for $r=0$ to the reader.
\end{rem}

Theorem~\ref{thm:functional_equation} in that case follows from the next proposition which we prove in the rest of this subsection.

\begin{prop}
    \label{prop:gamma_alpha_value}
    The value of $\mathcal{L}_{\chi,\eta}(R(\mu_+^* w_0)( \Psi_{1, \chi} + \Psi_{w_\alpha,\chi})\otimes \Psi^{1,J}_{\eta})$ is 
    \begin{equation*}
        \vol(I_V)\vol(I_J^1) \times \left\{
\begin{array}{lll}
    q_EL_E(1,\chi_i \chi_{i+1}^{-1})^{-1} & 
    1 \leq i < r  &\text{inert case,} \\
    q_FL_F(1,\chi_i \chi_{i+1}^{-1})^{-1} & 
        \left\{ \begin{array}{l}
        1 \leq i \leq r,  \\
             \text{or } r' < i < n 
        \end{array} \right. & \text{split case,}\\
    (q_E-1)\frac{L_E(\frac{1}{2},\chi_i \eta_{i-r+1}^{-1})L_E(\frac{1}{2},\chi_{i+1}^{-1} \eta_{i-r+1})}{L_E(1,\chi_i \chi_{i+1}^{-1})} & r \leq i < n_-  &\text{inert case,} \\
    (q_F-1)\frac{L_F(\frac{1}{2},\chi_i \eta_{i-r'+1})L_F(\frac{1}{2},\chi_{i+1}^{-1} \eta_{i-r+1}^{-1})}{L_F(1,\chi_i \chi_{i+1}^{-1})} & r < i \leq r'  &\text{split case,} \\
    q_FL_F(1,\chi_{n_-})^{-1} & i=n_-  &\text{inert even case,}
    \\
    q_Eq_FL_E(1,\chi_{n_-})^{-1} & i=n_- &\text{inert odd case}.
\end{array}
        \right.
    \end{equation*}
\end{prop}

\subsubsection{}
The next lemma is a straightforward adaptation of~\cite[Lemma~2.8.2]{Shen2}.

\begin{lem}
\label{lem:technical_equalities}
    Let $\alpha \in \Delta_V$, $\beta \in \Delta_W$. Then for any values of $\pm 1$
    \begin{align}
        I_V w_0^{\pm 1} \mu_+^{*,\pm 1} I_J^1 U^0 &= B_V^0 w_0^{\pm 1} \mu_+^{*,\pm 1} B_J^0 U^0, \label{eq:tech_1} \\
        I_V w_0^{\pm 1} \mu_+^{*,\pm 1} I_W U^0 &\subset B_V^0 w_0^{\pm 1} B_J^0 Y_+^{*,0} U^0, \label{eq:tech_new} \\
        I_V w_\alpha I_V w_0^{\pm 1} \mu_+^{*,\pm 1} I_J^1 U^0 &= B_V^0 w_\alpha N_\alpha^0 w_0^{\pm 1} \mu_+^{*,\pm 1} B_J^0 U^0, \label{eq:tech_2}  \\
       I_V w_0^{\pm 1} \mu_+^{*,\pm 1} I_J w_\beta I_J U^0 &= B_V^0 w_0^{\pm 1} N_\beta^0 w_\beta Y_+^{*,0} B_J^0 U^0. \label{eq:tech_3} 
    \end{align}
\end{lem}

\begin{lem}
\label{lem:alpha_base_value}
    We have
    \begin{equation*}
        \mathcal{L}_{\chi,\eta}(R(\mu_+^* w_0) \Psi_{1,\chi} \otimes \Psi^{1,J}_{\eta})=\vol(I_V)\vol(I_J^1).
    \end{equation*}
\end{lem}

\begin{proof}
    Using the integral expression \eqref{eq:generic_pairing} we have
    \begin{equation*}
        \mathcal{L}_{\chi,\eta}(R(\mu_+^* w_0) \Psi_{1,\chi} \otimes \Psi^{1,J}_{\eta})=\int_{I_V \times I_J^1} Y_{\chi,\eta}(g_V w_0^{-1} \mu_+^{*,-1} g_J) dg_V dg_J.
    \end{equation*}
    The result follows from the invariance properties of $Y_{\chi,\eta}$ in \eqref{eq:Y_defi} and Lemma~\ref{lem:technical_equalities} \eqref{eq:tech_1}.
\end{proof}

\subsubsection{}
We now compute $\mathcal{L}_{w_\alpha \chi,\eta}(R(\mu_+^* w_0) \Psi_{1,w_\alpha \chi}\otimes \Psi^{1,J}_{\eta})$. Using \eqref{eq:Y_defi} and Lemma~\ref{lem:technical_equalities} \eqref{eq:tech_2} this is
\begin{equation}
\label{eq:I_alpha_def}
    \int_{I_V w_\alpha I_V} \int_{I_J^1} Y_{\chi,\eta}(g_V w_0^{-1} \mu_+^{*,-1} g_J) dg_V dg_J=\vol(I_V w_\alpha I_V) \vol(I_J^1) \int_{N_\alpha^0} Y_{\chi,\eta}(w_\alpha n_\alpha w_0 \mu_+^*) dn_\alpha.
\end{equation}
Denote by $I_\alpha$ this last integral. 

\textbf{Case $\SL_2$} We first treat the case $D(G_\alpha)=\SL_2(E)$ or $\SL_2(F)$, so that $N_\alpha$ is isomorphic to $E$ or $F$. We write $\alpha^\vee$ for the corresponding coroot of the maximal torus of $D(G_\alpha)$ (which is a cocharacter over $E$ in the first case). For $t \neq 0$ we have $w_\alpha n_\alpha(t)= \alpha^{\vee}(-t^{-1}) n_\alpha(-t) n_{-\alpha}(t^{-1})$. Therefore
\begin{equation*}
    I_\alpha=\int_{\Val{t} \leq 1} (\chi^{-1} \delta_{B_V}^{\frac{1}{2}})(\alpha^{\vee}(t^{-1}))Y_{\chi,\eta}(w_0 n_{w_0(-\alpha)}(t^{-1})\mu_+^*)dt.
\end{equation*}

\textbf{Case $\SL_2$ 1.} In the inert case with $1 \leq i < r-1$, or the split case with $1 \leq i < r$ or $r'+1< i \leq n$, then $n_{w_0(-\alpha)}(t^{-1})\mu_+^*=\mu_+^* n_{w_0(-\alpha)}(t^{-1})$. In the inert case with $i=r-1$, or the split case with $i=r'+1$, $n_{w_0(-\alpha)}(t^{-1})\mu_+^*=\mu_+^* n_{w_0(-\alpha)}(t^{-1})u$ with $\psi_U(u)=1$. Finally, in the split case with $i=r$, $n_{w_0(-\alpha)}(t^{-1})\mu_+^*=\mu_+^* n_{w_0(-\alpha)}(t^{-1})h(0,t^{-1})$. Therefore we always have
\begin{equation}
    \label{eq:case_1_alpha}
    I_\alpha=\int_{\Val{t} \leq 1} (\chi^{-1} \delta_{B_V}^{\frac{1}{2}})(\alpha^{\vee}(t^{-1})) \overline{\psi_E}(t^{-1})dt.
\end{equation}

\textbf{Case $\SL_2$ 2.} In the inert case with $r \leq i < n_-$ set $j=i-r$, and in the split case with $r+1 \leq i \leq r'$, set $j=r'-i$. Let $\bar{.}$ be $c \in \gal(E/F)$ in the inert case, and identity in the split case. Define $\mathbf{t}_{j+1}(t)=t_W((1-t^{-1})w_{j+1}) \in T_W$ and $y_{j+1}^*=\sum_{k=1}^{m_-} w_k^*-\overline{t}^{-1} w_{j+1}^*$. Then $n_{w_0(-\alpha)}(t^{-1}) \mu_+^*=h(y_1^*,0)$ if $j=0$, and $n_{w_0(-\alpha)}(t^{-1}) \mu_+^*=h(y_{j+1}^*,0) n_{w_0(-\alpha)}(t^{-1})$ otherwise. But $h(y_{j+1}^*,0)=\mathbf{t}_{j+1}(t)^{-1} \mu_+^* \mathbf{t}_{j+1}(t)$ so that
\begin{equation}
\label{eq:case_2_alpha}
    I_\alpha=\int_{\Val{t} \leq 1} (\chi^{-1} \delta_{B_V}^{\frac{1}{2}})(\alpha^{\vee}(t^{-1})) (\chi^{-1} \delta_{B_V}^{\frac{1}{2}})({}^{w_0}\mathbf{t}_{j+1}(t)^{-1}) (\eta \delta_{B_J}^{-\frac{1}{2}})(\mathbf{t}_{j+1}(t))dt.
\end{equation}

\textbf{Case $\SL_2$ 3.} In the inert even case with $i=n_-$, $\alpha=2e_{n_-}$ and  $D(G_\alpha)$ is $\SL_2(F)$. Then $n_{w_0(-\alpha)}(t^{-1})\mu_+^*=\mu_+^* h(t^{-1} w_{m_-}, t^{-1})n_{w_0(-\alpha)}(t^{-1})$ and 
\begin{equation}
\label{eq:case_3_alpha}
    I_\alpha=\int_{\Val{t} \leq 1} (\chi^{-1} \delta_{B_V}^{\frac{1}{2}})(\alpha^{\vee}(t^{-1})) \overline{\psi_E}(t^{-1})dt.
\end{equation}

\textbf{Case $\mathrm{SU}_3$} Assume now that $D(G_\alpha)$ is $\mathrm{SU}_3(F)$, which occurs in the inert odd case with $i=n_-$. Replace $e_{n_-}$ by $\alpha=2e_{n_-}$. In that case, recall that
\begin{equation}
\label{eq:SU_param}
    n_{\alpha}(x_1,x_2)=\begin{pmatrix}
        1 & -\tau \overline{x_1} & x_2 \\
        & 1 & x_1 \\
        & & 1
    \end{pmatrix}, \; x_1, x_2 \in E, \; \Tr_{E/F}(\frac{x_2}{\tau})=-N_{E/F}(x_1).
\end{equation}
For $x_2 \neq 0$ we have $w_\alpha n_\alpha(x_1,x_2) w_0 \in T_V(\oo_E) n_\alpha( \frac{x_1}{\overline{x_2}},\frac{-1}{x_2}) \alpha^{\vee}(\frac{1}{\overline{x_2}}) w_0 n_\alpha(\frac{x_1}{\overline{x_2}},\frac{-1}{x_2})$ and furthermore $n_\alpha(x_1,x_2) \mu_+^* n_\alpha(x_1,x_2)^{-1}=\mu_+^* h(x_1 w_{m_+} + (-\tau \overline{x_1}+x_2) w_{m_-},\frac{1}{2}\Tr(x_2))$. Therefore, by Proposition~\ref{prop:doucle_class_description}, 
\begin{align}
\label{eq:SU_3_coord}
    Y_{\chi,\eta}(w_\alpha n_\alpha(x_1,x_2) w_0 \mu_+^*)&=(\chi^{-1} \delta_{B_V}^{\frac{1}{2}})\left(\alpha^\vee(x_2^{-1})\right)Y_{\chi,\eta}\left(w_0 \mu_+^* h\left(\frac{x_1}{\overline{x_2}}w_{m_+},-\frac{1}{2}\Tr(\frac{1}{x_2})\right)\right) \\
 &= \left\{
    \begin{array}{cc}
        0 & \text{if } \Val{\frac{x_1}{{x_2}}}>1  \\
         (\chi \delta_{B_V}^{-\frac{1}{2}})\left(\alpha^\vee(x_2)\right)\overline{\psi} \left(-\frac{1}{2}\Tr(\frac{1}{x_2})\right) & \text{if } \Val{\frac{x_1}{{x_2}}} \leq 1
    \end{array}
 \right. \nonumber
 \end{align}
 Let $x \in E^\times$ and $y \in F^\times$. In our coordinates, $x_1=x$ and $x_2=y-\tau \frac{N(x)}{2}$ and we are computing
 \begin{equation*}
     I_\alpha = \int_{\Val{x}_E \leq 1}  \int_{\Val{y}_F \leq 1} (\chi \delta_{B_V}^{-\frac{1}{2}})\left(\alpha^\vee(x_2)\right)\overline{\psi} \left(-\frac{1}{2}\Tr(\frac{1}{x_2})\right)  1_{\Val{x}_E \leq \Val{x_2}_E} dy dx.
 \end{equation*}
 We now split the integral.
 \begin{itemize}
     \item If $\Val{x}_E=1$, then the integrand is $1$ so that we get $\frac{q_E-1}{q_E}$.
     \item If $\Val{x}_E<1$ and $\Val{y}_E \leq \Val{N(x)}_E$, then $\Val{x_2}_E=\Val{N(x)}_E < \Val{x}_E$, and the integrand is $0$.
     \item If $\Val{x}_E<1$ and $\Val{y}_E > \Val{N(x)}_E$, we need $\Val{x}_E \leq \Val{x_2}_E=\Val{y}_E$ otherwise the integrand is zero. In this range, $\overline{\psi} \left(-\frac{1}{2}\Tr(\frac{1}{x_2})\right)=\overline{\psi}(-y^{-1})$.
 \end{itemize}
Therefore,
\begin{equation}
\label{eq:case_4_alpha}
    I_\alpha=\frac{q_E-1}{q_E}+\int_{\Val{x}_E <1} \int_{\substack{y \in F^\times \\ \Val{x}_E \leq \Val{y}_E \leq 1}} (\chi \delta_{B_V}^{-\frac{1}{2}})\left(\alpha^\vee(y)\right)\overline{\psi}(y^{-1})dydx.
\end{equation}

The integrals \eqref{eq:case_1_alpha}, \eqref{eq:case_2_alpha}, \eqref{eq:case_3_alpha} and \eqref{eq:case_4_alpha} can now be computed by standard techniques (\cite[Lemmas~8.5, 8.9 and 8.11]{Shen}) and Proposition~\ref{prop:gamma_alpha_value} follows from Lemma~\ref{lem:alpha_base_value}, \eqref{eq:I_alpha_def} and the value of $\vol(I_V w_\alpha I_V)$. Note that $(\chi,\eta) \in \mathcal{U}$ implies that these integrals converge. 

\subsection{Computation of \texorpdfstring{$\gamma(\chi,\eta,1,w_{\beta})$}{gamma}}

\subsubsection{}
Let $\beta=e_j-e_{j+1} \in \Delta_W$. As in section~\ref{subsubsec:gamma_alpha} we have
\begin{equation*}
      \gamma(\chi,\eta,1,w_\beta)= \frac{\mathcal{L}_{\chi,w_\beta \eta}(R(\mu_+^* w_0) \Psi_{1,\chi}\otimes (\Psi_{1,w_\beta \eta}^J+\Psi_{w_\beta, w_\beta \eta}^J))}{\mathcal{L}_{\chi, \eta}(R(\mu_+^* w_0) \Psi_{1,\chi}\otimes (\Psi_{1,\eta}^J+\Psi_{w_\beta,\eta}^J))}.
\end{equation*}
\begin{prop}
\label{prop:pi_defi}
    Recall that $\Delta_{T_W}'$ was defined in \S\ref{subsubsec:Gross_defi}. Set 
    \begin{equation*}
        \Pi(\chi,\eta)=\Delta_{T_W}^{',-1} \times \left\{
        \begin{array}{ll}
            \displaystyle  \prod_{j=1}^{m_-} L_E(\frac{1}{2},\eta_j \chi_{r+j}^{-1}) &  \text{ inert case,}\\
            \displaystyle  \prod_{j=1}^{m} L_F(\frac{1}{2},\eta_j \chi_{r'-j+1}) & \text{ split case.}
        \end{array}
        \right.
    \end{equation*}
    \label{prop:gamma_beta_value}
    The value of $\mathcal{L}_{\chi, \eta}(R(\mu_+^* w_0) \Psi_{1,\chi}\otimes (\Psi_{1,\eta}^J+\Psi_{w_\beta,\eta}^J))$ is 
    \begin{equation*}
        \vol(I_V)\vol(I_J) \Pi(\chi,\eta) \times \left\{
\begin{array}{lll}
    q_E \frac{L_E(\frac{1}{2},\eta_j \chi_{r+j+1}^{-1})L_E(\frac{1}{2},\chi_{r+j} \eta_{j+1}^{-1})}{L_E(1,\eta_j \eta_{j+1}^{-1})L_E(1,\chi_{r+j} \chi_{r+j+1}^{-1})} & 1 \leq j < m_- & \text{ inert case,}
    \\
    q_F \frac{L_F(\frac{1}{2},\eta_j \chi_{r'-j})L_F(\frac{1}{2},\eta_{j+1}^{-1} \chi_{r'-j+1}^{-1})}{L_F(1,\eta_j \eta_{j+1}^{-1})L_F(1,\chi_{r'-j} \chi_{r'-j+1}^{-1})} & 1 \leq j <m & \text{ split case,} \\
     q_F \frac{L_F(1,\chi_{m_-} \eta_{n_-})}{L_F(1,-\eta_{m_-})L_F(1,\chi_{n_-})} & j=m_- & \text{ inert even case,} \\
      q_E q_F \frac{L_F(1,\chi_{m_-} \eta_{n_-})}{L_E(1,-\eta_{m_-})L_E(1,\chi_{n_-})} & j=m_- & \text{ inert odd case.} \\
\end{array}
        \right.
    \end{equation*}
\end{prop}

\subsubsection{}
\begin{lem}
\label{lem:beta_base_value} We have
    \begin{equation*}
        \mathcal{L}_{\chi, \eta}(R(\mu_+^* w_0) \Psi_{1,\chi}\otimes \Psi_{1,\eta}^J)= \vol(I_V)\vol(I_J) \Pi(\chi,\eta).
    \end{equation*}
\end{lem}

\begin{proof}
    By Lemma~\ref{lem:technical_equalities} \eqref{eq:tech_3} we have
    \begin{equation*}
        \mathcal{L}_{\chi, \eta}(R(\mu_+^* w_0) \Psi_{1,\chi}\otimes \Psi_{1,\eta}^J)=\vol(I_V)\vol(I_J) \int_{Y_-^*(\oo_E)}Y_{\chi,\eta}(w_0 h(y^*,0))dy^*.
    \end{equation*}
    For $y^*$ with all coordinates being non zero, we have $h(y^*,0)=t_W(y^*)^{-1} \mu_+^* t_W(y^*)$ so that
    \begin{equation*}
        \int_{Y_-^*(\oo_E)}Y_{\chi,\eta}(w_0 h(y^*,0))dy^*= \int_{Y_-^*(\oo_E)} (\chi^{-1} \delta_{B_V}^{\frac{1}{2}})({}^{w_0}t_W(y^*)^{-1})(\eta \delta_{B_J}^{-\frac{1}{2}})(t_W(y^*))dy^*.
    \end{equation*}
    This can be computed by decomposing as a product (see~\cite[Lemma~8.9]{Shen}).
\end{proof}

\subsubsection{}
By Lemma~\ref{lem:technical_equalities} \eqref{eq:tech_3} again we have
\begin{equation}
\label{eq:I_beta_def}
        \mathcal{L}_{\chi, \eta}(R(\mu_+^* w_0) \Psi_{1,\chi}\otimes \Psi_{w_\beta,\eta}^J)=\vol(I_V)\vol(I_J w_\beta I_J) \int_{N_\beta^{0,'}} \int_{Y_-^*(\oo_E)}Y_{\chi,\eta}(w_0 n_\beta' w_\beta h(y^*,0))dy^* dn_{\beta}.
\end{equation}
Denote by $I_\beta$ this last double integral.

\textbf{Case $\SL_2$}. Assume that $D(G_\beta)$ is $\SL_2(F)$ or $\SL_2(E)$. We write again $\beta^\vee$ for the corresponding coroot of the maximal torus of $D(G(\beta))$, so that $n_\beta'(t) w_\beta=\beta^{\vee}(-t) n_{-\beta}'(t) n_\beta'(-t^{-1})$. Therefore
\begin{equation*}
    I_\beta=\int_{\Val{t}\leq 1} \int_{Y_-^*(\oo_E)}(\chi^{-1} \delta_{B_V}^{\frac{1}{2}})({}^{w_0}\beta^\vee(t))Y_{\chi,\eta}(w_0 n_\beta'(t^{-1}) h(y^*,0)n_\beta'(t^{-1})^{-1})dy^* dt.
\end{equation*}

\textbf{Case $\SL_2$ 1.} In the inert case with $j < m_-$ or in the split case with any $j$, set $y^{*}(t)=y^*-y_j \overline{t}^{-1} w_{j+1}^*$ so that when all coordinates are non zero we have $ n'_\beta(t^{-1}) h(y^*,0)n'_\beta(t^{-1})^{-1}=h(y^{*}(t),0)=t_W(y^{*}(t))^{-1} \mu_+^* t_W(y^{*}(t))$. Therefore
\begin{equation}
\label{eq:case_1_beta}
    I_\beta=\int_{\Val{t}\leq 1} \int_{Y_-^*(\oo_E)}(\chi^{-1} \delta_{B_V}^{\frac{1}{2}})({}^{w_0}\beta^\vee(t) {}^{w_0}t_W(y^{*}(t))^{-1}) (\eta \delta_{B_J}^{-\frac{1}{2}})(t_W(y^{*}(t))) dy^* dt.
\end{equation}

\textbf{Case $\SL_2$ 2.} In the inert even case with $j=m_-$ we have $\beta=2e_{m_-}$ and 
\begin{equation*}
    n'_\beta(t^{-1}) h(y^*,0)n'_\beta(t^{-1})^{-1}=h(y^{*}+t^{-1}y^*_{m_-} w_{m_-},0)=t_W^{-1}(y^{*}) \mu_+^* t_W(y^{*})h(t^{-1}y^*_{m_-} w_{m_-},t^{-1} N_{E/F}(y_{m_-}^*)).
\end{equation*}
We get 
\begin{equation}
\label{eq:case_2_beta}
    I_\beta=\int_{\Val{t}\leq 1} \int_{Y_-^*(\oo_E)}(\chi^{-1} \delta_{B_V}^{\frac{1}{2}})({}^{w_0}\beta^\vee(t) {}^{w_0}t_W(y^{*})^{-1}) (\eta \delta_{B_J}^{-\frac{1}{2}})(t_W(y^{*})) \overline{\psi}(t^{-1} N_{E/F}(y_{m_-}^*)) dy^* dt.
\end{equation}

\textbf{Case $\mathrm{SU}_3$} Assume that $D(G_\alpha)$ is $\mathrm{SU}_3$, which occurs in the inert odd case with $j=m_-$. Take $\beta=e_{m_-}$. Using the same coordinates as in \eqref{eq:SU_3_coord} we see that $Y_{\chi,\eta}(w_0 n'_\beta(x_1,x_2) w_\beta h(y^*,0))$ is
\begin{align*}
    &\; \; (\chi^{-1} \delta_{B_V}^{\frac{1}{2}})\left({}^{w_0}t_W(y^*)^{-1}\beta^{\vee}(x_2)^{-1}\right) (\eta \delta_{B_J}^{-\frac{1}{2}})(t_W(y^*)) Y_{\chi,\eta}\left(w_0 \mu_+^* h\left(y_{m_-}^*\frac{x_1}{\overline{x_2}}w_{m_+},-\frac{1}{2}\Tr(\frac{N_{E/F}(y_{m_-}^*)}{x_2})\right)\right) \\
    &= \left\{
    \begin{array}{lc}
        0 & \text{if } \Val{y_{m_-}^*\frac{x_1}{{x_2}}}>1,  \\
         (\chi^{-1} \delta_{B_V}^{\frac{1}{2}})\left(t_W(y^*)\beta^\vee(x_2)^{-1}\right)(\eta \delta_{B_J}^{-\frac{1}{2}})(t_W(y^*))\overline{\psi} \left(-\frac{1}{2}\Tr(\frac{N_{E/F}(y_{m_-}^*)}{x_2})\right) & \text{if } \Val{y_{m_-}^*\frac{x_1}{{x_2}}} \leq 1.
    \end{array}
 \right. \nonumber
 \end{align*}
 $I_\beta$ is a product of an integral over $y_1^*, \hdots, y_{m_--1}^*$, which is computed as in Lemma~\ref{lem:beta_base_value}, and of
 \begin{equation}
 \label{eq:case_3_beta}
     \int_{\Val{z}_E \leq 1} \int_{\substack{\Val{x}_E \leq 1 \\ \Val{y}_F \leq 1}}(\chi^{-1} \delta_{B_V}^{\frac{1}{2}})\left(\beta^\vee(x_2^{-1} z)\right)(\eta \delta_{B_J}^{-\frac{1}{2}})(\beta^\vee( z))\overline{\psi} \left(-\frac{1}{2}\Tr(\frac{N_{E/F}(z)}{x_2})\right)  1_{\Val{zx}_E \leq \Val{x_2}_E} dy dx dz,
 \end{equation}
 where $y \in F$, $x_2=y-\tau \frac{N_{E/F}(x)}{2}$, and we have replaced $y_{m_-}^*$ by $z$ to ease notations. We split the integral.
 \begin{itemize}
     \item If $\Val{y}_E \leq \Val{N_{E/F}(x)}_E$, then $\Val{x_2}_E=\Val{N_{E/F}(x)}_E$ and the integrand is zero unless $\Val{z}_E \leq \Val{x}_E$. In that case $\overline{\psi} \left(-\frac{1}{2}\Tr(\frac{N_{E/F}(z)}{x_2})\right)=1$ and we are simply integrating a product of multiplicative characters.
     \item If $\Val{y}_E > \Val{N_{E/F}(x)}_E$, then the integrand is zero unless $\Val{zx}_E \leq \Val{y}_E$. In this case we see that $\overline{\psi} \left(-\frac{1}{2}\Tr(\frac{N_{E/F}(z)}{x_2})\right)=\overline{\psi} \left(-\frac{N_{E/F}(z)}{y}\right)$. 
 \end{itemize}
The integrals \eqref{eq:case_1_beta}, \eqref{eq:case_2_beta} and \eqref{eq:case_3_beta} can now be computed by splitting further depending on valuations thanks to~\cite[Lemmas~8.5, 8.9 and 8.11]{Shen}, and Proposition~\ref{prop:gamma_beta_value} follows from Lemma~\ref{lem:beta_base_value} and \eqref{eq:I_beta_def}. This concludes the proof of Theorem~\ref{thm:functional_equation}.

\subsection{Proofs of Theorem~\ref{thm:W_formula} and Proposition~\ref{prop:W_formula2}}
\label{subsec:proofs}

In this section we work under the assumption that $r \geq 1$. The case $r=0$ follows from the modification explained in Remark~\ref{rem:r=0_case}. Let us first remark that the conventions of Theorem~\ref{thm:W_formula} are different from the ones we are currently using as we are inducing from $B$ rather than $B^+=B_V^- \times B_W$, and because moreover we are working with $\mathrm{WS}_{\chi,\eta}$ instead of $\mathrm{WS}_{\chi,\overline{\mu} \eta}$. Set
  \begin{equation}
    \label{eq:d_defi}
    \mathbf{d}_V(\chi)=\prod_{\alpha \in \Sigma^+_{V,\mathrm{nd}}} \frac{1}{1-\langle \chi,\alpha^\vee \rangle}, \quad \mathbf{d}_W(\eta)=\prod_{\beta \in \Sigma^+_{W,\mathrm{nd}}} \frac{1}{1-\langle \eta,\beta^\vee \rangle}.
\end{equation}
Then $\mathbf{d}(\chi\boxtimes\eta)=\mathbf{d}_V(w_0\chi)\mathbf{d}_W(\eta)$. Moreover, because ${}^{w_{0,V}} B_V=B_V^-$ the identification of $T$ with the canonical torus $T_G$ coming from $T \subset B$ and $T \subset B^+$ differ by a conjugation by $w_{0,V}$. In particular, after a change of variable we see that \eqref{eq:formula_intro} in Theorem~\ref{thm:W_formula} is equivalent to, for $\lambda \in \Lambda^-$,
\begin{equation}
\label{eq:formula_not_intro}
        \frac{( \Phi_{\chi, \eta}(\lambda),\phi^\circ )_{\nu}}{(\Phi_{\chi,\eta}(1),\phi^\circ)_\nu} = \frac{\Delta_{\mathrm{U}(W)}}{\Delta_{T_W}} \sum_{w \in W_G} \mathbf{b}(w_V \chi, w_W \eta) \mathbf{d}_V(w_V\chi) \mathbf{d}_W(w_W \mu \eta) \left( \left(w_0 w_V \chi \boxtimes w_W \mu\eta\right) \delta_{B^+}^{-\frac{1}{2}} \right)(\lambda).
\end{equation}

\subsubsection{}
Recall that $\Phi_{\chi,\eta}^I$ was defined in \S\ref{subsec:analytic}. By Remark~\ref{rem:iso_explicit}, for any $\phi \in \overline{\nu}_{\mu,\psi}$ we have
\begin{equation}
\label{eq:explicit_pairing}
    ( \Phi_{\chi,\eta}^I(g_V,g_W), \phi )_\nu=\mathcal{L}_{\chi,\eta}\left(R(g_V) \Phi_\chi^\circ \otimes S_{\eta}\left(R(g_W) \Phi_{\mu \eta}^\circ \otimes \phi \right) \right), \quad (g_V,g_W) \in G.
\end{equation}

\begin{prop}
\label{prop:phi_0_calc}
    For every $\lambda \in \Lambda^-$ and $(\chi,\eta)$ in general position we have 
    \begin{equation*}
    \frac{( \Phi_{\chi,\eta}^I(\lambda_V,\lambda_W), \phi^\circ )_\nu}{\Delta_{T_W}^{',-1} \Gamma(\chi,\eta)}=\sum_{w \in W_G} \mathbf{b}(w_V \chi, w_W \eta) \mathbf{d}_V(w_V\chi) \mathbf{d}_W(w_W \mu \eta) 
    \left( \left(w_0 w_V \chi \boxtimes w_W \mu\eta\right) \delta_{B^+}^{-\frac{1}{2}} \right)(\lambda).
    \end{equation*}
\end{prop}

\begin{proof}
    Let $\lambda=(\lambda_V,\lambda_W) \in \Lambda_V^+ \times \Lambda_W^-$. Note that ${}^{w_0} \lambda_V^{-1} B_V^0 {}^{w_0} \lambda_V \subset B_V^0$ and $\lambda_W^{-1} B_W^0 \lambda_W^0 \subset B_W^0$. By \eqref{eq:tech_new},
\begin{equation}
\label{eq:spheri_scalar}
    ( R(\mu_+^* w_01_{I_V {}^{w_0} \lambda_V I_V}) \otimes R(1_{I_W \lambda_W I_W}) \Phi_{\chi,\eta}^I (1) , \phi^\circ )_\nu = \vol(I_V {}^{w_0} \lambda_V I_V) \vol(I_W \lambda_W I_W) ( \Phi_{\chi,\eta}^I(\lambda_V,\lambda_W), \phi^\circ )_\nu.
\end{equation}
But by \eqref{eq:I_VaI_V}, \eqref{eq:I_WaI_W} and \eqref{eq:explicit_pairing} the LHS is 
\begin{align*}
    \frac{\vol(I_V{}^{w_0}\lambda_V I_V)\vol(I_W\lambda_W I_W)}{\vol(I_V)\vol(I_W)} \sum_{(w_V,w_W) \in W_G} &c_{w_0}^V(w_V \chi) c_{w_0}^W(w_V \mu \eta) (w_0 w_V \chi) (\lambda_V) (w_W \mu \eta)(\lambda_W) \delta_{B^+}^{-\frac{1}{2}}(\lambda) \\
    & \times \mathcal{L}_{\chi,\eta} \circ \overline{T}_{w^{-1},w_V \chi, w_W \eta } \left( R(\mu_+^* w_0) \Psi_{1,w_V\chi} \otimes \Psi_{1,w_W \eta }^J \right).
\end{align*}
It remains to use Theorem~\ref{thm:functional_equation} and Lemma~\ref{lem:beta_base_value} and to do elementary computations.
\end{proof}

\subsubsection{}
We now state the formula for $\mathcal{W}_{\chi,\eta}^I(1)=( \Phi_{\chi,\eta}^I(1), \phi^\circ )_\nu$. We postpone the proof until \S\ref{section:proof_normalization}.
\begin{prop}
\label{prop:normalization}
    We have 
    \begin{equation*}
        \mathcal{W}_{\chi,\eta}^I(1)=\Delta_{T_W}^{',-1} \Gamma(\chi,\eta) \Delta_{T_W} \Delta_{\mathrm{U}(W)}^{-1}.
    \end{equation*}
\end{prop}

\subsubsection{End of the proof of Theorem~\ref{thm:W_formula}}
\label{subsubsec:end_proof_W}

\begin{lem}
\label{lem:regular}
    The rational map 
    \begin{equation}
    \label{eq:rational_norm}
        (\chi,\eta) \mapsto \Phi_{\chi,\eta}^{\circ}:=(\Delta_{T_W}^{',-1} \Gamma(\chi,\eta) \Delta_{T_W} \Delta_{\mathrm{U}(W)}^{-1})^{-1}\Phi^I_{\chi,\eta},
    \end{equation}
    extends to a regular function. In particular, for every $(\chi,\eta)$ there exists a unique $\mathcal{W}^\circ_{\chi,\eta} \in \mathrm{WS}_{\chi,\eta}$ such that $\mathcal{W}^\circ_{\chi,\eta}(1)=1$.
\end{lem}

\begin{proof}
    Let $(.,.)$ be the natural pairing between $\Ind_H^G \nu_{\mu,\psi}$ and $\ind_H^G (\overline{\nu}_{\mu,\psi})$. We have to show that for every $\Phi \in (\ind_H^G ( \overline{\nu}_{\mu,\psi} ))^K$ the map $(\chi,\eta) \mapsto (\Phi^{\circ}_{\chi,\eta},\Phi)$ is regular. Let $\Phi_0$ be the generator of $(\ind_H^G (\overline{\nu}_{\mu,\psi}))^K$ built in Proposition~\ref{prop:rank1}. It is enough to show that $(\chi,\eta) \mapsto (\Phi^{\circ}_{\chi,\eta},\Phi_0)=(\Phi^{\circ}_{\chi,\eta}(1),\phi^\circ)_\nu$ is regular. But by Proposition~\ref{prop:normalization} this map is constant equal to $1$ so the result follows.
\end{proof}

To conclude the proof of Theorem~\ref{thm:W_formula}, note that by Theorem~\ref{thm:mult1intro} and Lemma~\ref{lem:regular} we have $\Phi_{\chi,\eta} \neq 0$ if and only if $(\Phi_{\chi,\eta}(1),\phi^\circ)_{\nu}\neq 0$, and moreover in that case $\Phi_{\chi,\eta}/(\Phi_{\chi,\eta}(1),\phi^\circ)_{\nu}=\Phi_{\chi,\eta}^{\circ}$. Theorem~\ref{thm:W_formula} therefore follows from \eqref{eq:formula_not_intro}, Proposition~\ref{prop:phi_0_calc}, \eqref{eq:rational_norm} and elementary computations.

\subsubsection{}\label{subsubsec:proof_prop} Recall that in Proposition~\ref{prop:W_formula2} we have stated in alternative version of Theorem~\ref{thm:W_formula} in the split case. We now prove a key result towards this formula. The proof of Proposition~\ref{prop:W_formula2} will be ended in \S\ref{subsubsec:end_proof_formula_2} once we define the representation $\mathcal{R}_{\overline{\mu}}$. Recall that $\phi^\times$ and $\phi^1$ were defined in \S\ref{subsubsec:spheri_vectors}. 

\begin{prop}
\label{prop:times_version}
    Assume that we are in the split case. For every $\lambda \in \Lambda^-$ and $(\chi,\eta)$ we have
    \begin{equation*}
        ( \Phi^\circ_{\chi, \eta}(\lambda),\phi^\times )_{\nu}=\Delta_{\mathrm{U}(W)} \sum_{w \in W_G} \mathbf{b}^\times(w_V \chi, w_W \eta) \mathbf{d}_V(w_V\chi) \mathbf{d}_W(w_W \eta) \left( \left(w_0 w_V \chi \boxtimes w_W \mu\eta\right) \delta_{B^+}^{-\frac{1}{2}} \right)(\lambda),
    \end{equation*}
    where
    \begin{equation}
        \label{eq:b_times_defi} \mathbf{b}^\times(\chi,\eta)=\prod_{i+j \leq r'+1} L_F^{-1}(\frac{1}{2},\chi_i \eta_j) \prod_{i+j > r'+1}  L_F^{-1}(\frac{1}{2},\chi_i^{-1} \eta_j^{-1}).
    \end{equation}
\end{prop}

\begin{proof}
    For $(\chi,\eta)$ in general position, using Lemma~\ref{lem:spheri_vectors} we see that, as in \eqref{eq:spheri_scalar}, we have
\begin{equation*}
    ( R(w_0 1_{I_V {}^{w_0} \lambda_V I_V}) \otimes R(1_{I_W \lambda_W I_W}) \Phi_{\chi,\eta}^I (1) , \phi^\times )_\nu = \vol(I_V {}^{w_0} \lambda_V I_V) \vol(I_W \lambda_W I_W) ( \Phi_{\chi,\eta}^I(\lambda_V,\lambda_W), \phi^\times )_\nu.
\end{equation*}
Because $\phi^1 \in \nu_{\mu,\psi}^{K_W}$ we know that $S_{\eta}(\Psi_{\mu \eta} \otimes \phi^1)=\Psi_{\eta }^{1,J}$. As $\Delta_{T_W}'=\Delta_{T_W}$, $\vol(T_W^0)=1$ and $\vol(I_J^1)=q^{-m}\vol(I_W)$, we see by \eqref{eq:phi_times} and Lemma~\ref{lem:alpha_base_value} that
\begin{equation*}
    \mathcal{L}_{\chi,\eta}(R(w_0)\Psi_{1,\chi} \otimes S_\eta(\Psi_{\mu \eta}^\circ \otimes \phi^\times))= q^{m}\mathcal{L}_{\chi,\eta}(R(\mu_+^* w_0) \Psi_{1,\chi} \otimes \Psi_{1,\eta }^{1,J})=\vol(I_V)\vol(I_W).
\end{equation*}
We conclude that Proposition~\ref{prop:times_version} holds for $(\chi,\eta)$ in general position as in Proposition~\ref{prop:phi_0_calc}. For general $(\chi,\eta)$, we use Lemma~\ref{lem:regular}.
\end{proof}

\section{Unfolding of tempered periods}

\label{section:unfolding}
In this section, always assume that $r \geq 1$ and temporarily leave the unramified setting: $E/F$ is any quadratic extension of local fields of characteristic zero, or $E=F \times F$. All the other objects are as described in \S\ref{subsec:notations}, except that contrary to the situation of \S\ref{subsubsec:V_W} we only assume that $W^\perp$ is split. This means that exist two free families $(v_1, \hdots, v_r)$ and $(v_1^*, \hdots, v_r^*)$ in $V$ such that, with $X=\mathrm{Vect}(v_1, \hdots, v_r)$ and $X^*=\mathrm{Vect}(v_1^*, \hdots, v_r^*)$, we have $W^\perp=X \oplus X^*$ and for every $1 \leq i,j \leq r$, $\langle v_i, v_j \rangle_V = \langle v_i^*, v_j^* \rangle_V=0$ and 
$\langle v_i, v_j^* \rangle_V=\delta_{i,j}$.

\subsection{Preliminaries}
\subsubsection{}
For $f$ and $g$ two positive functions on a set $X$, we write $f(x) \ll g(x)$ if there exists $C>0$ such that for all $x \in X$ we have $f(x) \leq Cg(x)$. We write $\ll_y$ if the constant $C$ depends on some additional datum $y$.

For $\mathbb{G}$ the $F$-points of a linear algebraic group over $F$, let $\varsigma$ be a (class of) logarithmic height function on $\mathbb{G}$, as in~\cite[Section~1.2]{BP2}. Note that if $\mathbb{G}' \leq \mathbb{G}$ we may take $\varsigma_{| \mathbb{G}'}$ as the logarithmic height function on $\mathbb{G}'$, hence the absence of reference to the group in the notation. Let $\Xi^\mathbb{G}$ be the Harish-Chandra special spherical function on $\mathbb{G}$ and $\mathcal{C}^w(\mathbb{G})$ be the weak Harish-Chandra Schwartz space (see~\cite[\S~1.5]{BP2}). It contains $C_c^\infty(\mathbb{G})$ as a dense subset. By definition, $C^w(\mathbb{G})=\cup_{d>0}C^w_d(\mathbb{G})$ where $C^w_d(\mathbb{G})$ is the space of $f \in C^\infty(\mathbb{G})$ such that $\Val{f(g)} \ll \Xi^\mathbb{G}(g) \varsigma(g)^{d}$ for $g \in \mathbb{G}$. 

Realize $\mathrm{U}(V)$ as a subgroup of $\mathrm{U}(V) \times \mathrm{U}(V)$ by the diagonal embedding. By~\cite[Lemma~II.1.5.]{Wald} in the $p$-adic case,~\cite[Proposition~31]{Va} in the Archimedean case, for all $\varepsilon >0$ we have 
\begin{equation}
\label{eq:HC_int}
    \int_{\mathrm{U}(V)}\Xi^{\mathrm{U}(V) \times \mathrm{U}(V)}(g) e^{-\varepsilon \varsigma(g)}dg < \infty.
\end{equation}

\subsubsection{}
\label{subsec:nu_defi}
In the unramified situation we have already equipped our groups $\LAG$ with a left-invariant Haar measure $dg$ in \S\ref{subsubsec:measure}. In general, we have to make choices. 

We give $F^k$ and $E^k$ the $k$-fold product of $d_\psi x$ and $d_{\psi_E}x$  the $\psi$- and $\psi_E$- autodual measure respectively, which we simply write $dx$. This applies to $X^* \cong E^r$ via our basis $(v_i^*)_{i=1}^r$. This also yields measures on the root subgroups $N_\alpha$ by the parametrizations $n_\alpha$ (see \S\ref{subsubsec:root_subgroups}). The unipotent groups $N_k$ and $N(X)$ are then given the product measure on the root subgroups. This coincides with our choices in the unramified setting.

On $G_k$, $\mathrm{U}(V)$ and $\mathrm{U}(W)$, if our data are unramified we keep $dg$, and otherwise we take any Haar measure. On $G_k$ there is another natural choice, namely $d_{\psi} g := \frac{\prod_{1 \leq i,j \leq k} d_{\psi_E} g_{i,j}}{\Val{\det g}^k}$. Let $\upsilon(G_{k}) >0$ be the quotient $dg (d_{\psi} g)^{-1}$. In the unramified situation we have $\upsilon(G_k)=\Delta_{G_k}$. For $k=r$, denote by $\psi_r$ the restriction of $\psi_U$ to $N_r$. Fourier inversion yields for $f \in C_c^\infty(G_r)$ and $g \in G_r$
\begin{equation}
\label{eq:Fourier_inversion_G}
    f(g)=\upsilon(G_{r-1})^{-1}\int_{N_{r-1} \backslash G_{r-1}} \int_{N_r} f(\gamma^{-1} n_r \gamma g) \overline{\psi}_r(n_r) dn_r d \gamma.
\end{equation}
Let $P_r=G_{r-1} \ltimes E^{r-1}$ be the mirabolic subgroup of $G_r$ fixing $v_r$ on the left equipped with $dp=dgdx$. Let $U_r$ be its unipotent radical. The isomorphism $P_r \backslash G_r \cong X^* \setminus \{0\}$ implies that for $f \in C_c^\infty(X^*)$ 
\begin{equation}
\label{eq:mirabolic_unfolding}
    \int_{X^*} f(x^*)dx^*=\frac{\upsilon(G_{r-1})}{\upsilon(G_{r})}\int_{P_r \backslash G_r} \Val{\det \gamma}f(\gamma^* v_r^*)d \gamma,
\end{equation}
where we recall that $\gamma^*$ was defined in \S\ref{subsubsec:Y_defi} as the transpose or conjugate transpose of $\gamma$.

The other groups we consider are products of the above-mentioned ones and are given the product measure.

\subsection{Fourier--Jacobi periods}
\label{subsec:FJ_p}

Denote by $\overline{\omega}_V$ the Weil representation of $\mathrm{U}(V)$ associated to $(\overline{\mu},\overline{\psi})$ defined in \S\ref{subsec:H-W}. By the mixed model described in \cite[Section~7.4]{GI}, we may realize it on $\mathcal{S}(X^*) \hat{\otimes} \overline{\nu}_{\mu,\psi}$, where $\mathcal{S}(X^*)$ is the Schwartz space on $X^*$ and $\hat{\otimes}$ is the completed tensor product. We can identify any $\Phi \in \overline{\omega}_V$ with a $\overline{\nu}_{\mu,\psi}$-valued function on $X^*$, or a $\cc$-valued function on $X^* \oplus Y_+^*$. Define
\begin{equation*}
    \Phi_{Y_+^*}:=\Phi(v_r^*) \in \overline{\nu}_{\mu,\psi}.
\end{equation*}
We equip $\overline{\omega}_V$ with the inner product $(\Phi,\Phi')_{\overline{\omega}} = \int_{X^*} (\Phi(x^*),\Phi'(x^*))_{\overline{\nu}} dx^*$. By \eqref{eq:a_action}, we have
\begin{equation}
\label{eq:omega_action}
    \overline{\omega}_V(\gamma) \Phi=\overline{\mu}\Val{.}^{\frac{1}{2}}(\gamma)\Phi(\gamma^* .), \; \gamma \in G_r.
\end{equation}
By~\cite[Lemma~20.1]{BLX}, there exists $\varepsilon >0$ such that 
\begin{align}
    \Val{(\overline{\omega}_V(g_V) \Phi_1, \Phi_2)_{\overline{\omega}}} &\ll_{\Phi_1,\Phi_2} e^{- \varepsilon \varsigma(g_V)}, \; g_V \in \mathrm{U}(V), \; \Phi_1, \Phi_2 \in \overline{\omega}_V, \label{eq:omega_estimate}     \\
     \Val{(\overline{\nu}_{\mu,\psi}(g_W h) \phi_1, \phi_2)_{\overline{\nu}}} &\ll_{\phi_1,\phi_2} e^{- \varepsilon \varsigma(g_W)}, \; g_W \in \mathrm{U}(W), \; h \in \mathbb{H}(\mathbb{W}), \; \phi_1, \phi_2 \in \overline{\nu}_{\mu,\psi}. \label{eq:nu_estimate}
\end{align}
Moreover, by~\cite[Lemma~20.1]{BLX} the linear forms 
\begin{equation*}
f \in C_c^\infty(G) \mapsto \int_{H} f(h) (\overline{\nu}_{\mu,\psi}(h) \phi_1, \phi_2)_{\overline{\nu}} dh, \text{ and } f \in C_c^\infty(G_r) \mapsto \int_{N_r} f(n_r) \psi_r(n_r) dn_r,
\end{equation*}
extend by continuity to $C^w(G)$ and $C^w(G_r)$ respectively, where we recall that $G=\mathrm{U}(V) \times \mathrm{U}(W)$. We denote them by $\mathcal{P}_H(f \otimes \phi_1 \otimes \phi_2)$ and $\mathcal{P}_{N_r}(f)$. We also introduce
\begin{equation*}
     \mathcal{P}_{\mathrm{U}(V)}(f\otimes \Phi_1 \otimes \Phi_2):= \int_{\mathrm{U}(V)} f(g) (\overline{\omega}_V(g) \Phi_1, \Phi_2)_{\overline{\omega}}dg, \; f \in \mathcal{C}^w(\mathrm{U}(V)\times \mathrm{U}(V)).
\end{equation*}
This integral converges absolutely by \eqref{eq:HC_int} and \eqref{eq:omega_estimate}.

Recall that $L=\mathrm{U}(V) \times (\mathrm{U}(W) \times G_r)$ and $H^L=H \times N_r \subset L$ (\S\ref{subsubsec:L}). Let $\overline{\nu}^L$ be the representation of $H^L$ defined by 
\begin{equation*}
    \overline{\nu}^L:= \overline{\nu}_{\mu,\psi} \boxtimes \psi_{r}.
\end{equation*}
Recall that $P(X)$ and $N(X)$ were defined in \S\ref{subsubsec:L}. It follows from \cite[Section~7.4]{GI} that for $n_r, n_r' \in N_r, \; n \in N(X), \; g_W \in \mathrm{U}(W)$, and $ g_V \in \mathrm{U}(V)$
\begin{equation}
\label{eq:nu_relation}
    \psi_{r}(n_r') \overline{\psi}_r(n_r) (\overline{\omega}_V(n_r' ng_Wg_V) \Phi)_{Y_+^*}=\overline{\nu}^L(n_rng_W,n_r')(\overline{\omega}_V(g_V) \Phi)_{Y_+^*}.
\end{equation}
For $\phi_1, \phi_2 \in \nu^L$, set
\begin{equation*}
    \mathcal{P}_{H^L}(f \otimes \Phi_1 \otimes \Phi_2) := \int_{H^L} f(h) (\overline{\nu}^L(h) \phi_1, \phi_2)_{\overline{\nu}} dh, \; f \in C_c^{\infty}(L),
\end{equation*}
which we extend by continuity to $C^w(L)$. Note that $\mathcal{P}_{H^L}=\mathcal{P}_H \otimes \mathcal{P}_{N_r}$.

For $(\mathbb{G},\mathbb{H}) \in \{ (\mathrm{U}(V) \times \mathrm{U}(V), \mathrm{U}(V)), (G,H),(L,H^L) \}$, if $\sigma$ is a tempered representation of $\mathbb{G}$ equipped with an invariant inner product $(.,.)$, then for every $\varphi_1, \varphi_2 \in \sigma$ the map $c_{\varphi_1,\varphi_2} : g \mapsto ( \sigma(g) \varphi_1, \varphi_2)$ belongs to $\mathcal{C}^w(\mathbb{G})$. We set for $\Phi_i \in \overline{\omega}_V,\overline{\nu}_{\mu,\psi}$ or $\overline{\nu}^L$
\begin{equation}
\label{eq:integral_of_coefficients}
    \mathcal{P}_{\mathbb{H}}(\varphi_1 \otimes \Phi_1, \varphi_2 \otimes \Phi_2):= \mathcal{P}_{\mathbb{H}}(c_{\varphi_1,\varphi_2} \otimes \Phi_1 \otimes \Phi_2).
\end{equation}
If $\varphi_1=\varphi_2$ and $\Phi_1=\Phi_2$, we will simply write $\mathcal{P}_{\mathbb{H}}(\varphi \otimes \Phi)$ for $\mathcal{P}_{\mathbb{H}}(\varphi_1 \otimes \Phi_1, \varphi_2 \otimes \Phi_2)$.

\subsection{Relations between local periods}

We embed $P(X)$ in $L$ by inclusion on the $\mathrm{U}(V)$ and projection on $\mathrm{U}(W) \times G_r$. The measure $dp$ is left-invariant.

\begin{lem}
\label{lem:tempered_fourier}
    For all $f \in \mathcal{C}^w(L)$, $\Phi_1, \Phi_2 \in \nu^L$ we have
    \begin{align}
    \label{eq:Fourier_unfolding1}
        &\int_{P(X)} f(p) (\overline{\omega}_V(p) \Phi_1, \Phi_2)_{\overline{\omega}} \delta_{P(X)}(p)^{\frac{1}{2}} d p \nonumber \\
        = &\frac{1}{\upsilon(G_r)} \int_{(N_r \backslash G_r)^2} \mathcal{P}_{H^L}\left(R(g_1)L(g_2)f \otimes  (\overline{\omega}_V(g_1) \Phi_1)_{Y_+^*} \otimes (\overline{\omega}_V(g_2) \Phi_2)_{Y_+^*} \right) \delta_{P(X)}(g_1 g_2)^{-\frac{1}{2}} dg_1 dg_2,
    \end{align}
\end{lem}

\begin{proof}
    Fix $\Phi_1$ and $\Phi_2$. By repeating the proof of~\cite[Proposition~8.6.2.1]{BPC} using \eqref{eq:omega_action}, \eqref{eq:omega_estimate} and \eqref{eq:nu_estimate}, we see that both sides are absolutely convergent and define continuous linear functionals on $\mathcal{C}^w(L)$. We are reduced to proving (\ref{eq:Fourier_unfolding1}) in the case $f=f_G \otimes f_r \in C_c^\infty(G) \otimes C_c^\infty(G_r)$. The integral on the left becomes
    \begin{equation}
        \label{eq:Fourier_unfolding0}
        \int_{G_r} f_r(g_1) \delta_{P(X)}(g_1)^{\frac{1}{2}} \int_{\mathrm{U}(W)} \int_{N(X)} f_G(g_1hn) (\overline{\omega}_{V}(g_1hn) \Phi_1, \Phi_2)_{\overline{\omega}} dn dh dg_1.
    \end{equation}
    For fixed $g_1 \in G_r$, write 
    \begin{align}
    \label{eq:Fourier_unfolding2}
        &\int_{\mathrm{U}(W) \times N(X)} f_G(g_1hn) (\overline{\omega}_{V}(g_1hn) \Phi_1, \Phi_2)_{\overline{\omega}} dn dh \nonumber \\
        = & \frac{\upsilon(G_{r-1})}{\upsilon(G_r)}\int_{\mathrm{U}(W) \times N(X)} \int_{ P_r \backslash G_r} f_G(g_1hn) \left((\overline{\omega}_V(g_2 g_1hn) \Phi_1)_{Y_+^*}, (\overline{\omega}_V(g_2)\Phi_2)_{Y_+^*} \right)_{\overline{\nu}} dg_2 dn dh,
    \end{align}
    where we have used \eqref{eq:mirabolic_unfolding} and \eqref{eq:omega_action}. By applying the Fourier inversion formula (\ref{eq:Fourier_inversion_G}) to $g\in G_r \mapsto f_G(g_2^{-1}gg_2 g_1hn)$ for $g=1$ we see that the inner integral in \eqref{eq:Fourier_unfolding2} is $\upsilon(G_{r-1})^{-1}$ times
    \begin{equation}
    \label{eq:before_collapse}
        \int_{ P_r \backslash G_r} \int_{N_{r-1} \backslash G_{r-1}} \int_{N_r} f_G(g_2^{-1} \gamma^{-1} n_r \gamma g_2 g_1 hn) \overline{\psi}_r(n_r) \left((\overline{\omega}_V(g_2 g_1hn) \Phi_1)_{Y_+^*}, (\overline{\omega}_V(g_2)\Phi_2)_{Y_+^*}\right)_{\overline{\nu}} dn_r d\gamma dg_2.
        \end{equation}
        
    Let us prove that 
    \begin{equation}
    \label{eq:collapse1}
        \int_{ P_r \backslash G_r} \int_{N_{r-1} \backslash G_{r-1}} \Val{\int_{N_r} f_G(g_2^{-1} \gamma^{-1} n_r \gamma g_2 g_1 hn) \overline{\psi}_r(n_r) dn_r\left((\overline{\omega}_V(g_2 g_1hn) \Phi_1)_{Y_+^*}, (\overline{\omega}_V(g_2)\Phi_2)_{Y_+^*}\right)_{\overline{\nu}}}  d\gamma dg_2
    \end{equation}    
    is finite. By the description of the restriction of $\overline{\omega}_V$ to $G_r$ in \eqref{eq:a_action} and by \eqref{eq:mirabolic_unfolding} we know that 
    \begin{equation*}
        \int_{P_r \backslash G_r} \Val{\left((\overline{\omega}_V(g_2 g_1hn) \Phi_1)_{Y_+^*}, (\overline{\omega}_V(g_2)\Phi_2)_{Y_+^*}\right)_{\overline{\nu}}}dg_2 < \infty.
    \end{equation*}
    It is therefore enough to show that 
    \begin{equation}
    \label{eq:collapse2}
        g_2 \in G_r \mapsto \int_{N_{r-1} \backslash G_{r-1}} \Val{\int_{N_r} f_G(g_2^{-1} \gamma^{-1} n_r \gamma g_2 g_1 hn) \overline{\psi}_r(n_r) dn_r}d\gamma
    \end{equation}
    is bounded. By the Iwasawa decomposition, every $g_2 \in G_r$ can be written as $g_2=g_{r-1} u_r t k$ where $g_{r-1} \in G_{r-1}$, $u_r \in U_r$, $t=e_r^\vee(s)$ for $s \in E^\times$ (i.e. the diagonal matrix with only $1$'s and $s$ at the last coordinate) and $k \in K_V \cap \GL_r$. Note that $\GL_{r-1}$ normalizes $U_r$, that $N_r$ is unimodular, and that $t$ commutes with $\GL_{r-1}$. For any $n_r \in N_r$ we have $\overline{\psi}_r(u_r n_r u_r^{-1})=\overline{\psi}_r(n_r)$ and $t^{-1} n_r t=t' n_r t'^{-1}$ where $t'=\prod_{i=1}^{r-1} e_i^\vee(s)$. Then $t'$ is in the center of $G_{r-1}$. Therefore, we see that 
      \begin{equation*}
       \eqref{eq:collapse2}=\int_{N_{r-1} \backslash G_{r-1}} \Val{\int_{N_r} f_G(k^{-1} \gamma^{-1} n_r \gamma k g_1 hn) \overline{\psi}_r(n_r) dn_r}d\gamma.
    \end{equation*}
    This proves that \eqref{eq:collapse2} is bounded and therefore that \eqref{eq:collapse1} is finite.

    We may therefore collapse the two outer integrals in \eqref{eq:before_collapse} to see that it is equal to
    
    \begin{align}
    \label{eq:after_collapse}
        &\int_{ P_r \backslash G_r} \int_{N_{r} \backslash P_{r}} \int_{N_r} f_G(g_2^{-1} p^{-1} n_r p g_2 g_1 hn) \overline{\psi}_r(n_r) \left((\overline{\omega}_V(p g_2 g_1hn) \Phi_1)_{Y_+^*}, (\overline{\omega}_V(p g_2)\Phi_2)_{Y_+^*}\right)_{\overline{\nu}} dn_r dp dg_2 \nonumber \\
        =&\int_{ N_r \backslash G_r} \int_{N_r} f_G(g_2^{-1}  n_r g_2 g_1 hn) \overline{\psi}_r(n_r) \left((\overline{\omega}_V(g_2 g_1hn) \Phi_1)_{Y_+^*}, (\overline{\omega}_V(g_2)\Phi_2)_{Y_+^*}\right)_{\overline{\nu}} dn_r dg_2,
        \end{align}
    where we have again used \eqref{eq:omega_action} to make $p$ appear in $\overline{\omega}_V$ and to integrate on $N_r \backslash P_r$ rather than $N_{r-1} \backslash G_{r-1}$.

    Note that in \eqref{eq:Fourier_unfolding0}, the integrals over $g_1$, $n$ and $h$ are all over compacts. We can therefore plug \eqref{eq:after_collapse} back into \eqref{eq:Fourier_unfolding0} and do the change of variables $g_1 \mapsto g_2^{-1} g_1$ to see that \eqref{eq:Fourier_unfolding0} is
    \begin{align*}
         \upsilon(G_r)^{-1}\int_{(N_r \backslash G_r)^2} \int_{\mathrm{U}(W) \times N(X)} \int_{N_r^2} &f_r(g_2^{-1} n_r' g_1) \delta_{P(X)}(g_1 g_2^{-1})^{\frac{1}{2}} f_G(g_2^{-1}n_rn_r'g_1 hn) \overline{\psi}_r(n_r)  \\
        &\left((\overline{\omega}_V(n_r' g_1 hn) \Phi_1)_{Y_+^*},(\overline{\omega}_V(g_2)\Phi_2)_{Y_+^*}\right)_{\overline{\nu}} 
        dn_r dn_r' dn dh dg_1 dg_2.
    \end{align*}
    The changes of variables $n_r \mapsto n_r(n_r')^{-1}$ and $g_1nh \mapsto nhg_1$ (which gives $\delta_{P(X)}(g_1)^{-1})$ and \eqref{eq:nu_relation} give
     \begin{align*}
        \upsilon(G_r)^{-1}\int_{(N_r \backslash G_r)^2} \int_{\mathrm{U}(W) \times N(X)} &\int_{N_r^2} f_1(g_2^{-1} n_r' g_1) \delta_{P(X)}(g_1 g_2)^{-\frac{1}{2}} f_2(g_2^{-1}n_rnh g_1)  \\
        &\left(\overline{\nu}^L(n_rnh,n_r') ((\overline{\omega}_V(g_1) \Phi_1)_{Y_+^*}),(\overline{\omega}_V(g_2)\Phi_2)_{Y_+^*}\right)_{\overline{\nu}} 
        dn_r dn_r' dn dh dg_1 dg_2.
    \end{align*}
    This is exactly what we were after by definition of $\mathcal{P}_{H^L}$.
\end{proof}

\begin{prop}
\label{prop:tempered_computations}
    Let $\tau$, $\sigma_W$ and $\sigma_V$ be irreducible tempered representations of $G_r$, $\mathrm{U}(W)$ and $\mathrm{U}(V)$ respectively, equipped with invariant inner products. Let $\Sigma:=I_{P(X)}^{\mathrm{U}(V)} (\tau \boxtimes \sigma_W)$ be the parabolic induction equipped with its canonical inner product. Then for $\varphi^1_V, \varphi^2_V \in \sigma_V$,  $\varphi^1_\Sigma, \varphi^2_\Sigma \in \Sigma$,  and $\Phi_1, \Phi_2 \in \overline{\omega}_V$ we have
    \begin{equation}
        \mathcal{P}_{\mathrm{U}(V)}(\varphi^1_V \otimes \varphi^1_\Sigma \otimes \Phi_1, \varphi^2_V \otimes \varphi^2_\Sigma \otimes \Phi_2) \nonumber 
        =\int_{(H \backslash \mathrm{U}(V))^2} \frac{\mathcal{P}_{H^L}\left(
        (\sigma_V(h_i) \varphi^i_V \otimes \varphi^i_\Sigma(h_i) \otimes (\overline{\omega}_V(h_i) \Phi_i)_{Y_+^*})_{i=1,2}\right)}{\upsilon(G_r)}dh_i. \label{eq:to_prove_temp}
    \end{equation}
\end{prop}

\begin{proof}
By definition of the invariant inner product on $\Sigma$ we see that the LHS of \eqref{eq:to_prove_temp} is
\begin{equation}
\label{eq:tempered_unfolding}
   \int_{\mathrm{U}(V)} \int_{P(X) \backslash \mathrm{U}(V)} (\varphi^1_\Sigma(h_2 h_1), \varphi^2_\Sigma(h_2))dh_2 (\sigma_V(h_1) \varphi^1_V, \varphi^2_V) (\overline{\omega}_V(h_1) \Phi_1, \Phi_2)_{\overline{\omega}}dh_1. 
\end{equation}
By~\cite[Theorem~2]{CHH88} and~\cite[Lemma~II.1.6]{Wald} we know that
\begin{equation*}
    \int_{P(X) \backslash \mathrm{U}(V)} \Val{(\varphi^1_\Sigma(h_2 h_1), \varphi^2_\Sigma(h_2))}dh_2  \ll \Xi^{\mathrm{U}(V)}(h_1).
\end{equation*}
Moreover, we have $\Val{(\sigma_V(h_1) \varphi^1_V, \varphi^2_V)} \ll \Xi^{\mathrm{U}(V)}(h_1)$. This implies that (\ref{eq:tempered_unfolding}) is absolutely convergent by \eqref{eq:omega_estimate} and \eqref{eq:HC_int}. Then by direct computation we have
\begin{align*}
    &\mathcal{P}_{\mathrm{U}(V)}(\varphi^1_V \otimes \varphi^1_\Sigma \otimes \Phi_1, \varphi^2_V \otimes \varphi^2_\Sigma \otimes \Phi_2) \\
    =& \upsilon(G_r)^{-1} \int_{(P(X) \backslash \mathrm{U}(V))^2}\int_{(N_r \backslash G_r)^2} \delta_{P(X)}^{-1}(g_i)\mathcal{P}_{H^L}\left(
     \sigma_V(g_i h_i) \varphi_V \otimes\varphi^i_\Sigma(g_i h_i) \otimes (\overline{\omega}_V(g_ih_i) \Phi_i)_{Y_+^*}\right) dg_i dh_i \\
    =&\upsilon(G_r)^{-1} \int_{(H \backslash \mathrm{U}(V))^2}\mathcal{P}_{H^L}\left(
     \sigma_V(h_i) \varphi_V^i \otimes \varphi^i_\Sigma(h_i) \otimes(\overline{\omega}_V(h_i) \Phi_i)_{Y_+^*}\right) dh_i,
\end{align*}
where in the second equality we have applied Lemma~\ref{lem:tempered_fourier} to the map in $\mathcal{C}^w(L)$ defined by
\begin{equation*}
    (g_V,g_W,g_r) \mapsto\left(\sigma_V(g_Vh_1) \varphi^1_V, \sigma_V(h_2) \varphi^2_V \right)  \left((\tau(g_r)\otimes \sigma_W(g_W)) \varphi^1_\Sigma(h_1), \varphi^2_\Sigma(h_2)\right),
\end{equation*}
\end{proof}

\section{Unramified $L$-function for Fourier--Jacobi models}

\label{section:L}

We now go back to the unramified setting described in \S\ref{section:WS} and prove Proposition~\ref{prop:normalization} (value of $\mathcal{W}^I_{\chi,\eta}(1))$ and Theorem~\ref{thm:II} (formula for $\mathcal{P}_H(\varphi^\circ,\phi^\circ)$ in the tempered case). 

\subsection{Satake parameters for unitary groups}

\subsubsection{}

Following~\cite[Section~8.7]{BPC}, we introduce some notations on dual complex groups. Define $W_F$ to be the Weil group of $F$, and let $\mathrm{Fr} \in W_F$ be an absolute Frobenius. In the inert case, let $W_E$ be the Weil group of $E$. Note that $\mathrm{Fr}$ is sent to $c$ in the quotient $W_F / W_E \simeq \gal(E/F)$. For convenience, we write $W_E$ for $W_F$ in the split case, so that $W_F / W_E$ is trivial.

For $V_k$ a $k$-dimensional unramified skew-Hermitian space over $E/F$, we identify the dual group $\widehat{\mathrm{U}}(V_k)$ of $\mathrm{U}(V_k)$ with $\GL_{k}(\cc)$ equipped with its standard pinning. The Langlands dual group of $\mathrm{U}(V_k)$ is $ {}^L \mathrm{U}(V_k)=\GL_k(\cc) \rtimes W_F$. The Galois action factorizes through $W_F / W_E$, with $c(g)=g^\star$ in the inert case where 
\begin{equation*}
    g^\star=J_k {}^tg^{-1}J_k^{-1}, \; J_k=\begin{pmatrix}
        & & 1 \\
        &~\reflectbox{$\ddots$} & \\
        (-1)^{k-1} & &
    \end{pmatrix}.
\end{equation*}
The dual group $\widehat{G}_k$ of $G_k$ is identified with $\GL_k(\cc) \times \GL_k(\cc)$, and its Langlands dual group is ${}^L G_k=(\GL_k(\cc) \times \GL_k(\cc)) \rtimes W_F$, where the Galois action factorizes through $W_F / W_E$ with $c(g_1,g_2)=(g_2,g_1)$ in the inert case. Note that in the split case there is also an action of $\gal(E/F)=\{1,c\}$ on ${}^L G_k$ : $c$ acts by $c(g_1,g_2)=(g_2,g_1)$ on $\widehat{G}_k$ and by identity on $W_F$. For $S \in {}^L G_k$, we denote this action by $S^c$.

We write $S \mapsto S^\star$ for the automorphism of ${}^L G_k$ which is identity on $W_F$ and sends $(g_1,g_2) \in \widehat{G}_k$ to $(g_2^\star,g_1^\star)$. There is a base change embedding $\mathrm{BC} : {}^L \mathrm{U}(V_k) \to {}^L G_k$ which is given by
\begin{equation}
\label{eq:Satake_embedding}
    g \in \widehat{\mathrm{U}}(V_k) \mapsto (g,g^\star) \in \widehat{G}_k,
\end{equation}
and which is identity on $W_F$. When the context is clear, we will identify an element $S \in {}^L \mathrm{U}(V_k)$ with its image in ${}^L G_k$.

We write $(\widehat{T}_{V_K},\widehat{B}_{V_k})$ for the standard Borel pair in $\widehat{\mathrm{U}}(V_k)$, and $(\widehat{T}_k,\widehat{B}_k)$ for the one in $\widehat{G}_k$. These are subgroups of $\widehat{\mathrm{U}}(V_k)$ and $\widehat{G}_k$ respectively which are stable by the action of $W_F$. Set ${}^L T_{V_k}=\widehat{T}_{V_k} \rtimes W_F, \; 
    {}^L B_{V_k}=\widehat{B}_{V_k} \rtimes W_F, \;
    {}^L T_k=\widehat{T}_k \rtimes W_F$ and $
    {}^L B_k=\widehat{B}_k \rtimes W_F$.

Recall that $G=\mathrm{U}(V) \times \mathrm{U}(W)$ and set $\widetilde{G}:=\mathrm{U}(W) \times \mathrm{U}(W)$. Their dual groups (resp. Langlands dual groups) are $\widehat{G}=\widehat{\mathrm{U}}(V) \times \widehat{\mathrm{U}}(W), \; \widehat{\widetilde{G}}=\widehat{\mathrm{U}}(W) \times \widehat{\mathrm{U}}(W),\; ( \text{resp. } {}^L G={}^L \mathrm{U}(V) \times_{W_F} {}^L \mathrm{U}(W), \; {}^L \widetilde{G}={}^L \mathrm{U}(W) \times_{W_F} {}^L \mathrm{U}(W) ) $.

Set $B=B_V \times B_W$, $T=T_V \times T_W$, $\widetilde{B}=B_W \times B_W$ and $\widetilde{T}=T_W \times T_W$. Write $\widehat{B}=\widehat{B}_V \times \widehat{B}_W$, $\widehat{T}=\widehat{T}_V \times \widehat{T}_W$, $\widehat{\widetilde{B}}=\widehat{B}_W \times \widehat{B}_W$ and $\widehat{\widetilde{T}}=\widehat{T}_W \times \widehat{T}_W$ and set ${}^L T=\widehat{T} \rtimes W_F, \; {}^L \widetilde{T}=\widehat{\widetilde{T}} \rtimes W_F$. 

Recall that $P(X)$ is the parabolic subgroup of $\mathrm{U}(V)$ stabilizing $X$, with Levi decomposition $P(X)=M(X)N(X)$ where $M(X)$ is the Levi subgroup stabilizing $X^*$. Set $M=M(X) \times \mathrm{U}(W) \subset G$. The corresponding $L$-groups are ${}^L M(X)$ and ${}^L M$ which are identified with the corresponding Levi subgroups in ${}^L \mathrm{U}(V)$ and ${}^L G$ respectively. Note that ${}^L M(X)$ is isomorphic to ${}^L G_r \times_{W_F} {}^L \mathrm{U}(W)$ via the map which is identity on $W_F$ and 
\begin{equation*}
    \begin{pmatrix}
        g_r^{(1)} & & \\
        & g_W & \\
        & & g_r^{(2)}
    \end{pmatrix} \mapsto ((g_r^{(1)},g_r^{(2)\star}),g_W).
\end{equation*}
on $\widehat{M}(X)$. For $S \in {}^L M(X)$, we denote by $(S^{(r)},S^{(m)}) \in {}^L G_r \times_{W_F} {}^L \mathrm{U}(W)$ its image by this morphism. If $S=(S_V,S_W) \in {}^L T$, we have $S_V \in {}^L M(X)$ and we set $\widetilde{S}:=(S_V^{(m)},S_W) \in {}^L \widetilde{G}$.

For $\LAG \in \{G,\widetilde{G},\mathrm{U}(V_k),G_k,M,M(X)\}$, let $W_{\mathbb{G}}$ be the Weyl group $\mathrm{Norm}_{\widehat{\mathbb{G}}}({}^L \mathbb{T})/\widehat{\mathbb{T}}$, where $\widehat{\mathbb{T}} \subset \widehat{\mathbb{G}}$ is the standard maximal torus. We have isomorphisms $W_{\mathrm{U}(V)}=W_V$, $W_{\mathrm{U}(W)}=W_W$ and $W_G=W_V \times W_W$. These groups act by conjugation on ${}^L T_V$, ${}^L T_W$, and 
${}^L T$ respectively, and we denote these actions by $wS$.

For a complex Lie groups ${}^L \mathbb{G}$ and a subgroup ${}^L \mathbb{Q}$ with respective identity components $\widehat{\mathbb{G}}$ and $\widehat{\mathbb{Q}}$, set 
\begin{equation*}
    D_{\widehat{\mathbb{G}}/\widehat{\mathbb{Q}}}(S)=\det\left(1-\mathrm{Ad}(S) \; | \; \mathrm{Lie}(\widehat{\mathbb{G}})/\mathrm{Lie}(\widehat{\mathbb{Q}})\right), \; S \in {}^L \mathbb{Q}.
\end{equation*}

For any $k$, the choice of the Borel pair $(T_k,B_k)$ in $G_k$ allows us to identify $\Lambda_k$ the group of cocharacters of $T_k$ with the group of characters of ${}^L T_k$ trivial on $W_F$. We will denote by $\lambda_k \mapsto \chi_{\lambda_k}$ this identification. For $\lambda_k \in \Lambda_k^+$ (the cone of positive cocharacters with respect to $B_k$), we write $ch_{\lambda_k}$ for the character of the irreducible representation of ${}^L G_k$ with highest weight $\chi_{\lambda_k}$. If $S_k \in {}^L T_k$, set $\overline{\mu} S_k=((I_k,-I_k),\mathrm{id}) .S_k$ in the inert case, and $\overline{\mu} S_k=((\overline{\mu}(\varpi)I_k,\mu(\varpi)I_k),\mathrm{id}) .S_k$ in the split case. Therefore, $\chi_{\lambda_k}(\overline{\mu} S_k)=\overline{\mu}(\lambda_k)\chi_{\lambda_k}(S_k)$.

\subsubsection{}
\label{eq:Rankin--Selberg_rep}
For $k, l \in \nn$, we define the representation
\begin{equation*}
  (S_k,S_l) \in  {}^L (G_k \times G_l) = {}^L G_k \times_{W_F} {}^L G_l  \mapsto S_k \stackrel{\mathrm{I}}{\otimes} S_l \in  \GL(\cc^k \otimes \cc^l \oplus \cc^k \otimes \cc^l),
\end{equation*}
which sends $((g_k^{(1)},g_k^{(2)}),(g_l^{(1)},g_l^{(2)})) \in \widehat{G}_k \times \widehat{G}_l$ to $g_k^{(1)} \otimes g_l^{(1)} \oplus g_k^{(2)} \otimes g_l^{(2)}$, factorizes through $W_F / W_E$ and sends $c$ to the operator that swaps the two copies of $\cc^k \otimes \cc^l$ in the inert case. 

By~\cite[Appendix~A]{LM}, for any quasisplit reductive group $\LAG$ over $F$, we have an isomorphism $H^1(W_F,Z(\widehat{\LAG}))\simeq \Hom(\LAG,\cc^\times)$. Let $\overline{\mu}_k$ be the character $(g_l,g_k) \mapsto \overline{\mu}(\det g_k)$ of $G_l \times G_k$. It corresponds to a class $[z_{\overline{\mu}}] \in H^1(W_F,Z(\widehat{G_l \times G_k}))$, and we choose a representative cocycle $z_{\overline{\mu}}$. Set 
\begin{equation*}
    a_{\overline{\mu}} : (\hat{g},w) \in {}^L(G_l \times G_k) \mapsto (z_{\overline{\mu}}(w) \hat{g},w) \in {}^L(G_l \times G_k),
\end{equation*}
and set $\stackrel{\mathrm{I}}{\otimes}_{\overline{\mu}}:=\stackrel{\mathrm{I}}{\otimes} \circ a_{\overline{\mu}}$. Note that if $S_l \in \widehat{G_l} . \mathrm{Fr}$ and $S_k \in \widehat{G_k} . \mathrm{Fr}$, we have $S_l \stackrel{\mathrm{I}}{\otimes}_{\overline{\mu}} S_k=S_l \stackrel{\mathrm{I}}{\otimes} \overline{\mu} S_k$. Composing $\stackrel{\mathrm{I}}{\otimes}_{\overline{\mu}}$ with the embedding $\mathrm{BC}$ of \eqref{eq:Satake_embedding}, we obtain representations
\begin{equation*}
    {}^L (\mathrm{U}(V) \times G_k) \to \GL(\cc^n \otimes \cc^k \oplus \cc^n \otimes \cc^k), \; \; \; {}^L (\mathrm{U}(W) \times G_k) \to \GL(\cc^m \otimes \cc^k \oplus \cc^m \otimes \cc^k),
\end{equation*}
which we denote again by $\stackrel{\mathrm{I}}{\otimes}_{\overline{\mu}}$ for simplicity, and
\begin{equation*}
    \mathcal{R}_{\overline{\mu}} : {}^L G \to \GL(\cc^n \otimes \cc^m \oplus \cc^n \otimes \cc^m), \; \; \; \widetilde{\mathcal{R}}_{\overline{\mu}} : {}^L \widetilde{G} \to \GL(\cc^m \otimes \cc^m \oplus \cc^m \otimes \cc^m).
\end{equation*}
Consider the subspace
\begin{align*}
    \mathcal{V}_{-}=\left\langle
    (e_i \otimes e_j,0) \; \middle| \; \substack{ 1 \leq i \leq n \\ 1 \leq j \leq m \\ i+j>r'+1} 
    \right\rangle \oplus \left\langle
    (0,e_i \otimes e_j) \; \middle| \; \substack{ 1 \leq i \leq n \\ 1 \leq j \leq m \\ i+j>r'+1} 
    \right\rangle, \\
       \left( \text{resp. }\widetilde{\mathcal{V}}_{-}=\left\langle
    (e_i \otimes e_j,0) \; \middle| \; \substack{ 1 \leq i \leq m \\ 1 \leq j \leq m \\ i+j>m+1} 
    \right\rangle \oplus \left\langle
    (0,e_i \otimes e_j) \; \middle| \; \substack{ 1 \leq i \leq m \\ 1 \leq j \leq m \\ i+j>m+1} 
    \right\rangle \right),
\end{align*}
where $(e_i)_{i=1}^k$ denotes the canonical basis of $\cc^k$ for any $k$. It is stable by ${}^L T$ (resp. by ${}^L \widetilde{T}$) and we set $\mathcal{R}_{-,\overline{\mu}}(S):=\mathcal{R}_{\overline{\mu}}(S)_{| \mathcal{V}_-}$ for $S \in {}^L T$ (resp. $\widetilde{\mathcal{R}}_{-,\overline{\mu}}(\widetilde{S}):=\mathcal{R}_{\overline{\mu}}(\widetilde{S})_{| \widetilde{\mathcal{V}}_-}$ for $\widetilde{S} \in {}^L \widetilde{T}$).

\subsubsection{}
\label{subsubsec:symplectic}
We can equip $\cc^n \otimes \cc^m \oplus \cc^n \otimes \cc^m$ with the symplectic pairing
\begin{equation*}
    \langle (u_n \otimes u_m ,v_n \otimes v_m), (u_n' \otimes u_m',v_n' \otimes v_m') \rangle := ({}^t u_n J_n^{-1} v_n')({}^t u_m J_m^{-1} v_m')-({}^t u_n' J_n^{-1} v_n)({}^t u_m' J_m^{-1} v_m).
\end{equation*}
The representation $\mathcal{R}_{\overline{\mu}}$ is symplectic. The subspace $\mathcal{V}_-$ is not a Lagrangian, but the subspace $\mathcal{Y}_-:=\mathcal{V}_- \oplus \left\langle
    (e_{j+r} \otimes e_{m-j+1},0) \; \middle| \; 1 \leq j \leq m 
    \right\rangle$ is. In the split case it is stable by the action of ${}^L T$. The Lagrangian subspace $\mathcal{Y}$ appearing in Proposition~\ref{prop:W_formula2} is $\mathcal{Y}=\mathcal{R}_{\overline{\mu}}(w_{0,V})\mathcal{Y}_-$. It is also stable by ${}^L T$, and we denote by $\mathcal{Y}_{\overline{\mu}}$ this representation of ${}^L T$.
\subsection{Proof of Proposition~\ref{prop:normalization}}
\label{section:proof_normalization}

\subsubsection{}

Let $\chi$ and $\eta$ be unramified characters of $T_V$ and $T_W$ respectively in general position as in \S\ref{section:Formula}. They correspond to semisimple conjugacy classes in ${}^L \mathrm{U}(V)$ and ${}^L \mathrm{U}(W)$ which we identify with representatives $S_V \in \widehat{T}_V.\mathrm{Fr}$ and $S_W \in \widehat{T}_W.\mathrm{Fr}$. Set $S=(S_V,S_W) \in {}^L T$. Recall that we have defined the regular functions $\mathbf{d}_V$, $\mathbf{d}_W$ and $\mathbf{b}$ in \eqref{eq:d_defi} and \eqref{eq:b_defi}. Let $w_{0,G}$ be the longest element in $W_G$. Then we have
\begin{equation}
\label{eq:Satake_reform}
    \mathbf{d}_V(\chi)\mathbf{d}_W(\eta)=D_{\widehat{G}/\widehat{B}}(w_{0,G} S)^{-1}, \text{ and } \mathbf{b}(\chi,\overline{\mu}\eta)=\det(1-q_F^{-\frac{1}{2}} \mathcal{R}_{-,\overline{\mu}}
        (w_{0,G}S)).
\end{equation}

\subsubsection{Proof of Proposition~\ref{prop:normalization} for \texorpdfstring{$r=0$}{r=0}}

We now write the proof of Proposition~\ref{prop:normalization} assuming first that $r=0$. We have to show that 
\begin{equation*}
    \mathcal{W}_{\chi,\overline{\mu}\eta}^I(1)=\Delta_{T_W}^{',-1} \Gamma(\chi,\overline{\mu}\eta) \Delta_{T_W} \Delta_{\mathrm{U}(W)}^{-1}.
\end{equation*}
By Proposition~\ref{prop:phi_0_calc} and \eqref{eq:Satake_reform}, this amounts to proving that for $S \in \widehat{T}.\mathrm{Fr}$ in general position we have 
\begin{equation}
\label{eq:normalization_computation}
    \sum_{w \in W_G} \frac{\det(1-q_F^{-\frac{1}{2}} \mathcal{R}_{-,\overline{\mu}}
        (wS))}{D_{\widehat{G}/\widehat{B}}(wS)}=\frac{\Delta_{T_W}}{ \Delta_{\mathrm{U}(W)}}.
\end{equation}
We claim that the LHS of \eqref{eq:normalization_computation} is constant in $S$. Indeed, note that the composition $S \mapsto S .\mathrm{Fr}\mapsto \det(1-q_F^{-\frac{1}{2}} \mathcal{R}_{-,\overline{\mu}}(S.\mathrm{Fr}))$ is a linear combination of $W_F$-invariant characters of $\widehat{T}$. Let $\lambda \in X^*(\widehat{T})^{W_F}$ be such a character. By the natural isomorphism $X^*(\widehat{T}) \cong \zz^n \times \zz^n$, we can write the coordinates $\lambda=(\lambda_{V,1}, \hdots, \lambda_{V,n}, \lambda_{W,1}, \hdots, \lambda_{W,n})$. For each $1 \leq i \leq n$ we have $i-n \leq \lambda_{V,i}, \lambda_{W,i} \leq i-1$. Let $\rho \in \frac{1}{2} X^*(\widehat{T})$ be the half-sum of the positive roots in $\widehat{T}$ with respect to $\widehat{B}$. The coordinates of $\lambda + \rho$ are all integers or all half-integers between $\frac{1-n}{2}$ and $\frac{n-1}{2}$, so that we have the following alternative.
        \begin{itemize}
            \item Either $\lambda$ is singular, i.e. there exists a coroot $\alpha^\vee$ such that $\langle \lambda + \rho, \alpha^\vee \rangle =0$, in which case $\sum_{w \in W_G} \frac{\lambda(wS)}{D_{\widehat{G}/\widehat{B}}(wS)}=0$ by the Weyl character formula \cite[Proposition~A.0.2.1]{BPC}.
            \item Either $\lambda + \rho$ is conjugated to $\rho$ under $W_G$ and $S \mapsto \sum_{w \in W_G} \frac{\lambda(wS)}{D_{\widehat{G}/\widehat{B}}(wS)}$ is constant.
        \end{itemize}
    This shows that \eqref{eq:normalization_computation} is constant in $S$, and the claim follows from elementary computations taking $S$ corresponding to the characters $(\chi,\eta)=w_{0,G}( \delta_{B_V}^{-\frac{1}{2}},\mu \delta_{B_V}^{-1}\delta_{B_J}^{\frac{1}{2}})$ in the inert case and $w_{0,G}(\delta_{B_V}^{-\frac{1}{2}},\mu \delta_{B_V}\delta_{B_J}^{-\frac{1}{2}})$ in the split case, noting that with this choice $\det(1-q_F^{-\frac{1}{2}} \mathcal{R}_{-,\overline{\mu}}
        (wS))$ is zero unless $w=1$ in the split or inert odd case, and $w=1$ or $w$ is the element corresponding to $(1,(1,n)) \in W_G \subset \mathfrak{S}_n \times \mathfrak{S}_n $ in the inert even case.

\subsubsection{Proof of Proposition~\ref{prop:normalization} for \texorpdfstring{$r\geq 1$}{r>0}}

    We now prove Proposition~\ref{prop:normalization} for $r \geq 1$. We first state two lemmas that will also be used in the proof of Proposition~\ref{prop:L_unfold}. Denote by $\Lambda_r^{++} \subset \Lambda_r^+$ the subcone of cocharacters that are dominant with respect to $B_{r+1}$ through the embedding $g \in G_r \mapsto \begin{pmatrix}
    g & \\
    & 1
\end{pmatrix} \in G_{r+1}$. Note that $\Lambda_r^{++} \subset \Lambda_V^+$. 

    \begin{lem}
    \label{lem:unfolding_W}
        For $\lambda_r \in \Lambda_r^{++}$ we have
        \begin{equation}
        \label{eq:W_unfolding}
            \mathcal{W}_{\chi,\overline{\mu}\eta}^I(\lambda_r)=\Delta_{T_W}^{',-1} \Gamma(\chi,\overline{\mu} \eta) \frac{\Delta_{T_W}}{\Delta_{\mathrm{U}(W)}}\sum_{w \in W_V} \frac{\det(1-q_F^{-\frac{1}{2}} (w S_V)^{(r)\star} \stackrel{\mathrm{I}}{\otimes}_{\overline{\mu}}S_W)}{D_{\widehat{\mathrm{U}}(V)/\widehat{B}_V}(wS_V)}ch_{\lambda_r}\left((wS_V)^{(r)}\right)\delta_{B_V}(\lambda_r)^\frac{1}{2}.
        \end{equation}
    \end{lem}

    \begin{proof}
        For any $S_W$ the map $S_V \in \widehat{T}_V . \mathrm{Fr} \mapsto \det(1-q_F^{-\frac{1}{2}} ( S_V)^{(r)\star} \stackrel{\mathrm{I}}{\otimes}_{\overline{\mu}}S_W)ch_{\lambda_r}\left((S_V)^{(r)}\right)$ is invariant under the action of $W_{M(X)}$. By the Weyl character formula \cite[Proposition~A.0.2.1]{BPC}, we have
        \begin{equation*}
            \sum_{w \in W_{M(X)}} D_{\widehat{\mathrm{U}}(V) / \widehat{B}_V}(w S_V)^{-1}=D_{\widehat{\mathrm{U}}(V) / \widehat{P}(X)}(S_V)^{-1}.
        \end{equation*}
        It follows that the RHS of \eqref{eq:W_unfolding} is
        \begin{equation*}
            \Delta_{T_W}^{',-1} \Gamma(\chi,\overline{\mu} \eta)\frac{\Delta_{T_W}}{\Delta_{\mathrm{U}(W)}}\sum_{w \in W_{M(X)} \backslash W_V} \frac{\det(1-q_F^{-\frac{1}{2}} (w S_V)^{(r)\star} \stackrel{\mathrm{I}}{\otimes}_{\overline{\mu}} S_W)}{D_{\widehat{\mathrm{U}}(V)/\widehat{P}(X)}(wS_V)}ch_{\lambda_r}\left((wS_V)^{(r)}\right)\delta_{B_V}(\lambda_r)^\frac{1}{2}.
        \end{equation*}
        By Remark~\ref{rem:iso_explicit} and Proposition~\ref{prop:phi_0_calc} specialized to $\lambda_V=\lambda_r \in \Lambda_V^+$ and $\lambda_W=1$, we know that
        \begin{equation*}
            \mathcal{W}_{\chi,\overline{\mu}\eta}^I(\lambda_r)=\Delta_{T_W}^{',-1} \Gamma(\chi,\overline{\mu} \eta) \sum_{w \in W_G} \frac{\det(1-q_F^{-\frac{1}{2}} \mathcal{R}_{-,\overline{\mu}}
        (wS))}{D_{\widehat{G}/\widehat{B}}(wS)} \chi_{\lambda_r}\left((wS_V)^{(r)}\right)\delta_{B_V}(\lambda_r)^\frac{1}{2}.
        \end{equation*}
        The natural projection $W_G \to W_V$ induces a bijection $W_M \backslash W_G=W_{M(X)} \backslash W_V$, where we recall that $M=M(X) \times \mathrm{U}(W)$. Therefore we have to show that
        \begin{equation}
        \label{eq:alternative_weyl_sum}
            \sum_{w \in W_M} \frac{\det(1-q_F^{-\frac{1}{2}} \mathcal{R}_{-,\overline{\mu}}
        (wS))}{D_{\widehat{G}/\widehat{B}}(wS)} \chi_{\lambda_r}\left((wS_V)^{(r)}\right)= \frac{\Delta_{T_W}}{\Delta_{\mathrm{U}(W)}}\frac{\det(1-q_F^{-\frac{1}{2}} S_V^{(r)\star} \stackrel{\mathrm{I}}{\otimes}_{\overline{\mu}}S_W)}{D_{\widehat{\mathrm{U}}(V)/\widehat{P}(X)}(S_V)}ch_{\lambda_r}(S_V^{(r)}).
        \end{equation}
        We have the equalities $W_M=W_{\widetilde{G}} \times W_{G_r}$ and
        \begin{equation*}
        D_{\widehat{G}/\widehat{B}}(S)=D_{\widehat{\mathrm{U}}(V) / \widehat{P}(X)}(S_V) D_{\widehat{\widetilde{G}} / \widehat{\widetilde{B}}}(\widetilde{S})D_{\widehat{G}_r / \widehat{B}_r}(S_V^{(r)}), \quad
        \mathcal{R}_{-,\overline{\mu}}
        (S)=S_V^{(r)\star} \stackrel{\mathrm{I}}{\otimes}_{\overline{\mu}} S_W  \oplus \widetilde{\mathcal{R}}_{-,\overline{\mu}}(\widetilde{S}).
    \end{equation*}
    It follows that the LHS of \eqref{eq:alternative_weyl_sum} is
    \begin{equation*}
        \frac{\det(1-q_F^{-\frac{1}{2}} S_V^{(r)\star} \stackrel{\mathrm{I}}{\otimes}_{\overline{\mu}} S_W)}{D_{\widehat{\mathrm{U}}(V)/\widehat{P}(X)}(S_V)} \sum_{\widetilde{w} \in W_{\widetilde{G}}} \frac{\det(1-q_F^{-\frac{1}{2}} \widetilde{\mathcal{R}}_{-,\overline{\mu}}
        (\widetilde{w} \widetilde{S}))}{D_{\widehat{\widetilde{G}}/\widehat{\widetilde{B}}}(\widetilde{w} \widetilde{S})}  \sum_{w_r \in W_{G_r}} \frac{\chi_{\lambda_r}(w_r S_V^{(r)})}{D_{\widehat{G}_r/\widehat{B}_r}(w_r S_V^{(r)})}.
    \end{equation*}
    By the $r=0$ case and the Weyl character formula \cite[Proposition~A.0.2.1]{BPC}, we conclude that 
    \begin{equation*}
           \sum_{\widetilde{w} \in W_{\widetilde{G}}} \frac{\det(1-q_F^{-\frac{1}{2}} \widetilde{\mathcal{R}}_{-,\overline{\mu}}
        (\widetilde{w} \widetilde{S}))}{D_{\widehat{\widetilde{G}}/\widehat{\widetilde{B}}}(\widetilde{w} \widetilde{S})}=\frac{\Delta_{T_W}}{ \Delta_{\mathrm{U}(W)}}, \quad \text{ and } \sum_{w_r \in W_{G_r}} \frac{\chi_{\lambda_r}(w_r S_V^{(r)})}{D_{\widehat{G}_r/\widehat{B}_r}(w_r S_V^{(r)})}=ch_{\lambda_r}(S_V^{(r)}).
    \end{equation*}        
    \end{proof}

    For $k \in \nn$, denote by $\mathrm{As}^{(-1)^k}$ the Asai representation $ \mathrm{As}^{(-1)^k} : {}^L G_r \to \GL(\cc^r \otimes \cc^r)$ given by $\mathrm{As}^{(-1)^k}(g)=g^{(1)} \otimes g^{(2)}$ for $g = (g^{(1)} ,g^{(2)} ) \in \widehat{G}_r$, and which factorizes through $W_F / W_E$ with $\mathrm{As}^{(-1)^k}(c)=(-1)^k s$ in the inert case, where $s \in \GL(\cc^r \otimes \cc^r)$ is characterized by $s(u \otimes v)=v \otimes u$.
    
    \begin{lem}
    \label{lem:L_function_unfolding}
        Let $S_V \in \widehat{T}_V.\mathrm{Fr}$, $S_W \in \widehat{T}_W.\mathrm{Fr}$ and $S_r \in \widehat{T}_r.\mathrm{Fr}$ be in general position. We have
        \begin{align*}
            &\det(1-q_F^{-1} S_W\stackrel{\mathrm{I}}{\otimes} S_r^{c})\det(1-q_F^{-1} \mathrm{As}^{(-1)^m}(S_r))\\
            =& \sum_{w \in W_V} \frac{\det(1-q_F^{-\frac{1}{2}} (wS_V)^{(r) \star} \stackrel{\mathrm{I}}{\otimes}_{\overline{\mu}} S_W)}{D_{\widehat{\mathrm{U}}(V)/\widehat{B}_V}(w S_V)}\det(1-q_F^{-\frac{1}{2}} (wS_V)^{(r) \star} \stackrel{\mathrm{I}}{\otimes}_{\overline{\mu}} S_r) \det(1-q_F^{-\frac{1}{2}} (wS_V)^{(m)} \stackrel{\mathrm{I}}{\otimes}_{\overline{\mu}} S_r),
        \end{align*}
        where $S_r^{c}$ is the conjugate of $S_r$ by $c$.
    \end{lem}

    \begin{proof}
        Because $S_l \stackrel{\mathrm{I}}{\otimes}_{\overline{\mu}} S_k=S_l \stackrel{\mathrm{I}}{\otimes} \overline{\mu} S_k$, this is the same proof as~\cite[Equation~(8.7.2.8)]{BPC}, noting that $\det(1-q_F^{-1} \overline{\mu}S_W\stackrel{\mathrm{I}}{\otimes} (\overline{\mu}S_r)^{c} )=\det(1-q_F^{-1} S_W\stackrel{\mathrm{I}}{\otimes} S_r^{c})$ and $\det(1-q_F^{-1} \mathrm{As}^{(-1)^{m+1}}(\overline{\mu}S_r))=\det(1-q_F^{-1} \mathrm{As}^{(-1)^m}(S_r))$.
    \end{proof}
    
    We now end the proof of Proposition~\ref{prop:normalization} in the case $r\geq 1$. We have to show that $\mathcal{W}_{\chi,\overline{\mu}\eta}^I(1)=\Delta_{T_W}^{',-1} \Gamma(\chi,\overline{\mu}\eta) \Delta_{T_W} \Delta_{\mathrm{U}(W)}^{-1}$, which by Lemma~\ref{lem:unfolding_W} amounts to 
    \begin{equation*}
        \sum_{w \in W_V} \frac{\det(1-q_F^{-\frac{1}{2}} (w S_V)^{(r)\star} \stackrel{\mathrm{I}}{\otimes}_{\overline{\mu}} S_W)}{D_{\widehat{\mathrm{U}}(V)/\widehat{B}_V}(wS_V)}=1.
    \end{equation*}
    This follows from Lemma~\ref{lem:L_function_unfolding} by taking the limit $S_r \to 0$.

\subsubsection{}

As the proof of Theorem~\ref{thm:II} is now complete, we rewrite its formula using Satake parameters. In the next section we will use it for $\lambda_W=1$.

\begin{theorem}
\label{thm:Sat_reform}
    Let $\varphi^\circ_{V} \in \sigma_V^{K_V}$ and $\varphi^\circ_W \in \sigma_W^{K_W}$ and $\mathcal{L}_H \in \Hom_H(\sigma_V \otimes \sigma_W \otimes \overline{\nu}_{\mu,\psi},\cc)$. For every $\lambda=(\lambda_V,\lambda_W) \in \Lambda^{-}$ we have
    \begin{align*}
        \mathcal{L}_H(\sigma_V(\lambda_V) \varphi^\circ_V \otimes \sigma_W(\lambda_W)\varphi^\circ_W \otimes \phi^\circ)&= \frac{\Delta_{\mathrm{U}(W)}}{\Delta_{T_W}} \mathcal{L}_H(\varphi^\circ_V \otimes \varphi^\circ_W \otimes \phi^\circ) \\
        &\times \sum_{w \in W_G} \frac{\det(1-q_F^{-\frac{1}{2}} \mathcal{R}_{-,\overline{\mu}}
        (wS))}{D_{\widehat{G}/\widehat{B}}(wS)} \left( \left((w_V,w_0 w_W) \chi \boxtimes \eta\right) \delta_{B^+}^{-\frac{1}{2}} \right)(\lambda).
    \end{align*}
\end{theorem}

\subsubsection{Proof of Proposition~\ref{prop:W_formula2}}
\label{subsubsec:end_proof_formula_2}
We end this section by completing the proof of the alternative formula of Proposition~\ref{prop:W_formula2} in the split case. It involved the Lagrangian subspace $\mathcal{Y}$ defined in \S\ref{subsubsec:symplectic}. Proposition~\ref{prop:W_formula2} is now a consequence of Proposition~\ref{prop:times_version}, \eqref{eq:Satake_reform} and the equalities $\mathbf{b}^{\times}(\chi,w_0\overline{\mu}\eta)=\det(1-q_F^{-\frac{1}{2}} \mathcal{Y}_{\overline{\mu}}(S))$ and $\mathbf{d}_V(\chi) \mathbf{d}_W(w_0 \eta)=D_{\widehat{G}/\widehat{B}^+}(S)^{-1}$, where $\mathbf{b}^{\times}$ was defined in \eqref{eq:b_times_defi}.

\subsection{Proof of Theorem~\ref{thm:II}}
\label{subsec:proof_II}
\subsubsection{}

In this section, let $\sigma_V$, $\sigma_W$ and $\tau$ be irreducible unramified representations of $\mathrm{U}(V)$, $\mathrm{U}(W)$ and $G_r$ respectively. Their Satake parameters are semisimple conjugacy classes in ${}^L \mathrm{U}(V)$, ${}^L \mathrm{U}(W)$ and ${}^L G_r$ which we identify with representatives $S_V \in \widehat{T}_V .\mathrm{Fr}$, $S_W \in \widehat{T}_W . \mathrm{Fr}$ and $S_r \in \widehat{T}_r . \mathrm{Fr}$ respectively. Set $S:=(S_V,S_W) \in {}^L T$.  Let $\varphi^\circ_{V} \in \sigma_V^{K_V}$, $\varphi^\circ_W \in \sigma_W^{K_W}$ and $\varphi_\tau^\circ \in \tau^{K_r}$ be spherical vectors. Let $\mathcal{L}_H \in \Hom_H(\sigma_V \otimes \sigma_W \otimes \overline{\nu}_{\mu,\psi},\cc)$ and $\mathcal{L}_W \in \Hom_{N_r}(\tau,\overline{\psi_r})$. Set $\Phi^\circ=1_{X^*(\oo_E)} \otimes \phi^\circ \in \overline{\omega}_V$. Write $\tau^c=\tau \circ c$. Recall that $P$ is the parabolic subgroup of $\mathrm{U}(V)$ stabilizing the flag $0 \subset E v_1 \subset \hdots \subset X$.

\begin{prop}
\label{prop:L_unfold}
Let $s \in \cc$. For $\Re(s)$ large enough we have the absolutely convergent sum
\begin{align}
\label{eq:L_unfold}
       &\sum_{\lambda_r \in \Lambda_r} \mathcal{L}_W\left(\tau(\lambda_r) \varphi^\circ_\tau\right) \mathcal{L}_H\left(\sigma_V(\lambda_r)\varphi^\circ_V \otimes \varphi^\circ_W \otimes (\overline{\omega}_V(\lambda_r)\Phi^\circ)_{Y_+^*} \right)(\Val{.}^{s} \delta_{P(X)}^{\frac{1}{2}} \delta_{P}^{-1})(\lambda_r) \nonumber \\
       =  &\frac{L(\frac{1}{2}+s,\sigma_V \times \tau  \otimes \overline{\mu})}{L(1+s, \sigma_W \times \tau^c )L(1+2s,\tau,\mathrm{As}^{(-1)^{m}})}\mathcal{L}_W(\varphi^\circ_\tau)\mathcal{L}_H(\varphi^\circ_V \otimes \varphi^\circ_W \otimes \phi^\circ).
\end{align}
Moreover, if $\sigma_V$, $\sigma_W$ and $\tau$ are tempered, \eqref{eq:L_unfold} holds if $\Re(s)>-\frac{1}{2}$.
\end{prop}
\begin{proof}
    Let $\lambda_r \in \Lambda_r^+$. By \eqref{eq:omega_action}, $(\overline{\omega}_V(\lambda_r)\Phi^\circ)_{Y_+^*}$ is $\overline{\mu}(\lambda_r)\Val{\lambda_r}^{\frac{1}{2}} \phi^\circ$ if $\lambda_r \in \Lambda_r^{++}$ and $0$ otherwise. By Lemma~\ref{lem:unfolding_W} and the normalization of Proposition~\ref{prop:normalization}, we therefore have
    \begin{align*}
        &\mathcal{L}_H\left(\sigma_V(\lambda_r)\varphi^\circ_V \otimes \varphi^\circ_W \otimes (\overline{\omega}_V(\lambda_r)\Phi^\circ)_{Y_+^*} \right)= \mathcal{L}_H(\varphi^\circ_V \otimes \varphi^\circ_W \otimes \phi^\circ)  \times \\
        &\left\{ \begin{array}{ll}
           \sum_{w \in W_V} \frac{\det(1-q_F^{-\frac{1}{2}} (w S_V)^{(r)\star} \stackrel{\mathrm{I}}{\otimes}_{\overline{\mu}}S_W)}{D_{\widehat{\mathrm{U}}(V)/\widehat{B}_V}(wS_V)}ch_{\lambda_r}\left((w \overline{\mu} S_V)^{(r)}\right)\delta_{B_V}(\lambda_r)^\frac{1}{2}|\lambda_r|^{\frac{1}{2}} & \text{if } \lambda_r \in \Lambda_r^{++}, \\
         0    &  \text{otherwise.}
        \end{array} \right.
    \end{align*}
    By Shintani and Casselman--Shalika formulae (\cite{Shi} and \cite{CS}), we know that for $\lambda_r \in \Lambda_r$
    \begin{equation*}
        \mathcal{L}_W(\tau(\lambda_r) \varphi_\tau^\circ) = \left\{ \begin{array}{ll}
           \mathcal{L}_W(\varphi_\tau^\circ)\delta_{B_r}(\lambda_r)^{\frac{1}{2}}ch_{\lambda_r}(S_\tau)  & \text{if } \lambda_r \in \Lambda_r^{+}, \\
           0  & \text{otherwise.}
        \end{array} \right.
    \end{equation*}
    By Cauchy identity \cite[Theorem~43.3]{Bump}, for $\Re(s)$ large enough and $w \in W_V$ we have
    \begin{equation*}
        \sum_{\lambda_r \in \Lambda_r^{++}} |\det \lambda_r |^{\frac{1}{2}+s} ch_{\lambda_r}((w \overline{\mu}S_V)^{(r)})ch_{\lambda_r}(S_r)=\det(1-q_F^{-\frac{1}{2}}(w S_V)^{(r)}\stackrel{\mathrm{I}}{\otimes}_{\overline{\mu}}S_{r,s}  ),
    \end{equation*}
    where we have set $S_{r,s}:=q_F^{-s} S_r$. Moreover, if $\sigma_V$ and $\tau$ are tempered, this series is absolutely convergent for $\Re(s)>-\frac{1}{2}$ as their Satake parameters have complex norm $1$. Therefore, it follows from Lemma~\ref{lem:unfolding_W} and the equality $(\delta_{B_V}^{\frac{1}{2}}\delta_{B_r}^{\frac{1}{2}}\delta_{P(X)}^{\frac{1}{2}}\delta_P^{-1})(\lambda_r)=1$ that the LHS of \eqref{eq:L_unfold} is
    \begin{equation*}
        \mathcal{L}_W(\varphi^\circ_\tau)\mathcal{L}_H(\varphi^\circ_V \otimes \varphi^\circ_W \otimes \phi^\circ)\sum_{w \in W_V} \frac{\det(1-q_F^{-\frac{1}{2}} (w S_V)^{(r)\star} \stackrel{\mathrm{I}}{\otimes}_{\overline{\mu}}S_W)}{D_{\widehat{\mathrm{U}}(V)/\widehat{B}_V}(wS_V)}\det(1-q_F^{-\frac{1}{2}} (w S_V)^{(r)}\stackrel{\mathrm{I}}{\otimes}_{\overline{\mu}}S_{r,s} )^{-1}.
    \end{equation*}
    Proposition~\ref{prop:L_unfold} now follows from Lemma~\ref{lem:L_function_unfolding}, and the equality of $L$-functions for any $w \in W_V$
    \begin{align*}
        &L(\frac{1}{2}+s,\sigma_V \times \tau \otimes \overline{\mu})^{-1}=\det(1-q_F^{-\frac{1}{2}}S_V\stackrel{\mathrm{I}}{\otimes}_{\overline{\mu}}S_{r,s} ) \\
        &=\det(1-q_F^{-\frac{1}{2}}(w S_V)^{(r)}\stackrel{\mathrm{I}}{\otimes}_{\overline{\mu}}S_{r,s} )\det(1-q_F^{-\frac{1}{2}} (w S_V)^{(m)}\stackrel{\mathrm{I}}{\otimes}_{\overline{\mu}}S_{r,s} )\det(1-q_F^{-\frac{1}{2}} (w S_V)^{(r) \star}\stackrel{\mathrm{I}}{\otimes}_{\overline{\mu}}S_{r,s} ).
    \end{align*}
    
\end{proof}

\subsubsection{}
We now prove Theorem~\ref{thm:II}. Assume that $\sigma_V$, $\sigma_W$ and $\tau$ are tempered. Set $\Sigma:=I_{P(X)}^{\mathrm{U}(V)} (\tau \boxtimes \sigma_W)$ a tempered representation of $\mathrm{U}(V)$. Without loss of generality we assume that $( \varphi^\circ_V, \varphi^\circ_V )=( \varphi^\circ_W,\varphi^\circ_W )=( \varphi^\circ_\tau,\varphi^\circ_\tau )=1$ for the chosen inner products. Let $\varphi^\circ_\Sigma \in \Sigma$ be the spherical vector with $\varphi^\circ_\Sigma(1)=\varphi^\circ_\tau \otimes \varphi^\circ_W$. Then $( \varphi^\circ_\Sigma, \varphi^\circ_\Sigma )=1$. By Proposition~\ref{prop:tempered_computations} and Proposition~\ref{prop:L_unfold} and the Iwasawa decomposition $\mathrm{U}(V)=T_r H K_V$ we have
\begin{equation*}
    \mathcal{P}_{\mathrm{U}(V)}(\varphi^\circ_V \otimes \varphi^\circ_\Sigma \otimes \Phi^\circ)=\Val{\frac{L(\frac{1}{2},\sigma_V \times \tau  \otimes \overline{\mu})}{L(1, \sigma_W \times \tau^c )L(1,\tau,\mathrm{As}^{(-1)^{m}})}}^2 \frac{\mathcal{P}_{N_r}(\varphi^\circ_\tau) \mathcal{P}_{H}(\varphi^\circ_V \otimes \varphi^\circ_W \otimes \phi^{\circ})}{\Delta_{G_r}}.
\end{equation*}
By~\cite{CS} and~\cite[Proposition~1.1.1]{Xue} we know that
\begin{equation*}
    \mathcal{P}_{N_r}(\varphi^\circ_\tau)=\frac{\Delta_{G_r}}{L(1,\tau,\mathrm{Ad})} \text{ and } \mathcal{P}_{\mathrm{U}(V)}(\varphi^\circ_V \otimes \varphi^\circ_\Sigma \otimes \Phi^\circ)=\Delta_{\mathrm{U}(V)}\frac{L(\frac{1}{2},\sigma_V \times \Sigma \otimes \overline{\mu})}{L(1,\sigma_V,\mathrm{Ad})L(1,\Sigma,\mathrm{Ad})}.
\end{equation*}
Theorem~\ref{thm:II} now follows from the following $L$-function equality
\begin{equation*}
    \frac{L(\frac{1}{2}, \sigma_V \times \Sigma \otimes \overline{\mu})}{L(1,\Sigma,\mathrm{Ad})}=\frac{L(\frac{1}{2},\sigma_V \times \sigma_W \otimes \overline{\mu})}{L(1,\sigma_W,\mathrm{Ad})L(1,\tau,\mathrm{Ad})}\Val{\frac{L(\frac{1}{2},\sigma_V \times \tau  \otimes \overline{\mu})}{L(1, \sigma_W \times \tau^c )L(1,\tau,\mathrm{As}^{(-1)^{m}})}}^2.
\end{equation*}

\printbibliography

\begin{flushleft}
Paul Boisseau \\
Aix Marseille Univ, \\
CNRS, I2M, \\
Marseille, France
\medskip
	
email:\\
paul.boisseau@univ-amu.fr \\
\end{flushleft}

\end{document}